\documentclass[11pt,letter]{amsart}
\usepackage[
top=2.7cm, bottom=2.5cm, inner=2.5cm, outer=2.4cm,
marginparwidth=2.3cm 
]{geometry}

\usepackage{amssymb,latexsym, amsmath, amsxtra, bbm}
\usepackage{verbatim}
\usepackage[abs]{overpic}
\usepackage{hyperref}

\usepackage{tikz}
\usepackage{mathtools}

\allowdisplaybreaks[4]

\theoremstyle{plain}
\newtheorem{theorem}{Theorem}[section]
\newtheorem*{theorem*}{Theorem}
\newtheorem*{conj*}{Conjecture}
\newtheorem{lemma}[theorem]{Lemma}
\newtheorem{prop}[theorem]{Proposition}
\newtheorem{cor}[theorem]{Corollary}

\newtheorem{thmx}{Theorem}

\theoremstyle{definition}
\newtheorem{definition}[theorem]{Definition}
\newtheorem{rem}[theorem]{Remark}

\newtheorem*{assA}{Assumption A}
\newtheorem*{assB}{Assumption B}
\newtheorem*{assC}{Assumption C}

\theoremstyle{remark}
\newtheorem*{remark}{Remark}

\numberwithin{equation}{section}
\numberwithin{theorem}{section}
\numberwithin{table}{section}
\numberwithin{figure}{section}

\providecommand{\defn}[1]{\emph{#1}}

\renewcommand{\leq}{\leqslant}

\renewcommand{\geq}{\geqslant}

\newcommand{\diam}  {\operatorname{diam}}

\newcommand{\card} {\operatorname{card}}
\newcommand{\supp}{\operatorname{supp}}
\newcommand{\scolon}{:}

\renewcommand{\:}{\colon}

\newcommand{\mesh}{\operatorname{mesh}}

\newcommand{\PPP}{\mathcal{P}}

\newcommand{\MMM}{\mathcal{M}}

\newcommand{\C}{\mathbb{C}}

\newcommand{\N}{\mathbb{N}}

\newcommand{\Q}{\mathbb{Q}}
\newcommand{\R}{\mathbb{R}}

\newcommand{\Z}{\mathbb{Z}}

\newcommand{\cC}{\mathcal{C}}

\newcommand{\cE}{\mathcal{E}}

\newcommand{\cH}{\mathcal{H}}
\newcommand{\cI}{\mathcal{I}}
\newcommand{\cJ}{\mathcal{J}}

\newcommand{\cL}{\mathcal{L}}
\newcommand{\cM}{\mathcal{M}}

\newcommand{\cP}{\mathcal{P}}

\newcommand{\cS}{\mathcal{S}}

\newcommand{\hK}{\widehat{K}}

\newcommand{\hT}{\widehat{T}}

\newcommand{\hmu}{\widehat{\mu}}

\newcommand{\hvarphi}{\widehat{\varphi}}

\newcommand{\oB}{\overline{B}}

\newcommand{\ou}{\overline{u}}

\providecommand{\abs}[1]{\lvert#1\rvert}
\newcommand{\Abs}[1]{\left\lvert#1\right\rvert}
\providecommand{\Absbig}[1]{\bigl\lvert#1\bigr\rvert}
\providecommand{\Absbigg}[1]{\biggl\lvert#1\biggr\rvert}

\providecommand{\norm}[1]{\|#1\|}
\providecommand{\Norm}[1]{\left\|#1\right\|}
\providecommand{\Normbig}[1]{\bigl\|#1\bigr\|}
\providecommand{\Normbigg}[1]{\biggl\|#1\biggr\|}

\renewcommand{\=}{\coloneqq}

\renewcommand{\top}{\operatorname{top}}

\newcommand{\hypLip}{\operatorname{-Lip}}

\begin{document}
	
	\title{On computability of equilibrium states}
	\author{Ilia Binder \and Qiandu He \and Zhiqiang Li \and Yiwei Zhang}
	\address{Ilia Binder, Department of Mathematics, University of Toronto, Bahen Centre, 40 St. George St., Toronto, Ontario, M5S 2E4, CANADA}
	\email{ilia@math.toronto.edu}
	\address{Qiandu~He, School of Mathematical Sciences, Peking University, Beijing 100871, CHINA}
	\email{heqiandu@stu.pku.edu.cn}
	\address{Zhiqiang~Li, School of Mathematical Sciences \& Beijing International Center for Mathematical Research, Peking University, Beijing 100871, CHINA}
	\email{zli@math.pku.edu.cn}
	\address{Yiwei~Zhang, School of Mathematical Sciences and Big Data, Anhui University of Science and Technology, Huainan, Anhui 232001, CHINA}
	\email{yiweizhang831129@gmail.com}
	
	

	\begin{abstract}
		Equilibrium states are natural dynamical analogues of Gibbs states in thermodynamic formalism. This paper investigates their computability within the framework of Computable Analysis. We show that the unique equilibrium state for a computable, open, topologically exact, distance-expanding map $T\:X\rightarrow X$ and a computable H\"older continuous potential $\varphi\:X\rightarrow\mathbb{R}$ is always computable. As an application, we establish the computability of equilibrium states for computable hyperbolic rational maps and their respective geometric potentials. Moreover, we develop a constructive method to exhibit the non-uniqueness of equilibrium states for some dynamical systems. We also present some computable dynamical systems whose equilibrium states are all non-computable.
	\end{abstract}
	
	\maketitle
	
	\tableofcontents
	
	\section{Introduction}\label{sec:intro}
	
	The decidability problem was the third question of Hilbert's program, which was proposed by Hilbert in Paris in 1900. In the 1930s, \defn{Halting Problem} given by Turing provided a counterexample to this question. Another significant contribution of Turing was to construct a definite method of computations called \defn{Turing machine}. Turing machines give a standard model of computation in discrete settings and have become the foundation of modern digital computers.
	
	However, most digital computers are used for computations in continuous settings nowadays. The research on computations with real numbers began with Turing's work on the original definition of computable real numbers in 1937 (\cite{Tu37}). The work of Banach and Mazur in 1937 (\cite{BM37}) provided definitions of computability for real objects (such as subsets of $\R^n$ and functions $f\:\R^n\rightarrow\R^m$). The ``bit model'' has been developed, serving as a foundation of the tradition of Computable Analysis (\cite{Gr55, La55, Ko91, Weih00}). Relevant to our paper, in \cite{HR09}, Hoyrup and Rojas established a natural computable structure on the set of probability measures, gave a definition of a computable probability measure, and investigated a computable metric space with a computable probability measure.
	
	\medskip
	\noindent\textbf{Computable Analysis on dynamical systems.}
	\smallskip
	
	In recent decades, there has been dramatic growth in research on the computability and computational complexity of many objects generated by dynamical systems (see e.g.~\cite{Ya21} and references therein). The initial value sensitivity and typical instability of many interesting systems imply that what is observed on the computer screen could be completely unrelated to what was meant to be simulated.
	
	This paper is devoted to the study of computability questions for a family of invariant measures called equilibrium states (see Definition~\ref{def:equilibriumstate}). It expands and continues the line of inquiry initiated by \cite{BBRY11} of the first-named author and Braverman, Rojas, and Yampolsky. Recall that for a \defn{rational map} $f(z)=P(z)/{Q(z)}$ on the Riemann sphere, where $P(z)$ and $Q(z)$ are mutually prime polynomials, with degree $\deg(f)=\max\{\deg(P),\,\deg(Q)\}\geq 2$, the \defn{Julia set} $\cJ_f$ is defined as the locus of the chaotic dynamics of $f$, namely, the complement of the set where the dynamics of $f$ is stable. Moreover, the weak$^*$ limit of averages of Dirac measures supported on the preimage points of a repelling periodic point $w$ is an $f$-invariant measure $\lambda$ (independent of the choice of $w$), called \defn{Brolin--Lyubich measure}.
	
	In \cite{BBRY11}, a uniform machine was designed to compute the Brolin--Lyubich measure for any given rational map. Together with the existence of polynomials whose coefficients are all computable but whose Julia sets are non-computable from the groundbreaking works of Braverman and Yampolsky (\cite{BY06, BY09}) dating back to a question of Milnor \cite[Section~1]{BY06}, the computability of Brolin--Lyubich measures leads to a conflict: heuristically speaking, a measure contains more information than its support, but in Computable Analysis, there exists a computable invariant probability measure whose support is, however, non-computable. Such a conflict can be reconciled by considering these two results as the computability properties of the same physical objects from the geometric and statistical perspectives, respectively. In this era of artificial intelligence and data science, the latter undoubtedly deserves closer investigation. 
	
	Regarding the algorithmic aspects of general invariant measures, in a fundamental paper \cite{GHR11}, Galatolo, Hoyrup, and Rojas established that for each computable map $T\:X\rightarrow X$ on a recursively compact set $X$, all isolated invariant measures are computable. As a counterexample, a computable dynamical system with no computable invariant measures is constructed in \cite[Subsection~4.1]{GHR11} as well. In Sections~\ref{sec:computability of equilibriumstates}~and~\ref{sec:counterExample}, we extend these results.
	
	Thermodynamic formalism is a powerful method for investigating invariant measures with prescribed local behavior under iterations of the dynamical system. This theory, inspired by statistical mechanics, was created by Ruelle, Sinai, and others in the seventies (\cite{Dob68, Si72, Bo75, Wa82}). Since then, thermodynamic formalism has been applied in many classical contexts (see e.g.~\cite{Bo75, Ru89, Pr90, KH95, Zi96, MauU03, BS03, Ol03, Yu03, PU10, MayU10}) and has remained at the forefront of research in dynamical systems.
	
	Among many other applications, thermodynamic formalism has played a central role in the study of the statistical properties of dynamical systems since its very early days.
	
	The key objects of investigation in thermodynamic formalism are invariant measures called equilibrium states. More precisely, for a continuous map $T\:X\rightarrow X$ of a compact metric space $(X,\rho)$, a $T$-invariant Borel probability measure $\mu$ on $X$, and a real-valued continuous function $\varphi\:X\rightarrow\R$, called \defn{potential}, one can define the \defn{measure-theoretic pressure} and the \defn{topological pressure} to describe the degree of complexity of the given dynamical system from different perspectives. The Variational Principle implies that the supremum over all $T$-invariant Borel probability measures on $X$ of the measure-theoretic pressure is equal to the topological pressure. We say that a Borel probability measure $\mu$ is an \defn{equilibrium state} for $T$ and $\varphi$ if $\mu$ maximizes the measure-theoretic pressure. In particular, if $\varphi$ is a constant function, then an equilibrium state reduces to a \defn{measure of maximal entropy} (see Subsection~\ref{subsec:basicconcepts}). In particular, the Brolin--Lyubich measure of a rational map $f\:\widehat{\C}\rightarrow\widehat{\C}$ is the measure of maximal entropy of $f$.
	
	In recent years, significant progress has been made in understanding the computability of dynamical invariants, such as topological entropy and pressure (see e.g.~\cite{Sp07, Sp08, BDWY22}). In the proof of Theorem~\ref{t:es_computable}, we use a novel approach to establish the computability of the topological pressure and the computability of the Jacobian function in a different setting.
	
	\medskip
	\noindent\textbf{Statement of main results.}
	\smallskip
	
	Our main contributions are threefold: First, we establish computability results for equilibrium states in uniformly expanding systems (see Theorems~\ref{t:es_computable}~and~\ref{t:es_computable_subspace}) and apply them to demonstrate the computability of equilibrium states for some complex dynamical systems (see Theorem~\ref{t:hyperbolic_rational}). Second, we introduce a novel mechanism to demonstrate non-uniqueness of equilibrium states (see Theorems~\ref{t:nonunique}~and~\ref{t:nonunique_spec}). Third, we construct computable systems with non-computable equilibrium states (see Theorems~\ref{t:maincounterexample}~and~\ref{t:maincounterexample2}).
	
	\medskip
	\noindent\textbf{Computability of equilibrium states.}
	\smallskip

	Our first result concerns the computability of equilibrium states. Here is a list of assumptions used in this context.
	\begin{assA}
		\quad
		\begin{enumerate}
			\smallskip
			\item[(i)] $(X,\,\rho,\,\mathcal{S})$ is a computable metric space and $X$ is a recursively compact set (in the sense of Definition~\ref{definition of recursively compact}).
			\smallskip
			\item[(ii)] $\varphi\:X\rightarrow\R$ is a H\"older continuous function with an exponent $v_0\in(0,1]$.
			\smallskip
			\item[(iii)] $T\:X\rightarrow X$ is an open, topologically exact, continuous, and distance-expanding map with respect to the metric $\rho$ with constants $\eta>0$, $\lambda>1$, and $\xi>0$ (in the sense of Definition~\ref{defndistanceexpanding}).
			\smallskip
			\item[(iv)] $\varphi$ and $T$ are both computable (in the sense of Definition~\ref{Algorithm about computable functions}).
		\end{enumerate}
	\end{assA}
	
	By classical results (see e.g.~Proposition~\ref{Jacobian} and \cite[Corollary~5.2.14]{PU10}), the Assumption~A~(i) through~(iii) imply that there exists a unique equilibrium state for $T$ and $\varphi$. The Assumption~A~(iv) allows us to input the information of the system in the computer.
	
	\begin{thmx}    \label{t:es_computable}
		Let the quintet $(X,\,\rho,\,\mathcal{S},\,\varphi,\,T)$ satisfy the Assumption~A in Section~\ref{sec:intro}, and $\mu$ be the unique equilibrium state for the map $T$ and the potential $\varphi$. Then $\mu$ is computable.
	\end{thmx}
	
	For a more detailed version of Theorem~\ref{t:es_computable}, see Theorem~\ref{maintheorem} below. The Assumption~A requires that the behavior of the dynamical system has nice properties in the whole space. However, the dynamical behavior may be good enough in an invariant subspace rather than in the whole space. Hence, we prove Theorem~\ref{t:es_computable_subspace}, a ``subspace'' version of Theorem~\ref{t:es_computable}. We give a list of assumptions that describe the systems considered in Theorem~\ref{t:es_computable_subspace} as follows.
	
	\begin{assB}
		\quad
		\begin{enumerate}
			\smallskip
			\item[(i)] $(X,\,\rho,\,\mathcal{S})$ is a computable metric space.
			\smallskip
			\item[(ii)]$(D,\,\rho^{\prime},\,\mathcal{S}^{\prime})$ is a computable metric space, where $D$ is a subset of $X$, $\rho^{\prime}$ is the restriction of the metric $\rho$ on $D$, and $\mathcal{S}'$ is a uniformly computable sequence of points in the computable metric space $(X,\,\rho,\,\mathcal{S})$. The subset $D$ is a recursively compact set in the computable metric space $(D,\,\rho^{\prime},\,\mathcal{S}')$.
			\smallskip
			\item[(iii)] $\varphi\:D\rightarrow\R$ is a H\"older continuous function with an exponent $v_0\in(0,1]$.
			\smallskip
			\item[(iv)] $T\:X\rightarrow X$ is a map such that $D$ is completely $T$-invariant. $T|_D\:D\rightarrow D$ is an open, topologically exact, continuous map that is distance-expanding with respect to the metric $\rho^{\prime}$ with constants $\eta>0$, $\lambda>1$, and $\xi>0$.
			\smallskip
			\item[(v)] $\varphi\:D\rightarrow\R$ and $T|_D\:D\rightarrow\R$ are both computable in the computable metric space $(D,\,\rho^{\prime},\,\mathcal{S}')$.
		\end{enumerate}
	\end{assB}
	
	The Assumption~B~(ii) through~(iv) give the uniqueness of the equilibrium state for the restriction $T|_D$ and the potential $\varphi$. Together with Theorem~\ref{maintheorem}, the Assumption~B~(ii) through~(v) imply that the equilibrium state is computable in the computable metric space $(D,\,\rho^{\prime},\,\mathcal{S}')$. Then the Assumption~B~(ii) and~(iii) allow us to deduce the computability in the computable metric space $(X,\,\rho,\,\mathcal{S})$ from the computability in the computable metric space $(D,\,\rho^{\prime},\,\mathcal{S}')$. Therefore, for the systems that satisfy the above assumptions, we prove the following theorem.
	
	\begin{thmx}  \label{t:es_computable_subspace}
		Let the septet $(X,\,D,\,\rho,\,\mathcal{S},\,\mathcal{S}',\,\varphi,\,T)$ satisfy the Assumption~B in Section~\ref{sec:intro}, and $\mu$ be the unique equilibrium state for the restriction $T|_D$ and the potential $\varphi$. Then $\mu$ is computable.
	\end{thmx}
	
	For a more detailed version of Theorem~\ref{t:es_computable_subspace}, see Theorem~\ref{maintheorem'} below. We apply Theorem~\ref{t:es_computable_subspace} (or Theorem~\ref{maintheorem'}) to establish Theorem~\ref{t:hyperbolic_rational}, i.e., the computability of the equilibrium states for the restriction of hyperbolic rational maps on their corresponding Julia sets and their corresponding geometric potentials. Indeed, there are many other dynamical systems whose dynamical behavior may be merely good enough in an invariant subspace. 
	
	We say that a rational map is \defn{hyperbolic} if its postcritical orbit closure is disjoint from its Julia set. 
	Hyperbolic rational maps are abundant and are expected to be dense among all rational maps. In the case of quadratic rational maps, see the renowned MLC Conjecture; namely, the Mandelbrot set is locally connected (refer to \cite[Appendix~G]{Mil06} for reference). The studies on equilibrium states as well as the theory of thermodynamic formalism for hyperbolic rational maps are by now well-developed and established. Given $t\in\R$, it is well-known that a hyperbolic rational map $f$ restricted to the Julia set $\cJ_{f}$ admits a unique non-atomic equilibrium state $\mu_{t,f}$ which is supported on $\cJ_{f}$ for the geometric potential $t\varphi_{f}(z)=t\log\bigl(f^{\#}(z)\bigr)$ (see (\ref{def:sphericalf}) for the definition of the spherical derivative $f^{\#}(z)$). 
	
	With the above conventions, we state the following theorem.
	
	\begin{thmx}  \label{t:hyperbolic_rational}
		There exists an algorithm which, on input oracles of the coefficients of a hyperbolic rational map $f$ of degree $d\geq 2$ and an oracle of $t\in\R$, computes the unique equilibrium state for $f|_{\cJ_f}$ and $t\varphi_f(z)$. 
	\end{thmx}
	
	Note that the proof of Theorem~\ref{t:hyperbolic_rational} depends on the computability of the hyperbolic Julia sets (\cite{Br04}). Moreover, since the computability of the topological pressure is a consequence of the proof of Theorem~\ref{maintheorem}, the pressure function $\mathrm{P}(t)\= P\bigl(f|_{\cJ_f},-t\varphi_f\bigr)$ is computable. Because $\mathrm{P}(t)$ is strictly decreasing with respect to $t\in\R$, as a byproduct, we can compute its unique zero, that is, by Bowen's formula (see e.g.~\cite[Theorem~9.1.6 and Corollary~9.1.7]{PU10}), the Hausdorff dimension of $\cJ_f$. For related research on the computation of Hausdorff dimensions of invariant sets, see also \cite{JP02, JP04} and the references therein.
	
    Notably, our approach can be further developed to establish the computability of equilibrium states for more general uniformly hyperbolic systems or even partially hyperbolic systems. Due to space limitations, we focus on the current setup in this paper and postpone further investigation to future works.
	
	\medskip
	\noindent\textbf{Non-uniqueness of equilibrium states.}
	\smallskip

	Determining which dynamical system admits a unique equilibrium state is a central problem in ergodic theory and a valuable tool for studying the statistical properties of a system (for example, detecting the existence of a phase transition).
	
	As indicated by Hofbauer \cite{Ho77}, this question was motivated by the studies of \defn{intrinsically ergodic} dynamical systems, i.e., those having a unique measure of maximal entropy, see e.g.~\cite{Bu97, Bu05, Bo74, CT12, Ho79, Ho81, Pa64, Weis70, Weis73}.
	
	For general equilibrium states beyond measures of maximal entropy, most results focus on proving the uniqueness, and there are limited examples of systems and potentials with multiple equilibrium states. In the context of a uniformly hyperbolic (or expanding) map, Bowen \cite{Bo75} proved that an equilibrium state exists and is unique if the potential is H\"older continuous and the map is topologically transitive. In addition, the theory for finite shifts was developed and used to achieve similar results for smooth dynamics. In contrast, Hofbauer \cite{Ho77} constructed a particular class of continuous but non-H\"older potentials for the full shift, each of which admits two equilibrium states.
	
	Beyond uniform hyperbolicity, the uniqueness of the equilibrium state was studied by many authors, including Bruin, Keller, Li, Rivera-Letelier, Iommi, Doobs, and Todd \cite{BK98, LRL14, IJT15, IT10, DT23} for interval maps; Denker and Urba\'nski \cite{DU91} for rational maps; the third-named author, Das, Przytycki, Tiozzo, Urba\'nski, and Zdunik \cite{Li18, DPTUZ21} for branched covering maps; Buzzi, Sarig, and Yuri \cite{BS03, Yu03} for countable Markov shifts and for piecewise expanding maps in one and higher dimensions. For local diffeomorphisms with some non-uniform expansion, there are results due to Varandas and Viana \cite{VV10}, Pinheiro \cite{Pi11}, Ramos and Viana \cite{RV17}. All of these results focus on establishing the existence and uniqueness of the equilibrium state for a potential with low oscillation.
	
	On the other hand, much fewer examples that admit multiple equilibrium states have been explored. They include
	\begin{enumerate}
		\smallskip
		\item[(i)] Makarov and Smirnov \cite{MS00, MS03}, and Rivera-Letelier and Przytycki \cite{PRL11} systematically studied the pressure function for rational maps $f$ with geometric potential $\phi_t= -t \log \bigl( f^\# \bigr)$, and constructed examples admitting multiple equilibrium states. For a similar study on (generalized) interval maps, see Rivera-Letelier and Przytycki \cite{PRL14}.
		\smallskip
		\item[(ii)] Examples of intermittent maps admitting multiple equilibrium states for the geometric potential were studied in \cite{VV10} via Pesin theory.
	\end{enumerate}
	
	Applying Computable Analysis, we establish the following result.
	
	\begin{thmx}   \label{t:nonunique}
		Let $(X,\,\rho,\,\mathcal{S})$ be a computable metric space, and $X$ a recursively compact set. Assume that $T\:X\rightarrow X$ is a computable map satisfying that the measure-theoretic entropy function $\mathcal{H}\:\PPP(X)\rightarrow\R$ given by
		\begin{equation*}
			\mathcal{H}(\mu)\=
			\begin{cases}
				h_{\mu}(T) & \text{if }\mu\in\MMM(X,T); \\
				0 & \text{if }\mu\notin\MMM(X,T)
			\end{cases}
		\end{equation*}
		is upper-computable on the space $\MMM(X,T)$ of $T$-invariant Borel probability measures, and $\varphi\:X\rightarrow\R$ is a computable function satisfying that $P(T,\varphi)$ is lower-computable. Then the set $\cE(T,\varphi)$ of equilibrium states for $T$ and $\varphi$ is recursively compact. Consequently, if there exists a non-computable equilibrium state for $T$ and $\varphi$, then there exist at least two equilibrium states for $T$ and $\varphi$.
	\end{thmx}
	
	As a corollary, we obtain the following result under easier-to-formulate conditions. We refer to (\ref{e:nonwandering}) for the definition of non-wandering sets.
	
	\begin{thmx}\label{t:nonunique_spec}
		Let $(X,\,\rho,\,\mathcal{S})$ be a computable metric space, and $X$ a recursively compact set. Assume that $T \: X \rightarrow X$ is a computable map with $h_{\top}(T)=0$ satisfying that there exists no computable point in the non-wandering set $\Omega(T)$. Let $\varphi\:X\rightarrow\R$ be a computable function satisfying that the pressure $P(T,\varphi)$ is computable. Then there exist at least two equilibrium states for $T$ and $\varphi$.
	\end{thmx}
	
	For example, the computable map $T\:\mathbb{S}^1\rightarrow\mathbb{S}^1$ constructed in \cite[Subsection~4.1]{GHR11} satisfies that $h_{\top}(T)=0$ and $\Omega(T)$ contains no computable point. Then for each computable function $\varphi\: X\rightarrow\R$ satisfying that $P(T,\varphi)$ is computable, there exist at least two equilibrium states for the map $T$ and the potential $\varphi$. 
	
	\medskip
	\noindent\textbf{Counterexamples.}
	\smallskip
	
	In \cite{GHR11}, a computable
	system with no computable invariant measure is constructed. On the
	basis of this, we establish the following statement.
	
	\begin{thmx}\label{t:maincounterexample}
		There exists a computable homeomorphism $T\:\mathbb{S}^1\rightarrow \mathbb{S}^1$ of the unit circle $\mathbb{S}^1$ (such as the map $T$ constructed in \cite[Subsection~4.1]{GHR11}) with the following property: for each continuous potential $\varphi\:\mathbb{S}^1\rightarrow\R$,
		\begin{enumerate}
			\smallskip
			\item[(i)] there exists at least one equilibrium state for $T$ and $\varphi$,
			\smallskip
			\item[(ii)] there is no computable $T$-invariant probability measure.
		\end{enumerate}
	\end{thmx}
	
	Recall that the topological entropy reflects a degree of complexity of a dynamical system. We, furthermore, show that there is a dynamical system on the $2$-torus with arbitrarily high topological entropy that admits equilibrium states for each continuous potential and whose equilibrium states are all non-computable. More precisely, define $\hT\: \mathbb{S}^1\times\mathbb{S}^1\rightarrow \mathbb{S}^1\times\mathbb{S}^1$ by $\hT(x,y)=(T(x),T_d(y))$ for each $x,\,y\in\mathbb{S}^1$, where $T_d(x)\=dx$ for $d\in\N$ and $x\in\mathbb{S}^1.$
	
	\begin{thmx}\label{t:maincounterexample2}
		The map $\hT \: \mathbb{S}^1\times\mathbb{S}^1 \rightarrow  \mathbb{S}^1\times\mathbb{S}^1$ constructed above has non-zero topological entropy and satisfies that for each continuous potential $\hvarphi\:\mathbb{S}^1\times\mathbb{S}^1\rightarrow\R$, the following statements are true:
		\begin{enumerate}
			\smallskip
			\item[(i)] There exists at least one equilibrium state for $\hT$ and $\hvarphi$.
			\smallskip
			\item[(ii)] There is no computable $\hT$-invariant Borel probability measure.
		\end{enumerate}
	\end{thmx}
	
	\medskip
	\noindent\textbf{Strategy of the proofs.}
	\smallskip
	
	In the following, we discuss the strategies applied to the proofs of our main results.
	
	Let us first concentrate on the proof of Theorem~\ref{t:es_computable}. Here is a detailed version of Theorem~\ref{t:es_computable}.
	
	\begin{theorem}\label{maintheorem}
		There exists an algorithm with the following property:
		
		\smallskip
		
		For each quintet $(X,\,\rho,\,\mathcal{S},\,\varphi,\,T)$ satisfying the Assumption~A in Section~\ref{sec:intro}, this algorithm outputs a Borel probability measure $\mu_n\in \PPP(X)$ supported on finitely many points of $\mathcal{S}$ such that the Wasserstein--Kantorovich distance
		\begin{equation*}
			W_{\rho}(\mu_n,\mu)\leq 2^{-n},
		\end{equation*}
		where the measure $\mu$ is the unique equilibrium state for $T$ and $\varphi$, after inputting the following data:
		\begin{enumerate}
			\smallskip
			\item[(i)] an algorithm outputting a net $G\subseteq\mathcal{S}$ of $X$ with any precision (in the sense of Definition~\ref{definition of recursively precompactness}),
			\smallskip
			\item[(ii)] two algorithms computing functions $\varphi$ and $T$ (in the sense of Definition~\ref{Algorithm about computable functions}), respectively,
			\smallskip
			\item[(iii)] two rational constants $a_0>0$ and $v_0\in(0,1]$ with $\varphi\in C^{0,v_0}(X)$ and $a_0\geq\abs{\varphi}_{v_0,\rho}$,
			\smallskip
			\item[(iv)] three rational constants $\eta>0$, $\lambda>1,$ and $\xi>0$ satisfying that $T$ is distance-expanding with constants $\eta,\,\lambda$, and $\xi$ (in the sense of Definition~\ref{defndistanceexpanding}),
			\smallskip
			\item[(v)] a constant $n\in\N$.
		\end{enumerate}
	\end{theorem}
	
	The result of Theorem~\ref{maintheorem} is uniform, in the sense that there is a single algorithm that
	takes the associated algorithms and constants as parameters and computes the corresponding equilibrium state.
	
	Our proof of Theorem~\ref{maintheorem} adapts some ideas from the proof of Theorem~\ref{t:es_computable} in \cite{BBRY11}. In addition, since we work on equilibrium states, the Jacobians of the desired measures become the focus of investigation. To establish the computability of the Jacobians, we need to estimate the convergence rate of the iterations of normalized Ruelle--Perron--Frobenius operators (see the definition in Subsection~\ref{subsec:Ruelleoperator}) with explicit expressions for related constants (see Theorem~\ref{t:eigenfunction}).  
	
	To estimate the rate, we introduce the cone technique into the study of Computable Analysis by inventing a new class of cones and designing a novel approach to using cones in thermodynamic formalism in order to control the constants for the purpose of Computable Analysis.
	
	The cone technique is an important tool for estimating the convergence rate of the iterations of normalized Ruelle--Perron--Frobenius operators. This technique, rooted in projective geometry, allows us to bound the contraction properties of transfer operators by restricting functions to some special convex subsets (cones). In the setting of smooth (or more generally, Markov) expanding maps with smooth observable, related bounds were investigated by Ferrero (\cite{Fe81}), Rychlik (\cite{Ry89}), and Hunt (\cite{Hu96}). In two groundbreaking papers \cite{Liv95a, Liv95b} in 1995, Liverani used the cone technique (based on an earlier work of Birkhoff \cite{Bi57}) and gave a new upper bound of the convergence rate of the iterations of a normalized Ruelle--Perron--Frobenius operator for general (non-Markov) non-smooth hyperbolic systems with discontinuities. A critical ingredient of this method is the existence of a convex cone of functions that is mapped strictly inside itself by the Ruelle--Perron--Frobenius operator. Moreover, the diameter of the image with respect to some projective metric gives an estimate for the upper bound of the convergence rate (see Proposition~\ref{prop:Liscontraction}).
	
	In the setting of smooth expanding maps on compact connected manifolds, Viana \cite[Chapter~2]{Vi97} provides a comprehensive estimate for the convergence rate of the iterations of a normalized Ruelle--Perron--Frobenius operator in H\"{o}lder norm. Notably, because of the connectedness, some crucial local estimates of the oscillations of the functions in the cone can be extended globally (to the whole manifold). This, however, cannot work in our setting, since we allow for general disconnected spaces. To remedy this, we add a new parameter into the definition of cones that we introduce, effectively restricting to smaller cones of functions (see Definition~\ref{def:coneavb}), to estimate the diameters (see Proposition~\ref{prop:finitediameter2}). This extra parameter takes care of the required global bounds on oscillations. Moreover, to ensure that the operator is a strict contraction on these new cones, we use a different operator and utilize additional techniques (see Proposition~\ref{prop:invariancewrtRuelleoperator}). Finally, we use the convergence rate of the new operator to obtain the desired bounds for the convergence rate of the normalized Ruelle--Perron--Frobenius operator.
	
	We need this new finer estimate for the convergence rate (Theorem~\ref{t:eigenfunction}) to establish the computability of the eigenfunction $u_{\varphi}$ of the Ruelle--Perron--Frobenius operator. Furthermore, in Subsection~\ref{subsec:computability of Jacobian}, we deduce the computability of the Jacobian of the equilibrium state whose explicit expression is given by
	\begin{equation*}
		J_{\varphi}(x) \=\frac{u_{\varphi}(T(x))}{u_{\varphi}(x)}\exp(P(T,\varphi)-\varphi(x))\quad\text{for each }x\in X.
	\end{equation*}	
    
	In Subsection~\ref{subsec:lemmalower-computable}, we establish the lower-computable openness of some subsets $U$ and $W$ of $\PPP(X)$ by expressing them as unions of uniformly computable sequences of lower-computable open sets. It is worth noting that the lower-computable openness of $U$ depends on the computability of $T$, while the lower-computable openness of $W$ depends on the computability of the Jacobian function $J_{\varphi}(x)$ in addition to the computability of $T$. In Subsection~\ref{subsec:computability of equilibriumstates}, we complete the proof of Theorem~\ref{maintheorem} and apply it to establish Theorem~\ref{t:es_computable_subspace}. Here is a detailed version of Theorem~\ref{t:es_computable_subspace}.
	
	\begin{theorem}\label{maintheorem'}
		There exists an algorithm with the following property:
		
		\smallskip
		
		For each septet $(X,\,D,\,\rho,\,\mathcal{S},\,\mathcal{S}',\,\varphi,\,T)$ satisfying the Assumption~B in Section~\ref{sec:intro}, this algorithm outputs $\mu_n\in \PPP(X)$ supported on finitely many points of $\mathcal{S}$ such that the Wasserstein--Kantorovich distance
		\begin{equation*}
			W_{\rho}(\mu_n,\mu)\leq 2^{-n},
		\end{equation*}
		where the measure $\mu$ is the unique equilibrium state for the restriction $T|_D$ and $\varphi$, after inputting the following data:
		\begin{enumerate}
			\smallskip
			\item[(i)] an algorithm outputting a net $G\subseteq\mathcal{S}'$ of $D$ with any given precision (in the sense of  Definition~\ref{definition of recursively precompactness}),
			\smallskip
			\item[(ii)] two algorithms computing functions $\varphi$ and $T|_D$ (in the sense of  Definition~\ref{Algorithm about computable functions}), respectively,
			\smallskip
			\item[(iii)] two rational constants $a_0>0$ and $v_0\in(0,1]$ with $\varphi\in C^{0,v_0}(X)$ and $a_0\geq\abs{\varphi}_{v_0,\rho}$,
			\smallskip
			\item[(iv)] three rational constants $\eta>0$, $\lambda>1,$ and $\xi>0$ satisfying that $T$ is distance-expanding with constants $\eta,\,\lambda$, and $\xi$ (in the sense of  Definition~\ref{defndistanceexpanding}),
			\smallskip
			\item[(v)] a constant $n\in\N$.
		\end{enumerate}
	\end{theorem}
	
	\smallskip
	
	Next, in Section~\ref{sec:exampleofhyperbolicrationalmap}, we provide the proof of Theorem~\ref{t:hyperbolic_rational}. By Theorem~\ref{maintheorem'}, it suffices to check that the septet $\bigl( \widehat{\C},\,\cJ_f,\,d_{\widehat{\C}},\,\Q^2,\,\bigcup_{n\in\N}f^{-n}(w),\,t\varphi_f,\,f \bigr)$ satisfies the Assumption~B in Section~\ref{sec:intro}, where $\Q^2$ is defined as $\{a+bi\scolon a,\,b\in\Q\}$, and $w$ is a repelling periodic point of the rational map $f$. The main difficulty here is to show that $\cJ_f$ is recursively compact in the computable metric space $\bigl( \cJ_f,\,d_{\widehat{\C}},\,\bigcup_{n\in\N}f^{-n}(w) \bigr)$. We use the computability of the hyperbolic Julia sets proved by Braverman \cite{Br04} to establish this statement.
	
	All of our analyses for Theorems~\ref{maintheorem} and~\ref{maintheorem'} should still work after necessary modifications if we replace the algorithms in the inputs by oracles. For the simplicity of the presentation, we leave the details to the reader.
	
	\medskip
	\noindent\textbf{Structure of the paper.}
	\smallskip
	
	In Section~\ref{sec:thermodynamicalFormalism}, we introduce our notations and review some basic notions and results in ergodic theory and thermodynamic formalism in particular. Section~\ref{sec:Computable Analysis} covers some basics of Computable Analysis. Section~\ref{sec:computability of equilibriumstates} is dedicated to the proofs of Theorems~\ref{maintheorem} and~\ref{maintheorem'}. In Subsection~\ref{subsec:conemethod}, we prove that the convergence of the iterations of a normalized Ruelle--Perron--Frobenius operator is at an exponential rate with explicit expressions for the related constants. Based on this, we demonstrate the computability of topological pressure $P(T,\varphi)$ and the function $u_{\varphi}$, which leads to the computability of the Jacobian of the equilibrium state in Subsection~\ref{subsec:computability of Jacobian}. Then, we prove Lemma~\ref{lemma:lower-computable} in Subsection~\ref{subsec:lemmalower-computable} and establish Theorems~\ref{maintheorem} and~\ref{maintheorem'} in Subsection~\ref{subsec:computability of equilibriumstates}. In Section~\ref{sec:exampleofhyperbolicrationalmap}, we introduce some results on the dynamical and algorithmic aspects of hyperbolic rational maps and prove Theorem~\ref{t:hyperbolic_rational}. In Section~\ref{s:nonunique}, we establish Theorem~\ref{t:nonunique}, from which we derive Theorem~\ref{t:nonunique_spec}. In Section~\ref{sec:counterExample}, we investigate and construct some computable systems whose equilibrium states are all non-computable, establishing Theorems~\ref{t:maincounterexample} and~\ref{t:maincounterexample2}. In Subsection~\ref{subsec:constructionofT}, we recall the computable map $T\:\mathbb{S}^1\rightarrow\mathbb{S}^1$ given in \cite[Subsection~4.1]{GHR11} and show that it satisfies the statements in Theorem~\ref{t:maincounterexample}. Finally, in Subsection~\ref{asymptotichexpansiveness}, we construct $\hT\= T\times T_d$ to establish Theorem~\ref{t:maincounterexample2}, where $T_d(x)\=dx \pmod{1}$ for $d\in\N$ and $x\in\mathbb{S}^1$.

	\subsection*{Acknowledgments}
	The authors would like to thank the anonymous referees for their valuable suggestions. I.~Binder was partially supported by an NSERC Discovery grant. Z.~Li and Q.~He were partially supported by NSFC Nos.~12471083,~12101017, 12090010, 12090015, and BJNSF No.~1214021. Q.~He was also partially supported by Peking University Funding Nos.~7101303303 and~6201001846. Y.~Zhang was partially supported by NSFC Nos.~12161141002 and 12271432, and USTC--AUST Math Basic Discipline Research Center.

	\section{Preliminaries on thermodynamic formalism}\label{sec:thermodynamicalFormalism}
	
	In this section, we go over some notations, key concepts, and useful results in ergodic theory and thermodynamic formalism.
	
	\subsection{Notations}  \label{subsec:notations}

	Denote $\N\=\{1,\,2,\,3,\,\ldots\}$ and $\N_0\=\{0\}\cup\N$. Set $\N^*\=\bigcup_{k\in\N}\N^k$.
	
	For a subset $A\subseteq X$, the characteristic function of $A$ is denoted by $\mathbbm{1}_A$. Define the global characteristic function $\mathbbm{1}\=\mathbbm{1}_X$.
	
	Let $(X,\rho)$ be a compact metric space. Let $C(X)$ and $C^{*}(X)$ denote the Banach space of continuous functions from $X$ to $\R$ equipped with the norm $\norm{u}_{\infty}\=\sup\{ \abs{u(x)} \scolon x\in X\}$ and the dual space of $C(X)$, respectively. The function $\varphi\colon X\rightarrow\R$ is said to be \emph{H\"older continuous (with respect to the metric $\rho$) with an exponent} $\alpha\in (0,1]$ if there exists $C>0$ satisfying $\abs{\varphi(x) - \varphi(y) } \leq C\rho(x,y)^{\alpha}$ for each pair of $x,\,y\in X$. Denote by $C^{0,\alpha}(X,\rho)$ the Banach space of real-valued H\"older continuous functions with an exponent $\alpha$. For each $\varphi\in C^{0,\alpha}(X,\rho)$, we define
	\begin{equation*}
		\abs{\varphi}_{\alpha,\rho}\=\sup\bigl\{ \abs{\varphi(x)-\varphi(y)}/\rho(x,y)^{\alpha} : x,\,y\in X,\,x\neq y\bigr\}.
	\end{equation*}
	More specifically, we say that a function $\phi\in C(X)$ is \defn{$c$-Lipschitz} if $\phi\in C^{0,1}(X,\rho)$ satisfies that $\abs{\phi}_{1,\rho}\leq c$. Denote by $c\hypLip(X)$ the set of real-valued $c$-Lipschitz continuous functions. We denote by $B_{\rho}(x,r)$ the open ball of radius $r>0$ centered at $x\in X$ and often omit the metric $\rho$ in the subscript when it is clear from the context. 

    Let $\PPP(X)$ and $\MMM(X,T)$ denote the set of Borel probability measures endowed with the weak$^*$ topology and the set of $T$-invariant Borel probability measures on $X$, respectively. The \defn{Wasserstein--Kantorovich metric} $W_{\rho}$ on $\PPP(X)$ metrizes the weak$^*$ topology when $X$ is compact and is defined by
	\begin{equation}   \label{e:WK_metric}
		W_{\rho}(\mu,\nu)\=\sup\biggl\{\Absbigg{\int\!f\,\mathrm{d}\mu-\int\!f\,\mathrm{d}\nu}\scolon f \in 1 \hypLip(X)\biggr\}\quad\text{ for each pair of }\mu,\,\nu\in\cP(X).
    \end{equation}
	
	Let $T\:X\rightarrow X$ be a continuous map. We define the \defn{non-wandering set} $\Omega(T)$ for $T$ by
	\begin{equation}   \label{e:nonwandering}
		\Omega(T)\=\Bigl\{x\in X: \bigcup_{n\in\N}T^{-n}(B(x,r))\cap B(x,r)\neq\emptyset \text{ for each }r>0\Bigr\}.
    \end{equation}
    Moreover, for all $\varphi\in C(X)$, $n\in\N$, and $x\in X$, we define
	\begin{equation*}
		S_n\varphi(x)\=\sum\limits_{k=0}^{n-1}\varphi\circ T^k(x),
	\end{equation*}
    where $T^k$ denotes the $k$-th iteration of the map $T$.

	\subsection{Basic concepts in ergodic theory}\label{subsec:basicconcepts}

	We begin with a brief introduction to distance-expanding maps, topologically transitive maps, and topologically exact maps. See e.g., \cite[Chapters~3 and~4]{PU10} for more details.
	
	Given two constants $\eta>0$ and $\lambda>1$, assume that $(X,\rho)$ is a compact metric space and $T\:X\rightarrow X$ is a continuous open map satisfying
	\begin{equation}\label{defnOfDistance-expansion}
		\rho(T(x),T(y))\geq\lambda\rho(x,y)\quad\text{ for each pair of }x,\,y\in X \text{ with }\rho(x,y)\leq 2\eta.
	\end{equation}
	Then there exists $\xi>0$ such that \begin{equation}\label{eq:defofdistanceexpandingxi}
		B(T(x),\xi)\subseteq T(B(x,\eta))\quad\text{ for each }x\in X.
	\end{equation} 

    By (\ref{defnOfDistance-expansion}), $T$ is injective on $B(x,\xi)$. Hence, for each $x\in X$, we can define the branch $T_x^{-1}\:B(T(x),\xi)\rightarrow B(x,\eta)$ of the inverse map. Moreover, for all $x\in X$, $n\in\N$, the composition $T_{x_0}^{-1}\circ T_{x_1}^{-1}\circ\cdots\circ T_{x_{n-1}}^{-1}\:B(T^{n}(x),\xi)\rightarrow X$ is well-defined, and is denoted by $T_x^{-n}$, where $x_j=T^j(x)$ for each integer $0\leq j\leq n-1$. Moreover, we have (see e.g.~\cite[Section~4.1]{PU10})
    \begin{equation*}
		T^{-n}(A)=\bigcup_{x\in T^{-n}(y)} T^{-n}_x(A) \quad\text{ for all }y\in X\text{ and }A\subseteq B(y,\xi),
    \end{equation*}    
    and
    \begin{equation}\label{defnOfDistance-expandingxi}
		\rho\bigl(T_x^{-n}(y),T_x^{-n}(z)\bigr)\leq\lambda^{-n}\rho(y,z) \quad \text{ for all }n\in\N,\,x\in X, \text{ and }y,\,z\in B\bigl(T^n(x),\xi\bigr).
    \end{equation}
	
	\begin{definition}\label{defndistanceexpanding}
		Let $(X,\rho)$ be a compact metric space. A continuous map $T\: X\rightarrow X$ is said to be \defn{distance-expanding} with respect to the metric $\rho$ if there exist constants $\eta>0$ and $\lambda>1$ satisfying (\ref{defnOfDistance-expansion}). Moreover, if $T$ is open, then we say that such a map $T$ is \defn{distance-expanding (with respect to $\rho$) with constants $\eta,\,\lambda$, and $\xi$} if the constants $\eta,\,\lambda$, and $\xi$ satisfy both (\ref{defnOfDistance-expansion}) and (\ref{eq:defofdistanceexpandingxi}). 
	\end{definition}
	
	We recall the definition of topologically transitive maps and topologically exact maps here (see e.g.~\cite[Definitions~4.3.1,~4.3.3]{PU10}).
	
	\begin{definition}\label{def:topologicallytransandexact}
		Let $X$ be a topological space. A continuous map $T\:X\rightarrow X$ is called \defn{topologically transitive} if, for each pair of non-empty open sets $U,\, V\subseteq X$, there exists $n\in\N$ with $T^n(U)\cap V\neq\emptyset$. As a stronger property, we say that a continuous map $T\:X\rightarrow X$ is \defn{topologically exact} if, for each non-empty open set $U\subseteq X$, there exists $n\in\N$ with $T^n(U)=X$.
	\end{definition}  
	
	Under the above conventions, given constants $\eta,\,\xi>0$, and $\lambda>1$, a list of assumptions applied in the rest of this section is listed below:
	
	\begin{assC}
		\quad
		\begin{enumerate}
			\smallskip
			\item[(i)] The metric space $(X,\rho)$ is compact.
			\smallskip
			\item[(ii)]
			The map $T\: X\rightarrow X$ is open, continuous, topologically transitive, and distance-expanding with respect to $\rho$ with constants $\eta$, $\lambda$, and $\xi$.
		\end{enumerate}
	\end{assC}
	
	Next, we recall some definitions and results in ergodic theory. For more details, we refer to \cite[Chapter~2]{PU10}.
	
	Let $X$ be a topological space. A \defn{measurable partition} $\mathcal{A}$ of $X$ is a cover $\mathcal{A}=\{A_j\scolon j\in J\}$ of $X$ consisting of finitely or countably many mutually disjoint Borel sets. Unless otherwise stated, we assume that our partitions are all measurable. For each partition $\mathcal{A}$ of $X$ and $x\in X$, the unique element of $\mathcal{A}$ containing $x$ is denoted by $\mathcal{A}(x)$.
	
	Let $\mathcal{A}=\{A_j\scolon j\in J\}$ and $\mathcal{B}=\{B_k\scolon k\in K\}$ be two covers of $X$, where $J$ and $K$ are the corresponding index sets. We say that $\mathcal{A}$ is a \defn{refinement} of $\mathcal{B}$ if for each $j\in J$, there exists $k=k(j)\in K$ such that $A_j\subseteq B_k$. The \defn{common refinement} $\mathcal{A} \vee \mathcal{B}$ of $\mathcal{A}$ and $\mathcal{B}$ is a cover defined as
	$\mathcal{A} \vee \mathcal{B} \= \{ A_j\cap B_k\scolon j\in J,\,k\in K \}.$ For each continuous map $T\:X\rightarrow X$, by $T^{-1}(\mathcal{A})$ we denote the partition $\bigl\{T^{-1}(A_j):j\in J \bigr\}$, and for each $n\in\N$, write
	\begin{equation*}
		\mathcal{A}_T^n \=  \bigvee\limits_{j=0}^{n-1} T^{-j}(\mathcal{A})= \mathcal{A}\vee T^{-1}(\mathcal{A})\vee\cdots\vee T^{-(n-1)}(\mathcal{A}).
	\end{equation*}
	
	\begin{definition}\label{def:topologicalPressureAndEntropy}
		Let $(X,\rho)$ be a compact metric space, $T\:X\rightarrow X$ be a continuous map, and $\varphi\:X\rightarrow\R$ be a continuous function (which is often called a \defn{potential}). Denote by $\mathcal{Y}$ the set of sequences $\{\mathcal{V}_n\}_{n\in\N}$ of finite open covers of $X$ satisfying that $\mesh(\mathcal{V}_n)$ converges to zero as $n$ tends to $+\infty$, where the mesh of a cover $\mathcal{V}$ of $X$ is defined as $
			\mesh(\mathcal{V}) \= \sup\{\diam(V)\scolon V\in\mathcal{V}\}.$
		Define the \defn{topological pressure} as the supremum of the following limits taken over all sequences $\{\mathcal{V}_n\}\in\mathcal{Y}$:
		\begin{equation}\label{computationofpressure}
			\lim\limits_{n\to+\infty}\lim\limits_{m\to+\infty}\frac{1}{n} \log \biggl( \inf_{\mathcal{V}} \biggl\{\sum_{U\in\mathcal{V}}  \exp (\sup\{S_m\varphi(x)\scolon x\in U\})\biggr\} \biggr),
		\end{equation}
		where $\mathcal{V}$ ranges over all covers of $X$ that are refinements of the cover $(\mathcal{V}_n)_T^m$.
		In particular, if the potential $\varphi$ is identically zero, then the topological pressure $P(T,\varphi)$ is also called the \defn{topological entropy} of $T$, and is denoted by $h_{\top}(T)$.
	\end{definition}
	
	Let $X,\,\rho,\,T$ satisfy the Assumption~C in Subsection~\ref{subsec:basicconcepts}, and $\mu\in\MMM(X,T)$ be a $T$-invariant Borel probability measure. Now we define \defn{measure-theoretic entropy} $h_{\mu}(T)$ of $T$ with respect to $\mu$. Recall that a $\mu$-measurable function $J\:X\rightarrow [0,+\infty)$ is a \defn{Jacobian} of $T$ with respect to $\mu$ if for each $\mu$-measurable set $A\subseteq X$ on which $T$ is injective, $T(A)$ is still $\mu$-measurable and $\mu(T(A))=\int_A\!J\,\mathrm{d}\mu.$ By \cite[Proposition~2.9.5]{PU10}, there exists a unique Jacobian $J_{\mu}$ of $T$ with respect to $\mu$. Then the \defn{measure-theoretic entropy} of $T$ with respect to $\mu$ can be defined by $h_{\mu}(T)\=\int\!\log\bigl(J_{\mu}\bigr)\,\mathrm{d}\mu$ (see e.g.~\cite[Proposition~2.9.7]{PU10}).
	
	Under the above notations, we have the following result (see e.g.~\cite[Theorem~3.4.1]{PU10}).
	
	\begin{theorem}[Variational Principle]\label{Variational Principle}
		Let $(X,\rho)$ be a compact metric space. Then for each continuous map $T\:X\rightarrow X$ and each continuous function $\varphi\:X\rightarrow\R$, we have
		\begin{equation*}
			P(T,\varphi)=\sup\left\{h_{\mu}(T)+\int\!\varphi\,\mathrm{d}\mu\scolon\mu\in \MMM(X,T)\right\}.
		\end{equation*}
		In particular, if $\varphi$ is identically zero, then $h_{\top}(T)=\sup\{h_{\mu}(T)\scolon\mu\in \MMM(X,T)\}$.
	\end{theorem}
	
	\begin{definition}\label{def:equilibriumstate}
		Let $(X,\rho)$ be a compact metric space. Then for each continuous map $T\:X\rightarrow X$ and each continuous function $\varphi\:X\rightarrow\R$, a measure $\mu\in \MMM(X,T)$ is called an \defn{equilibrium state} for $T$ and $\varphi$ if
		\begin{equation*}
			P(T,\varphi)=h_{\mu}(T)+\int\!\varphi\,\mathrm{d}\mu.
		\end{equation*}
		We denote by $\cE(T,\varphi)$ the set of all equilibrium states for $T$ and $\varphi$. In particular, if $\varphi$ is identically zero, then an equilibrium state for $T$ and $\varphi$ is also called a \defn{measure of maximal entropy} of $T$.
	\end{definition}

	\subsection{Ruelle--Perron--Frobenius operators and Gibbs states}\label{subsec:Ruelleoperator}

	In this subsection, we review some definitions and results for Ruelle--Perron--Frobenius operators, Gibbs states, and equilibrium states.
	
	Let $X,\,\rho,\,T$ satisfy the Assumption~C in Subsection~\ref{subsec:basicconcepts}, and $\varphi\:X\rightarrow\R$ be a continuous function. Recall that the \defn{Ruelle operator} $\mathcal{L}_{\varphi}$ acting on $C(X)$ is given by
	\begin{equation}\label{ruelle operator}
		\mathcal{L}_{\varphi}(u)(x) \= \sum\limits_{y\in T^{-1}(x)}u(y)\exp(\varphi(y))
	\end{equation}
	for each $u\in C(X)$ and each $x\in X$.
	
	By \cite[Theorem~5.2.8]{PU10}, there exists an eigenmeasure $m\in \PPP(X)$ and an eigenvalue $c>0$ of the adjoint operator $\mathcal{L}_{\varphi}^*$ of $\mathcal{L}_{\varphi}$, i.e., $\mathcal{L}^*_{\varphi} (m)=cm$. Moreover, by \cite[Proposition~5.2.11]{PU10}, the constant $c$ is precisely $\exp(P(T,\varphi)).$
	
	It is convenient to consider the \emph{normalized Ruelle--Perron--Frobenius operator} $\mathcal{L}_{\overline{\varphi}}$, where $\overline{\varphi}$ is given by $\overline{\varphi}(x) \= \varphi(x)-P(T,\varphi)$ for each $x\in X$. Hence, by (\ref{ruelle operator}), we have
	\begin{equation}\label{e:eigenmeasure}
		\mathcal{L}_{\overline{\varphi}}^*(m)=e^{-P(T,\varphi)}\mathcal{L}_{\varphi}^*(m)=m.
	\end{equation}
	
	A measure $\mu\in \PPP(X)$ is a \defn{Gibbs state} for $T$ and $\varphi$ if there exist two constants $P\in\R$ and $C\geq 1$ such that
	\begin{equation}\label{e:gibbsstate}
		C^{-1}\leq\frac{\mu(T^{-n}_x(B(T^n(x),\xi)))}{\exp(S_n\varphi(x)-Pn)}\leq C\quad\text{ for all }x\in X\text{ and }n\in\N.
	\end{equation}
	
	The following result implies that a $T$-invariant Gibbs state for $T$ and $\varphi$ is an ergodic equilibrium state for $T$ and $\varphi$.
	
	\begin{prop}
		\label{Jacobian}
		Let $X,\,\rho,\,T$ satisfy the Assumption~C in Subsection~\ref{subsec:basicconcepts}. Then for all $\mu\in \MMM(X,T)$ and H\"older continuous function $\varphi \: X \rightarrow \R$, the following statements are equivalent:
		\begin{enumerate}
			\smallskip
			\item[(i)] $\mu$ is a Gibbs state for $T$ and $\varphi$.
			\smallskip
			\item[(ii)] $\mu$ is an ergodic equilibrium state for $T$ and $\varphi$.
			\smallskip
			\item[(iii)] $\exp(-\psi)$ is a Jacobian of $T$ with respect to $\mu$, where
			\begin{equation*}
				\psi \= \varphi-P(T,\varphi)-\log(u_{\varphi}\circ T)+\log(u_{\varphi}).
			\end{equation*}
			Here $u_{\varphi}$ is a subsequential limit of the sequence $\bigl\{\frac{1}{n}\sum_{j=0}^{n-1}\mathcal{L}_{\overline{\varphi}}^j(\mathbbm{1})\bigr\}_{n\in\N}$.
		\end{enumerate}
	\end{prop}
	
	Note that (i)$\,\Rightarrow\,$(ii) immediately follows from \cite[Proposition~5.1.5~and~Corollary~5.2.13]{PU10}, (ii)$\,\Rightarrow\,$(iii) follows from \cite[Lemma~5.6.1~and Theorem~5.6.2]{PU10}, and (iii)$\,\Rightarrow\,$(i) follows from \cite[Propositions~4.4.3~and~5.2.10]{PU10}.
	
	The existence and uniqueness of the subsequential limit $u_{\varphi}$ is well-known. We include a proof in Corollary~\ref{cor:existenceofuvarphi} together with some quantitative bounds specialized for the purpose of Computable Analysis. 
	
	In the remainder of this subsection, we quantitatively re-develop some familiar facts about the Ruelle--Perron--Frobenius operators from the perspective of Computable Analysis in preparation for more technical analysis in Section~\ref{sec:computability of equilibriumstates}.
	
	\begin{lemma}\label{lemma:invariancewrtRuelleoperator}
		Fix arbitrary constants $\eta,\,\xi,\,a_0>0,\,\lambda>1$, and $v_0\in(0,1]$. Let $X,\,\rho,\,T$ satisfy the Assumption~C in Subsection~\ref{subsec:basicconcepts} with constants $\eta$, $\lambda$, and $\xi$, and $\varphi\in C^{0,v_0}(X,\rho)$ satisfy $a_0\geq\abs{\varphi}_{v_0,\rho}$. Then we have
		\begin{equation}\label{e:holderoflogiteration}
			\mathcal{L}_{\overline{\varphi}}^n(\mathbbm{1})(x)\leq\exp\bigl(a\rho(x,y)^{v_0}\bigr)\mathcal{L}_{\overline{\varphi}}^n(\mathbbm{1})(y)
		\end{equation}
		for each $n\in\N$ and each pair of $x,\,y\in X$ with $\rho(x,y)<\xi$. Here $\overline{\varphi}(x) \= \varphi(x)-P(T,\varphi)$ for each $x \in X$, and $a\=\frac{a_0}{\lambda^{v_0}-1}$. 
	\end{lemma}
	
	\begin{proof}
		Fix an integer $n>0$ and a pair of $x,\,y\in X$ with $\rho(x,y)<\xi$. Since $T$ is open and distance-expanding with constants $\eta$, $\lambda$, and $\xi$, for each $\overline{x}\in T^{-n}(x)$, by (\ref{defnOfDistance-expandingxi}), one can find a point $y(\overline{x}) \=  T_{\overline{x}}^{-n}(y)\in T^{-n}(y)$ satisfying $\rho \bigl(T^i(\overline{x}),T^i(y(\overline{x})) \bigr)\leq\lambda^{i-n}\rho(x,y)<\rho(x,y)<\xi$ for each integer $0\leq i\leq n-1$. Since $\varphi\in C^{0,v_0}(X,\rho)$, we have $\overline{\varphi}\in C^{0,v_0}(X,\rho)$ with $a_0\geq\abs{\overline{\varphi}}_{v_0,\rho}$. Hence, we have
		\begin{align*}
			\Abs{S_n\overline{\varphi}(\overline{x})-S_n\overline{\varphi}(y(\overline{x}))}
			&\leq\sum_{i=0}^{n-1}\Abs{\overline{\varphi}\bigl(T^i(\overline{x})\bigr)-\overline{\varphi}\bigl(T^i(y(\overline{x}))\bigr)} \\
			&\leq\sum_{i=0}^{n-1} a_0\rho\bigl(T^i(\overline{x}),T^i(y(\overline{x}))\bigr)^{v_0}\\
			&\leq a_0\rho(x,y)^{v_0}\cdot\sum_{i=0}^{n-1}\lambda^{(i-n)v_0} \\
			&\leq a\rho(x,y)^{v_0}
		\end{align*}
		for each $\overline{x}\in T^{-n}(x)$. Together with (\ref{ruelle operator}), we can conclude
		\begin{equation*}
			\mathcal{L}_{\overline{\varphi}}^n(\mathbbm{1})(x)=\sum_{\overline{x}\in T^{-n}(x)}e^{S_n\overline{\varphi}(\overline{x})}
            \leq\sum_{\overline{x}\in T^{-n}(x)} e^{a\rho(x,y)^{v_0}}\cdot e^{S_n\overline{\varphi}(y(\overline{x}))}
            \leq e^{a\rho(x,y)^{v_0}} \mathcal{L}^n_{\overline{\varphi}}(\mathbbm{1})(y). \qedhere
		\end{equation*}
	\end{proof}
	
	\begin{lemma}\label{lemma:boundforquotient}
		Fix arbitrary constants $\eta,\,\xi,\,a_0>0,\,\lambda>1$, and $v_0\in(0,1]$. Let $X,\,\rho,\,T$ satisfy the Assumption~C in Subsection~\ref{subsec:basicconcepts} with constants $\eta$, $\lambda$, and $\xi$, and $G=\{x_i : i\in[1,l]\cap\N\}$ be a $\xi$-net of $X$. Assume that the function $\varphi\in C^{0,v_0}(X,\rho)$ satisfies $a_0\geq\abs{\varphi}_{v_0,\rho}$. Put $a\=\frac{a_0}{\lambda^{v_0}-1}$, $D\=\max_{x\in X}\card\bigl(T^{-1}(x)\bigr)$, and $\overline{\varphi}(x) \= \varphi(x)-P(T,\varphi)$ for each $x \in X$. Then for all $x,\,y\in X$, and $n\in\N$, we have
		\begin{align}
			\qquad\qquad\,&\mathcal{L}_{\overline{\varphi}}^n(\mathbbm{1})(x)\leq C\mathcal{L}_{\overline{\varphi}}^n(\mathbbm{1})(y)\qquad\text{ and }			\label{e:C}\\
			&C^{-1}\leq\mathcal{L}^n_{\overline{\varphi}}(\mathbbm{1})(x)\leq C,	\label{e:boundoffunction}
		\end{align}
		where $C\=D^N\exp(4a\xi^{v_0}+2N \norm{\varphi}_{\infty})>1$. Here $N\in\N$ satisfies that
		\begin{equation}\label{e:cover}
			\bigcup_{k=0}^NT^k(B(x_i,\xi))=X
		\end{equation}
		for each integer $i \in [1,l]$.
	\end{lemma}
	
	Since $X$, $\rho$, and $T$ satisfy the Assumption C in Subsection~\ref{subsec:basicconcepts}, there exists an integer $N\in\N$ satisfying (\ref{e:cover}) and $D<+\infty$.
	
	\begin{proof}
		By Lemma~\ref{lemma:invariancewrtRuelleoperator}, we have
		\begin{equation}\label{e:inxi}
			 \mathcal{L}_{\overline{\varphi}}^n(\mathbbm{1})(x) \big/ \mathcal{L}_{\overline{\varphi}}^n(\mathbbm{1})(y)
			 \leq\exp(a\rho(x,y)^{v_0})\leq\exp(a\xi^{v_0})
		\end{equation}
		for each pair of $x,\,y\in X$ with $\rho(x,y)<\xi$ and each $n\in\N$.
		
		First, we establish (\ref{e:C}). To accomplish this, we fix an integer $n$ and a pair of $x,\,y\in X$. Since $G$ is a $\xi$-net, there exist two integers $i,\,j \in [1, l]$ with $\rho(x_i,x)<\xi$ and $\rho(x_j,y)<\xi$. Hence, by (\ref{e:cover}), there exists an integer $N_0 \in [0,N]$ with $T^{N_0}(B(x_i,\xi))\cap B(x_j,\xi)\neq\emptyset$. Let $u\in T^{-N_0}(B(x_j,\xi))\cap B(x_i,\xi)$. Then by (\ref{ruelle operator}) and (\ref{e:inxi}), we have
		\begin{align*}
			\mathcal{L}_{\overline{\varphi}}^n(\mathbbm{1})(x_i)&\leq \exp(a\xi^{v_0})\cdot\mathcal{L}_{\overline{\varphi}}^n(\mathbbm{1})(u)\\
			&=\exp(a\xi^{v_0})\sum_{\overline{u}\in T^{-n}(u)}\exp(S_{n+N_0}\overline{\varphi}(\overline{u})-S_{N_0}\overline{\varphi}(u))\\
			&\leq\exp\bigl(a\xi^{v_0}-N_0\inf_{x\in X}\overline{\varphi}(x)\bigr)\sum_{\overline{u}\in T^{-(n+N_0)}(T^{N_0}(u))}\exp(S_{n+N_0}\overline{\varphi}(\overline{u}))\\
			&=\exp \bigl(a\xi^{v_0}-N_0\inf_{x\in X}\overline{\varphi}(x)  \bigr)\sum_{\overline{u}\in T^{-(n+N_0)}(T^{N_0}(u))}\exp\bigl(S_n\overline{\varphi}\bigl(T^{N_0}(\overline{u})\bigr)+S_{N_0}\overline{\varphi}(\overline{u})\bigr)\\
			&\leq\exp\bigl(a\xi^{v_0}+N_0\sup_{x\in X}\overline{\varphi}(x)-N_0\inf_{x\in X}\overline{\varphi}(x)\bigr)\sum_{\overline{u}\in T^{-(n+N_0)}(T^{N_0}(u))}\exp\bigl(S_n\overline{\varphi}\bigl(T^{N_0}(\overline{u})\bigr)\bigr)\\
			&=\exp\bigl(a\xi^{v_0}+N_0\sup_{x\in X}\varphi(x)-N_0\inf_{x\in X}\varphi(x)\bigr)\sum_{\overline{u}\in T^{-(n+N_0)}(T^{N_0}(u))}\exp\bigl(S_n\overline{\varphi}\bigl(T^{N_0}(\overline{u})\bigr)\bigr)\\
			&\leq\exp(a\xi^{v_0}+2N\norm{\varphi}_{\infty})D^N\sum_{\overline{u}\in T^{-n}(T^{N_0}(u))}\exp(S_n\overline{\varphi}(\overline{u}))\\
			&=\exp(a\xi^{v_0}+2N\norm{\varphi}_{\infty})D^N\mathcal{L}_{\overline{\varphi}}^n(\mathbbm{1})\bigl(T^{N_0}(u)\bigr)\\
			&\leq\exp(2a\xi^{v_0}+2N\norm{\varphi}_{\infty})D^N\mathcal{L}_{\overline{\varphi}}^n(\mathbbm{1})(x_j).
		\end{align*}
		Hence, combining (\ref{e:inxi}) with $\rho(x_i,x)<\xi$ and $\rho(x_j,y)<\xi$, it follows that for each $n\in\N$,
		\begin{align*}
			\mathcal{L}_{\overline{\varphi}}^n(\mathbbm{1})(x)
			\leq \exp(a\xi^{v_0})\mathcal{L}_{\overline{\varphi}}^n(\mathbbm{1})(x_i)
			&\leq\exp(3a\xi^{v_0}+2N\norm{\varphi}_{\infty})D^N\mathcal{L}_{\overline{\varphi}}^n(\mathbbm{1})(x_j)\\
			&\leq\exp(4a\xi^{v_0}+2N\norm{\varphi}_{\infty})D^N\mathcal{L}_{\overline{\varphi}}^n(\mathbbm{1})(y).
		\end{align*}
		Thus, we complete the proof of (\ref{e:C}). By (\ref{e:eigenmeasure}), for each $n\in\N$, we have
		\begin{equation}\label{e:integralequalsto1}
			\int\mathcal{L}^n_{\overline{\varphi}}(\mathbbm{1})\,\mathrm{d}m=\int\!\mathbbm{1}\,\mathrm{d}m=1.
		\end{equation}
		Then by (\ref{e:integralequalsto1}), we have that $\inf_{x\in X}\mathcal{L}^n_{\overline{\varphi}}(\mathbbm{1})(x)
		\leq 1
		\leq\sup_{x\in X}\mathcal{L}^n_{\overline{\varphi}}(\mathbbm{1})(x).$ Hence, by (\ref{e:C}), we obtain (\ref{e:boundoffunction}) for all $n\in\N$ and $x\in X$.
	\end{proof}
	
	\begin{cor}\label{cor:existenceofuvarphi}
		Fix arbitrary constants $\eta,\,\xi,\,a_0>0,\,\lambda>1$, and $v_0\in(0,1]$. Let $X,\,\rho,\,T$ satisfy the Assumption~C in Subsection~\ref{subsec:basicconcepts} with constants $\eta$, $\lambda$, and $\xi$, and $\varphi\in C^{0,v_0}(X,\rho)$ satisfy $a_0\geq\abs{\varphi}_{v_0,\rho}$. Define $\overline{\varphi}(x) \= \varphi(x)-P(T,\varphi)$ for each $x \in X$. Then the sequence $\bigl\{\frac{1}{n}\sum_{j=0}^{n-1}\mathcal{L}_{\overline{\varphi}}^j(\mathbbm{1})\bigr\}_{n\in\N}$ converges uniformly to a function $u_{\varphi}\in C(X)$ satisfying
		\begin{equation}\label{e:uvarphiiseigenfunction}
			\mathcal{L}_{\overline{\varphi}}(u_{\varphi})=u_{\varphi},
		\end{equation}
		\begin{equation}\label{e:holderofloguvarphi}
			u_{\varphi}(x)\leq \exp ( a_0\rho(x,y)^{v_0} / (\lambda^{v_0}-1) )\cdot u_{\varphi}(y)
		\end{equation}
		for each pair of $x,\,y\in X$ with $\rho(x,y)<\xi$,
		and
		\begin{equation}\label{e:inequalityuphi}
			C^{-1}\leq u_{\varphi}(x)\leq C
		\end{equation}
		for each $x\in X$, where $C\geq 1$ is a constant from Lemma~\ref{lemma:boundforquotient}. Moreover, if $m\in \PPP(X)$ satisfies (\ref{e:eigenmeasure}), then        	
		\begin{equation}\label{e:uvarphiintegral}
			\int\! u_{\varphi}\,\mathrm{d}m=1,
		\end{equation}
		and the measure $\mu$ given by
		\begin{equation}\label{e:defofmu}
			\mu(A)\=\int_A\!u_{\varphi}\,\mathrm{d}m,\quad\text{ for each Borel set }A\subseteq X,
		\end{equation}
		is the unique $T$-invariant Gibbs state for $T$ and $\varphi$.
	\end{cor}
	
	\begin{proof}
		To prove this corollary, we first demonstrate
		(\ref{e:uvarphiiseigenfunction}), (\ref{e:holderofloguvarphi}), (\ref{e:inequalityuphi}), and (\ref{e:uvarphiintegral}) for a subsequential limit of the sequence $\bigl\{\frac{1}{n}\sum_{j=0}^{n-1}\mathcal{L}_{\overline{\varphi}}^j(\mathbbm{1})\bigr\}_{n\in\N}$, then prove that this sequence has a unique subsequential limit. Finally, we show that $\mu$ is the unique $T$-invariant Gibbs state for $T$ and $\varphi$.
		
		First, by (\ref{e:holderoflogiteration}) and (\ref{e:boundoffunction}), we have 
		\begin{equation}\label{e:Holderofiteration}
			\Absbig{\mathcal{L}_{\overline{\varphi}}^n(\mathbbm{1})(x)-\mathcal{L}_{\overline{\varphi}}^n(\mathbbm{1})(y)}
			\leq\Absbigg{\frac{\mathcal{L}_{\overline{\varphi}}^n(\mathbbm{1})(x)}{\mathcal{L}_{\overline{\varphi}}^n(\mathbbm{1})(y)}-1 }    \cdot     \Absbig{\mathcal{L}_{\overline{\varphi}}^n(\mathbbm{1})(y)}
			\leq C\biggl(\exp\biggl(\frac{a_0\rho(x,y)^{v_0}}{\lambda^{v_0}-1}\biggr)-1\biggr)
		\end{equation}
		for all $n\in\N$ and $x,y\in X$ with $\rho(x,y)<\xi$ and $\mathcal{L}_{\overline{\varphi}}^n(\mathbbm{1})(x)\geq\mathcal{L}_{\overline{\varphi}}^n(\mathbbm{1})(y)$. For each $n\in\N$, set $u_n\=\frac{1}{n}\sum_{j=0}^{n-1}\mathcal{L}_{\overline{\varphi}}^j(\mathbbm{1})$. By (\ref{e:boundoffunction}) and (\ref{e:Holderofiteration}), $\{u_n\}_{n\in\N}$ is a uniformly bounded sequence of equicontinuous functions on $X$. By the Arzel\`a--Ascoli theorem (see e.g.~\cite[Theorem~4.44]{Fo99}), there exists a continuous function $u_{\varphi}$ and an increasing sequence $\{n_i\}_{i\in\N}$ such that $u_{n_i}\rightarrow u_{\varphi}$ uniformly on $X$ as $i\rightarrow+\infty$. Thus it follows from the definition of $\{u_n\}_{n\in\N}$, (\ref{e:holderoflogiteration}), (\ref{e:boundoffunction}), and (\ref{e:integralequalsto1}) that $u_{\varphi}$ satisfies (\ref{e:uvarphiiseigenfunction}), (\ref{e:holderofloguvarphi}), (\ref{e:inequalityuphi}), and (\ref{e:uvarphiintegral}).
		
		Next, it remains to prove that $u_{\varphi}$ is the unique subsequential limit of $\{u_n\}_{n\in\N}$. This part of the proof follows verbatim the same as the proof of \cite[Theorem 5.16]{Li18}.
		
		
		By \cite[Proposition~5.2.11]{PU10}, $m$ is a Gibbs state for $T$ and $\varphi$. According to \cite[Proposition 5.1.1]{PU10} and (\ref{e:defofmu}), $\mu$ is also a Gibbs state for $T$ and $\varphi$. Finally, we shall prove that $\mu$ is $T$-invariant. It suffices to show that $\int\!(g\circ T)\,\mathrm{d}\mu=\int\!g\,\mathrm{d}\mu$ for each $g\in C(X)$. Indeed, by (\ref{e:defofmu}), (\ref{e:eigenmeasure}), and (\ref{e:uvarphiiseigenfunction}), we obtain that
		\begin{align*}
			\int\!(g\circ T)\,\mathrm{d}\mu
            &=\int\!u_{\varphi}\cdot(g\circ T)\,\mathrm{d}m
            =\int\!u_{\varphi}\cdot(g\circ T)\,\mathrm{d}\bigl(\cL^*_{\overline{\varphi}}(m)\bigr) \\
            &=\int\!\cL_{\overline{\varphi}}(u_{\varphi})\cdot g\,\mathrm{d}m
            =\int\!u_{\varphi}\cdot g\,\mathrm{d}m
            =\int\!g\,\mathrm{d}\mu.
		\end{align*} 
		Therefore, by \cite[Corollary~5.2.14]{PU10}, $\mu$ is the unique $T$-invariant Gibbs state for $T$ and $\varphi$.
	\end{proof}

	\section{Preliminaries on Computable Analysis}\label{sec:Computable Analysis}
	In this section, we recall some notions and results from Computable Analysis. The definitions we adopt in this section are consistent with those in \cite{Weih00}. Consequently, it is convenient to think of the algorithms or machines mentioned below as Type-2 machines defined in \cite[Definition~2.1.1]{Weih00}. For more details, we refer the reader to \cite[Section~3]{BBRY11}, \cite[Section~2]{GHR11}, and \cite{Weih00}.
	
	\subsection{Algorithms and computability over the reals}
	
	\begin{definition}
		Given $k\in\N$, we say that a function $f\:\N^k\rightarrow\Z$ is \defn{computable}, if there exists an algorithm $\mathcal{A}$ which, upon input a sequence of $k$ positive integers $\{x_i\}_{i=1}^k$, outputs the value of $f(x_1,x_2,\ldots,x_k)$.
	\end{definition}
	
	For a countable set $S$, by an \defn{effective enumeration} of $S$ we mean 
	an enumeration $S=\{x_i\}_{i\in\N}$ satisfying that there exists an algorithm $\mathcal{A}$ which, upon input $i\in\N$ outputs $x_i$. 
	
	\begin{definition}
		A real number $x$ is called \begin{enumerate}
			\smallskip
			\item[(i)] \defn{computable} if there exist two computable functions $f\colon\N\rightarrow\Z$ and $g\colon\N\rightarrow\N$ satisfying that for each $n\in\N$,
			$
				\Absbig{\frac{f(n)}{g(n)}-x}<2^{-n};
			$
			\smallskip
			\item[(ii)] \defn{lower-computable} (resp.\ \defn{upper-computable}) if there exist two computable functions $f\colon\N\rightarrow\Z$ and $g\colon\N\rightarrow\N$ satisfying that $\bigl\{\frac{f(n)}{g(n)}\bigr\}_{n\in\N}$ is an increasing (resp.\ decreasing) sequence and converges to $x$ as $n\rightarrow+\infty$.
		\end{enumerate}
	\end{definition}
	
	\subsection{Computable metric spaces}\label{subsec:computability metric space}
	
	The above definitions equip the real numbers with a computability structure. This can be extended to virtually any separable metric space. We now give a short introduction.
	\begin{definition}\label{def:computablemetricspace}
		A \defn{computable metric space} is a triple $(X,\,\rho,\,\mathcal{S})$, where
		\begin{enumerate}
			\smallskip
			\item[(i)] $(X,\rho)$ is a separable metric space;
			\smallskip
			\item[(ii)] $\mathcal{S}=\{s_n\scolon n\in\N\}$ is a dense subset of $X$;
			\smallskip
			\item[(iii)] there exists an algorithm which, on input $i,\,j,\,m\in\N$, outputs $y_{i,j,m}\in\Q$ satisfying 
              \begin{equation*}
                   \abs{ y_{i,j,m}-\rho(s_i,s_j) } <2^{-m} .
              \end{equation*}
		\end{enumerate}
		The points in $\mathcal{S}$ are said to be \defn{ideal}. Due to the existence of a computable bijection between $\N^3$ and $\N$, there exists an effective enumeration $\{B_l\}_{l\in\N}$ of the set $\{B(s_i,j/k)\scolon i,\,j,\,k\in\N\}$ of balls with rational radii centered at points in $\mathcal{S}$. These balls are called the \defn{ideal balls} in $(X,\,\rho,\,\mathcal{S})$.
	\end{definition}
	
	\begin{definition}\label{def:uniformcomp}
		Let $(X,\,\rho,\,\mathcal{S})$ be a computable metric space. We say that a point $x\in X$ is \defn{computable} if there exists a computable function $f\colon\N\rightarrow\N$ such that $\rho\bigl(s_{f(n)},x\bigr)<2^{-n}$ for each $n\in\N.$ Moreover, a sequence of points $\{x_i\}_{i\in\N}$ is said to be a \defn{uniformly computable sequence of points} if there exists a computable function $f\colon\N^2\rightarrow\N$ such that $\rho\bigl(s_{f(n,m)},x_m\bigr)<2^{-n}$ for all $n,\,m\in\N$.
	\end{definition}
	
	\begin{remark}
		For each finite sequence $\{x_i\}_{i=1}^n$ of computable points, we see it as $\{x_i\}_{i\in\N}$, where $x_j=x_n$ for each integer $j\geq n$. Then by Definition~\ref{def:uniformcomp}, it is not hard to see that $\{x_i\}_{i\in\N}$ is a uniformly computable sequence. As a convention, we will say that a finite sequence $\{x_i\}_{i=1}^n$ of computable points is a uniformly computable sequence directly. Similarly, for other definitions of computable objects detailed below, we always have that a finite sequence of computable objects is always a uniformly computable sequence.
	\end{remark} 
	
	\begin{definition}\label{uppercomputableclosed}
		In a computable metric space $(X,\,\rho,\,\mathcal{S})$, an open set $U\subseteq X$ is called \defn{lower-computable} if there is a computable function $f\colon\N\rightarrow\N$ such that $U=\bigcup_{n\in\N}B_{f(n)}$. A closed set $K$ is said to be \defn{upper-computable} if its complement is a lower-computable open set.
	\end{definition}
	
	\begin{definition}\label{uniformlylowercomputableopen}
		In a computable metric space $(X,\,\rho,\,\mathcal{S})$, a family $\{U_i\}_{i\in\N}$ of lower-computable open sets is called a \defn{uniformly computable sequence of lower-computable open sets} if there is a computable function $f\colon\N^2\rightarrow\N$ such that $U_i=\bigcup_{n\in\N}B_{f(i,n)}$.
	\end{definition}
	
	By Definitions~\ref{uppercomputableclosed},~\ref{uniformlylowercomputableopen}, and the existence of computable bijections between $\N^2$ and $\N$, we obtain the following result.
	
	\begin{prop}
		\label{lower-computable open}
		For a uniformly computable sequence $\{U_i\}_{i\in \N}$ of lower-computable open sets in a computable metric space $(X,\,\rho,\,\mathcal{S})$, their union $U=\bigcup_{i\in \N}U_i$ is a lower-computable open set.
	\end{prop}
	
	Before the definition of computable functions between computable metric spaces, we recall the definition of an oracle of a point in a computable metric space.
	
	\begin{definition}\label{def:oracle}
		Given a computable metric space $(X,\,\rho,\,\mathcal{S})$ and a point $x\in X$, we say that a function $\varphi\colon\N\rightarrow\N$ is an \defn{oracle} for $x\in X$ if $\rho\bigl( s_{\varphi(m)},x\bigr) < 2^{-m}$ for each $m\in\N.$
	\end{definition}
	
	\begin{definition}\label{Algorithm about computable functions}
		Assume that $(X,\,\rho,\,\mathcal{S})$ and $(X',\,\rho',\,\mathcal{S}')$ are two computable metric spaces. Then a function $f\: X\rightarrow X'$ is \defn{computable} if there exists an algorithm which, for each $x\in X$ and each $n\in\N$, on input $n\in\N$ and an oracle $\varphi$ for $x$, outputs $m\in\N$ satisfying $\rho'(s'_m,f(x))<2^{-n}$. Moreover, a sequence $\{f_i\}_{i\in\N}$ of functions $f_i\:X\rightarrow X'$ is called a \defn{uniformly computable sequence of functions} if there exists an algorithm which, for each $x\in X$ and each $n\in\N$, on input $i,\,n\in\N$, and an oracle $\varphi$ for $x$, outputs $m\in\N$ satisfying $\rho'(s'_m,f_i(x))<2^{-n}$.
	\end{definition}
	
	For example, in \cite{Weih00}, Examples~4.3.3 and~4.3.13.5 give the computability of the exponential function $\exp\:\R\rightarrow\R$ and the logarithmic function $\log\colon\R^+\rightarrow\R$, respectively. The following proposition is a classical result that describes a topological property of a computable function (see e.g.~\cite[Proposition~5.2.14]{BH21}). 
	
	\begin{prop}
		\label{definition of computable function}
		Assume that $(X,\,\rho,\,\mathcal{S})$ and $(X',\,\rho',\,\mathcal{S}')$ are two computable metric spaces with $\cS=\{s_i\}_{i\in\N}$ and $\cS'=\bigl\{s_i'\bigr\}_{i\in\N}$, and $\bigl\{B_i'\bigr\}_{i\in\N}$ is the effective enumeration of ideal balls of $(X',\,\rho',\,\mathcal{S}')$. If $f\:X\rightarrow X'$ is computable, then $\bigl\{f^{-1}(B_i') \bigr\}_{i\in\N}$ is a uniformly computable sequence of lower-computable open sets in the computable metric space $(X,\,\rho,\,\mathcal{S})$, and preimages of lower-computable open sets are still lower-computable open sets.
	\end{prop}

	According to Definition~\ref{Algorithm about computable functions}, we can obtain the following corollary.
	
	\begin{cor}\label{cor:uniformlycomputablefunction}
		Assume that $(X,\,\rho,\,\mathcal{S})$ and $(X',\,\rho',\,\mathcal{S}')$ are two computable metric spaces with $\cS=\{s_i\}_{i\in\N}$ and $\cS'=\{s_i'\}_{i\in\N}$, and $\{B_i'\}_{i\in\N}$ is the effective enumeration of ideal balls of $(X',\,\rho',\,\mathcal{S}')$. If a sequence $\{f_i\}_{i\in\N}$ of functions $f_i\:X\rightarrow X'$ is a uniformly computable sequence of functions, then $\bigl\{f_j^{-1}(B_i'): i,\,j\in\N\bigr\}$ is a uniformly computable sequence of lower-computable open sets in the computable metric space $(X,\,\rho,\,\mathcal{S})$.
	\end{cor}
	
	\begin{definition}\label{defn:domainofcomputability}
		Fix an effective enumeration $\{q_n\}_{n\in\N}$ of $\mathbb{Q}$. Let $(X,\,\rho,\,\mathcal{S})$ be a computable metric space, and $C$ be a subset of $X$. A function $T\:X\rightarrow\R$ is said to be \defn{upper-computable on $C$} if there is a uniformly computable sequence $\{U_i\}_{i\in\N}$ of lower-computable open sets in the computable metric space $(X,\,\rho,\,\mathcal{S})$ satisfying
		\begin{equation*}
			T^{-1}(-\infty,q_i)\cap C= U_i\cap C.
		\end{equation*}
	\end{definition}
	
	\begin{definition}
		\label{defn:computableinverse}
		Let $(X,\,\rho,\,\mathcal{S})$ be a computable metric space, and $T\: X\rightarrow X$ a function with $p=p(x) \= \card\bigl( T^{-1}(x) \bigr)<+\infty$ for each $x\in X$. Then the inverse $T^{-1}$ of $T$ is said to be \defn{computable} if there exists an algorithm $\mathcal{A}$ such that for each $x\in X$ and each $n\in\N$, we can input an oracle $\varphi$ for $x$ in $\mathcal{A}$ to produce $\{y_i\scolon 1\leq i\leq p(x)\}$ with the following property:
		
		\smallskip
		
		If $T^{-1}(x)=\{x_i\scolon 1\leq i\leq p(x)\}$, then for each $1\leq i\leq p(x)$, there exists $1\leq j\leq p(x)$ such that
		$
			\rho(x_i,y_j)<2^{-n}.
		$
	\end{definition}
	
	It follows immediately from Definition~\ref{Algorithm about computable functions} that the computable real-valued functions are closed under finitely many operations from the following list: addition, multiplication, division, scalar multiplication, $\max,$ and $\min$ (see e.g.~\cite[Corollary~4.3.4]{Weih00}).
	
	At the end of this subsection, we recall the definitions of a recursively compact set and a recursively precompact metric space introduced in \cite[Section~2]{GHR11}.
	
	\begin{definition}\label{definition of recursively compact}
		In a computable metric space $(X,\,\rho,\,\mathcal{S})$, a set $K\subseteq X$ is said to be \defn{recursively compact} if it is compact and there is an algorithm which, on input a sequence $\{i_j\}_{j=1}^p$ in $\N$ and a sequence $\{q_j\}_{j=1}^p$ in $\Q^+$, halts if and only if $K\subseteq\bigcup_{j=1}^pB\bigl(s_{i_j},q_j\bigr)$.
	\end{definition}
	
	The following are some fundamental properties of recursively compact sets as discussed in \cite[Proposition~1]{GHR11}.
	
	\begin{prop}\label{prop:thecomplementofrecursivelycompactset}
		Let $(X,\,\rho,\,\mathcal{S})$ be a computable metric space. Assume that $K\subseteq X$ is a recursively compact set and $f\colon X\rightarrow\R$ is a computable function. Then the following statements are true:
		\begin{enumerate}
			\smallskip
			\item[(i)] A point $x\in X$ is computable if and only if the singleton $\{x\}$ is a recursively compact set.
			\smallskip
			\item[(ii)] $X\smallsetminus K$ is a lower-computable open set.
			\smallskip
			\item[(iii)] If $U\subseteq X$ is a lower-computable open set, then $K\smallsetminus U$ is recursively compact.
			\smallskip
			\item[(iv)] $\inf_{x\in K}f(x)$ is lower-computable and $\sup_{x\in K}f(x)$ is upper-computable.
			\smallskip
			\item[(v)] If $K^{\prime}\subseteq X$ is a recursively compact set, then so is $K^{\prime}\cap K$.
		\end{enumerate}
	\end{prop}
	
	\begin{definition}\label{definition of recursively precompactness}
		A computable metric space $(X,\,\rho,\,\mathcal{S})$ is \defn{recursively precompact} if there is an algorithm which, on input $n\in\N$, outputs a subset $\{i_1,\,i_2,\,\ldots,\,i_p\}$ of $\N$ satisfying that $\{s_{i_1},\,s_{i_2},\,\ldots,\,s_{i_p}\}$ is a $2^{-n}$-net of $X$.
	\end{definition}
	
	\subsection{Computability of probability measures}
	
	Following \cite[Proposition~4.1.3]{HR09}, if $(X,\,\rho,\,\mathcal{S})$ is a computable metric space, and $X$ is bounded, then $\bigl(\PPP(X),\, W_{\rho},\,\mathcal{R}_{\mathcal{S}}\bigr)$ is a computable metric space, where the Wasserstein--Kantorovich metric $W_{\rho}$ is recalled in (\ref{e:WK_metric}), and $\mathcal{R}_{\mathcal{S}}\subseteq \PPP(X)$ is the set of Borel probability measures $\mu$ satisfying $\supp(\mu)\subseteq\cS,\,\card(\supp(\mu))<+\infty$, and $\mu(\{x\})\in\Q$ for each $x\in\cS$. Recall that the support $\supp(\mu)$ of a Borel measure $\mu$ is the complement of the union of all the open sets $U$ with $\mu(U)=0$.
	
	\begin{definition}
		\label{definition of computable measure}
		Let $(X,\,\rho,\,\mathcal{S})$ be a computable metric space, and $X$ be a bounded set. Then a \defn{computable measure} $\mu$ is a computable point of $\bigl(\PPP(X),\,W_{\rho},\,\mathcal{R}_{\mathcal{S}}\bigr)$.
	\end{definition}
	
	By \cite[Lemma~2.12 and Proposition~4]{GHR11}, due to the completeness of $\PPP(X)$ with respect to the metric $W_{\rho}$, one can conclude the following result.
	
	\begin{prop}
		\label{recursively compactness of measure space}
		Let $(X,\,\rho,\,\mathcal{S})$ be a computable metric space, and $X$ be recursively compact. Then $\PPP(X)$ is a recursively compact set in $\bigl(\PPP(X),\,W_{\rho},\,\mathcal{R}_{\mathcal{S}}\bigr)$.
	\end{prop}
	
	\begin{prop}[{\cite[Corollary~4.3.2]{HR09}}]
		Assume that $(X,\,\rho,\,\mathcal{S})$ is a computable metric space, and $\{f_i\}_{i \in \N}$ is a uniformly computable sequence of real-valued functions on $X$. If there is an algorithm which, on input $i\in\N$, outputs $M_i$ such that $\norm{f_i}_{\infty}\leq M_i$, then the sequence of integral operators $\mathcal{I}_i\: \PPP(X)\rightarrow\R$, $i \in \N$, defined by $\mathcal{I}_i(\mu) \= \int\! f_i\,\mathrm{d}\mu$,
		is a uniformly computable sequence of functions.
	\end{prop}
	
	Assume, in addition, that $X$ is recursively compact. Then by \cite[Proposition~2.13]{BRY12}, upper bounds of computable functions on $X$ can be computed. Hence, we have the following result.
	
	\begin{cor}\label{computabilityofintegraloperator}
		Let $(X,\,\rho,\,\mathcal{S})$ be a computable metric space, and $X$ be recursively compact. Assume that $\{f_i\}_{i\in\N}$ is a uniformly computable sequence of real-valued functions on $X$. Then the sequence of integral operators $\mathcal{I}_i\: \PPP(X)\rightarrow\R$, $i\in \N$, defined by $\mathcal{I}_i(\mu) \= \int\! f_i\,\mathrm{d}\mu$, is a uniformly computable sequence of functions.
	\end{cor}
	
	Finally, we consider a family of computable functions.
	
	\begin{definition}\label{def:hatfunction}
		Let $(X,\rho)$ be a metric space. Consider arbitrary constants $r,\,\epsilon>0$, and a point $u\in X$. Then the function $g_{u,r,\epsilon}$ given by
		\begin{equation}\label{hat function}
			g_{u,r,\epsilon}(x) \= \Absbigg{ 1-\frac{ \abs{\rho(x,u)-r}^+}{\epsilon} }^+\quad\text{ for each }x\in X,
		\end{equation}
		is called a \defn{hat function}. Here $\abs{a}^+ \= \max\{a,\,0\}$.
	\end{definition}
	
	The hat function $g_{u,r,\epsilon}(x)$ is an $\epsilon^{-1}$-Lipschitz continuous function that takes the value $1$ in the ball $B(u,r)$ and vanishes outside the enlarged ball $B(u,r+\epsilon)$.
	
	\begin{definition}\label{def:testfunction}
		In a computable metric space $(X,\,\rho,\,\mathcal{S})$, assume that $\mathcal{F}_0(\mathcal{S})$ is the set of functions of the form $g_{u,\,r,\,1/n}(x)$, where $u\in\mathcal{S},\,r\in\Q^+$, $n\in\N$, and $\mathfrak{E}(\mathcal{S})$ is the smallest set of functions containing $\mathcal{F}_0$ and the constant function $1$ closed under $\max,\,\min$, and finite rational linear combinations. Then the elements in $\mathfrak{E}(\mathcal{S})$ are called \defn{test functions}.
	\end{definition}
	
	\begin{remark}\label{rem:theeffectiveenumerationoftestfunctions}
		Note that there exists a computable bijection between $\N^*$ and $\N$. Hence, from Definitions~\ref{def:hatfunction} and~\ref{def:testfunction}, it is not hard to construct an effective enumeration $\{\varphi_j\}_{j\in\N}$ of $\mathfrak{E}(\mathcal{S})$, i.e., $\{\varphi_j\}_{j\in\N}$ is a sequence of uniformly computable functions. We fix such an effective enumeration of $\mathfrak{E}(\cS)$ and call it the \emph{effective enumeration} of $\mathfrak{E}(\cS)$ in $(X,\,\rho,\,\cS)$.
	\end{remark}
	
	Assume that $X$ is a compact separable metric space and $\mathcal{S}$ a dense subset of $X$. Then by the Stone--Weierstrass theorem (see e.g.~\cite[Theorem~4.45]{Fo99}), we have that $\mathfrak{E}(\mathcal{S})$ is dense in $C(X)$. Thus, using the dominated convergence theorem, we conclude the following proposition.
	\begin{prop}
		\label{test function}
		Let $(X,\,\rho,\,\mathcal{S})$ be a computable metric space, $X$ be recursively compact, and $\{\varphi_j\}_{j\in\N}$ be a computable enumeration of $\mathfrak{E}(\mathcal{S})$. Then for each pair of Borel probability measures $\mu,\,\nu\in \PPP(X)$, $\mu=\nu$ if and only if $\int\!\varphi_j\,\mathrm{d}\mu=\int\!\varphi_j\,\mathrm{d}\nu$ for each $j\in\N.$
	\end{prop}
	
	\section{Computable Analysis on thermodynamic formalism}\label{sec:computability of equilibriumstates}
	
	In this section, we develop the cone technique to establish Theorems~\ref{maintheorem} and~\ref{maintheorem'}. In Subsection~\ref{subsec:conemethod}, we estimate the rate of convergence of iterations of the normalized Ruelle--Perron--Frobenius operator. In Subsection~\ref{subsec:computability of Jacobian}, we prove that the Jacobian function $J_{\varphi}(x)$ (see (\ref{expressionofJ})) is computable. In Subsection~\ref{subsec:lemmalower-computable}, we demonstrate that two subsets $U$ and $W$ of $\PPP(X)$ (from Lemmas~\ref{recursivelycompactofinvariantmeasure} and~\ref{lemma:lower-computable}) are both lower-computable open sets. In Subsection~\ref{subsec:computability of equilibriumstates}, we complete the proofs of Theorems~\ref{maintheorem} and~\ref{maintheorem'}.
	
	\subsection{Ruelle--Perron--Frobenius operators and cones}\label{subsec:conemethod}

	In this subsection, we develop the cone technique, combining ideas from \cite{PU10}, to prove in Theorem~\ref{t:eigenfunction} that the convergence of iterations of the normalized Ruelle--Perron--Frobenius operator is at an exponential rate with explicit expression for the related constants. The cone technique, first used by Birkhoff \cite{Bi57} in the study of Banach spaces geometry, was reintroduced by Liverani \cite{Liv95a, Liv95b} for the research on the decay of correlations. For a quick introduction to the cone technique in the context of connected spaces, see \cite[Section~2.2]{Vi97}. The lack of connectedness of $X$ poses a key obstacle in our study of Computable Analysis. Due to this obstacle and the demand on the computability (in the sense of Computable Analysis) of various constants involved, we need to refine the cone technique further.
	
	In this subsection, we denote
	\begin{equation}\label{e:hK}
		\hK (\lambda,b,b')\=  2\log\biggl(\frac{1+\lambda}{1-\lambda}\cdot\frac{b^2}{b-b'}\biggr)>0
	\end{equation}
	for all constants $\lambda,\,b,\,b^{\prime}\in\R$ with $1\leq b'<b$ and $\lambda\in (0,1)$.
	
	\begin{theorem}		\label{t:eigenfunction}
		Fix arbitrary constants $\eta,\,\xi,\,a_0>0,\,\lambda>1$, and $v_0\in(0,1]$. Let $(X,\rho)$ be a compact metric space, $T\colon X\rightarrow X$ be a continuous, topologically exact, open, and distance-expanding with constants $\eta$, $\lambda$, and $\xi$. Assume that $\varphi\in C^{0,v_0}(X,\rho)$ satisfies $a_0\geq\abs{\varphi}_{v_0,\rho}$.
		
		Put $\lambda_1\=\frac{1+\lambda^{-v_0}}{2}<1$ and  $a'\=\frac{6a_0\lambda^{v_0}-2a_0}{(\lambda^{v_0}-1)^2}>0$. Fix an arbitrary constant $C'>\max \bigl\{C^2,\,2\exp(a') \bigr\}$, where $C>1$ is the constant from Lemma~\ref{lemma:boundforquotient}. Define $\overline{\varphi}(x) \= \varphi(x)-P(T,\varphi)$ for each $x\in X$, and $Q\=\exp(\inf_{x\in X}\overline{\varphi}(x))>0$. 
		Suppose that $\epsilon\in(0,\xi)$ is a constant with $2\exp(a'\epsilon^{v_0})\leq C'$, and that $G'=\{x'_i : i\in [1,n^{\prime}]\cap\N\}$ is an $(\epsilon/2)$-net of $X$.
		Fix an arbitrary integer $m\in\N$ satisfying $\bigcap_{i=1}^{n'}T^m(B(x'_i,\epsilon/2))=X.$
		
		Then the following inequality 
		\begin{equation*}
			\Normbig{\mathcal{L}^{mk}_{\overline{\varphi}}(\mathbbm{1})-u_{\varphi}}_{\infty}\leq CeZ(1-\exp(-Z))^{k-1}
		\end{equation*}
		holds for each $k\in\N$ with $0\leq Z(1-\exp(-Z))^{k-1}\leq 1$, where $u_{\varphi}$ is the function from Corollary~\ref{cor:existenceofuvarphi} and $Z\=\hK\bigl(\lambda_1,C',\frac{(C^{-2}Q^m+2)C'}{2C^{-2}Q^m+2}\bigr)>0$.
	\end{theorem}
	
	The proof of Theorem~\ref{t:eigenfunction} will be given at the end of this subsection.
	
	First, we introduce some notations in the cone technique. Let $E$ be a vector space over $\R$. A \defn{convex cone} in $E$ is a subset $\cC\subseteq E\smallsetminus\{0\}$ satisfying the following properties:
	\begin{enumerate}
		\smallskip
		\item[(i)] $tv\in\cC$ for all $v\in\cC$ and $t>0$.
		\smallskip
		\item[(ii)]$t_1v_1+t_2v_2\in\cC$ for all $v_1,\,v_2\in\cC$ and $t_1,\,t_2>0$.
		\smallskip
		\item[(iii)] $\{v\in E :  v\in\overline{\cC},\,-v\in \overline{\cC}\}=\{0\}$.
	\end{enumerate}
	Here $\overline{\cC}$ is the set consisting of $w\in E$ that satisfies that there exists $v\in\cC$ and $\{t_n\}_{n\in\N}$ in $\R^+$ satisfying $w+t_nv\in\cC$ for each $n\in\N$ and that $t_n$ converges to $0$ as $n$ tends to $+\infty$.
	
	Consider an arbitrary pair of $v_1,\,v_2\in\cC$. Then we define
	\begin{equation}\label{e:defofalphabeta}
		\alpha(v_1,v_2) \= \sup\{t>0\scolon v_2-tv_1\in\cC\}\quad\text{ and }\quad\beta(v_1,v_2) \= \inf\{s>0\scolon sv_1-v_2\in\cC\},
	\end{equation}
	with the conventions that $\sup\emptyset=0,\,\inf\emptyset=+\infty$. By the definition of convex cones, we have $\alpha(v_1,v_2)\leq\beta(v_1,v_2),\,\alpha(v_1,v_2)<+\infty,$ and $\beta(v_1,v_2)>0$. Thus $\theta(v_1,v_2)$ given by	\begin{equation}\label{e:defoftheta}
		\theta(v_1,v_2) \=
		\begin{cases}
			+\infty&\text{if }\alpha(v_1,v_2)=0\text{ or }\beta(v_1,v_2)=+\infty;\\
			\log(\beta(v_1,v_2))-\log(\alpha(v_1,v_2))&\text{otherwise}
		\end{cases}
	\end{equation}
	is well-defined and takes value in $[0,+\infty]$. Here $\theta$ is called the \defn{projective metric} associated to $\cC$.
	
	\begin{prop}[{\cite[Proposition~2.2]{Vi97}}]\label{p:projective_metric}
		Let $\cC$ be a convex cone. Then the projective metric $\theta$ satisfies the following properties:
		\begin{enumerate}
			\smallskip
			\item[(i)] $\theta(v_1,v_2) = \theta(v_2,v_1)$ for all $v_1,\,v_2 \in \cC$,
			\smallskip
			\item[(ii)] $\theta(v_1,v_2) + \theta(v_2,v_3) \geq \theta(v_1,v_3)$ for all $v_1,\,v_2,\,v_3 \in \cC$,
			\smallskip
			\item[(iii)] for all $v_1,\,v_2\in\cC$, $\theta(v_1,v_2)=0$ if and only if there exists $t>0$ with $v_1 = t v_2$.
		\end{enumerate}
	\end{prop}
	
	Moreover, the following proposition asserts that $L\:\cC_1\rightarrow\cC_2$ is a strict contraction with respect to $\theta_1$ and $\theta_2$ if $L(\cC_1)$ has finite $\theta_2$-diameter. Here $\theta_i$ is the projective metric associated to $\cC_i$ for each $i \in \{1,\,2\}$.
	
	\begin{prop}[{\cite[Proposition~2.3]{Vi97}}]\label{prop:Liscontraction}
		Let $E_i$ be a vector space, and $\cC_i\subseteq E_i$ be a convex cone with the projective metric $\theta_i$ for each $i\in\{1,\,2\}$. Assume that $L\: E_1\rightarrow E_2$ is a linear operator with $L(\cC_1)\subseteq\cC_2$, and $U\=\sup\{\theta_2(L(v),L(w))\scolon v,\,w\in\cC_1\}<+\infty$. Then for each pair of $v,\,w\in\cC_1$,
		\begin{equation*}
			\theta_2(L(v),L(w))\leq \bigl(1-e^{-U}\bigr)\cdot\theta_1(v,w).
		\end{equation*}

	\end{prop}
	
	To prove Theorem~\ref{t:eigenfunction}, we introduce the following family of convex cones with the given constants $\xi>0$ and $v_0\in(0,1]$.
	
	\begin{definition}\label{def:coneavb}
		Let $(X,\rho)$ be a compact metric space. For each $a>0$ and each $b>1$, denote by $\cC(a,b)\subseteq C(X)$ the set consisting of continuous functions $u\:X\rightarrow\R$ satisfying the following properties:
		\begin{enumerate}
			\smallskip
			\item[(i)] $u(x)>0$ for each $x\in X$,
			\smallskip
			\item[(ii)] $u(x)\leq \exp(a\rho(x,y)^{v_0})\cdot u(y)$ for all $x,\,y\in X$ with $\rho(x,y)<\xi$,
			\smallskip
			\item[(iii)] $u(x)\leq bu(y)$ for each pair of $x,\,y\in X$.
		\end{enumerate}
		It is not hard to show that $\cC(a,b)$ is a convex cone. Let $\alpha_{a,b}$ and $\beta_{a,b}$ be the corresponding functions associated to $\cC(a,b)$ defined in (\ref{e:defofalphabeta}), and $\theta_{a,b}$ be the projective metric associated to $\cC(a,b)$ defined in (\ref{e:defoftheta}).
	\end{definition}
	
	Consider an arbitrary pair of $\varphi_1,\,\varphi_2\in\cC(a,b).$ We shall compute the expression of $\alpha_{a,b}(\varphi_1,\varphi_2)$. By (\ref{e:defofalphabeta}), $\alpha_{a,b}(\varphi_1,\varphi_2)$ is the supremum of all positive real numbers $t$ satisfying the following three properties for all $s,\,m,\,n,\,x,\,y\in X$ with $0<\rho(x,y)<\xi$:
	\begin{enumerate}
		\smallskip
		\item[(i)] $\varphi_2(s)-t\varphi_1(s)>0$, or equivalently, $t<\frac{\varphi_2(s)}{\varphi_1(s)}$;
		\smallskip
		\item[(ii)] $\frac{\varphi_2(x)-t\varphi_1(x)}{\varphi_2(y)-t\varphi_1(y)}\leq\exp(a\rho(x,y)^{v_0})$, or equivalently, $t\leq \frac{e^{a\rho(x,y)^{v_0}}\varphi_2(y)-\varphi_2(x)}{e^{a\rho(x,y)^{v_0}}\varphi_1(y)-\varphi_1(x)}$;
		\smallskip
		\item[(iii)] $\frac{\varphi_2(n)-t\varphi_1(n)}{\varphi_2(m)-t\varphi_1(m)}\leq b$, or equivalently, $t\leq\frac{b\varphi_2(m)-\varphi_2(n)}{b\varphi_1(m)-\varphi_1(n)}$.
	\end{enumerate}
	
	In other words,
	\begin{equation}\label{e:expressionofalpha}
		\alpha_{a,b}(\varphi_1,\varphi_2) =\min\biggl\{\inf\limits_{x,\,y}\biggl\{\frac{e^{a\rho(x,y)^{v_0}}\varphi_2(x)-\varphi_2(y)}{e^{a\rho(x,y)^{v_0}}\varphi_1(x)-\varphi_1(y)}\biggr\}, \, \inf\limits_{m,\,n\in X}\biggl\{\frac{b\varphi_2(m)-\varphi_2(n)}{b\varphi_1(m)-\varphi_1(n)}\biggr\},\, \inf\limits_{s\in X}\biggl\{\frac{\varphi_2(s)}{\varphi_1(s)}\biggr\}\biggr\},
	\end{equation}
	where the infimum in the first term is taken over all pairs of $x,\,y\in X$ with $0<\rho(x,y)<\xi$.
	
	Now we demonstrate that the images of a convex cone under iterations of the Ruelle--Perron--Frobenius operators contract uniformly.
	
	\begin{prop}\label{prop:invariancewrtRuelleoperator}
		Fix arbitrary constants $\eta,\,\xi,\,a_0>0,\,\lambda>1$, and $v_0\in(0,1]$. Let $(X,\rho)$ be a compact metric space, $T\colon X\rightarrow X$ be a continuous, topologically exact, open, and distance-expanding with constants $\eta$, $\lambda$, and $\xi$. Assume that $\varphi\in C^{0,v_0}(X,\rho)$ satisfies $a_0\geq\abs{\varphi}_{v_0,\rho}$.
		
		Put $\lambda_1\=\frac{1+\lambda^{-v_0}}{2}<1$ and $a'\=\frac{6a_0\lambda^{v_0}-2a_0}{(\lambda^{v_0}-1)^2}>0$. Define $\overline{\varphi}(x) \= \varphi(x)-P(T,\varphi)$ for each $x\in X$, and $Q\=\exp(\inf_{x\in X}\overline{\varphi}(x))>0$. Fix a constant $b>2\exp(a')$.
		Suppose that $\epsilon\in (0,\xi)$ is a constant with $2\exp(a'\epsilon^{v_0})\leq b$, and that $G'=\{x'_1,\,x'_2,\,\ldots,\,x'_{n'}\}$ is an $(\epsilon/2)$-net of $X$.
		Let $m\in\N$ be a constant satisfying $\bigcap_{i=1}^{n'}T^m(B(x'_i,\epsilon/2))=X.$ Put $b'\=\frac{(C^{-2}Q^m+2)b}{2C^{-2}Q^m+2}\in\bigl(\frac{b}{2},b\bigr)$, where $C>1$ is from Lemma~\ref{lemma:boundforquotient}.
		
		Denote by $L\:C(X)\rightarrow C(X)$ the operator given by
		\begin{equation}\label{e:defofL}
			L(u)(x)\=\frac{\mathcal{L}_{\overline{\varphi}}(uu_{\varphi})(x)}{u_{\varphi}(x)}=\sum_{\overline{x}\in T^{-1}(x)}\frac{u(\overline{x})u_{\varphi}(\overline{x})}{u_{\varphi}(x)}\cdot\exp(\overline{\varphi}(\overline{x}))
		\end{equation}
		for all $u\in C(X)$ and $x\in X$. Here $u_{\varphi}$ is the function from Corollary~\ref{cor:existenceofuvarphi}. 
		
		Then $L(\mathbbm{1})=\mathbbm{1}$ and
		\begin{equation}\label{e:invarianceofL}
			L^m(\cC(a',b))\subseteq\cC(\lambda_1a',b')  \subseteq \cC(a',b).
		\end{equation}		
	\end{prop}
	
	\begin{proof}
		By (\ref{e:uvarphiiseigenfunction}) and (\ref{e:defofL}), we obtain that $L(\mathbbm{1})=\mathbbm{1}$. It remains to prove (\ref{e:invarianceofL}). Since $\lambda_1<1$ and $b'<b$, it follows from Definition~\ref{def:coneavb} that $\cC(\lambda_1a',b') \subseteq \cC(a',b)$. Hence, it suffices to show that $L^m(\cC(a^{\prime},b))\subseteq\cC(\lambda_1a^{\prime},b^{\prime})$. 
		
		Now we consider an arbitrary function $u\in\cC(a',b)$ and prove that $L^m(u)$ satisfies properties~(i),~(ii), and~(iii) in Definition~\ref{def:coneavb} for $\cC(\lambda_1a',b')$.
		
		Since $u\in\cC(a',b)$, then $u(x)>0$ for each $x\in X$. By (\ref{e:defofL}), $L(u)(x)>0$ for each $x\in X$. By induction, we obtain that $L^m(u)(x)>0$ for each $x\in X$, so $L^m(u)$ satisfies property~(i) in Definition~\ref{def:coneavb} for $\cC(\lambda_1a',b')$.
		
		Next, we prove that $L(u)$ satisfies property~(ii) in Definition~\ref{def:coneavb} for $\cC(\lambda_1a',b')$, namely,
		\begin{equation*}
			L(u)(x)\leq\exp(\lambda_1a'\rho(x,y)^{v_0})\cdot L(u)(y)
		\end{equation*}
		for each pair of $x,\,y\in X$ with $\rho(x,y)<\xi$. Consider a pair of $x,\,y\in X$ with $\rho(x,y)<\xi$. Since $T$ is an open and distance-expanding with constants $\eta$, $\lambda$, and $\xi$, for each $\overline{x}\in T^{-1}(x)$, there exists a point $y(\overline{x}) \=  T_{\overline{x}}^{-1}(y)\in T^{-1}(y)$ such that $\rho(\overline{x},y(\overline{x}))\leq\lambda^{-1}\rho(x,y)<\xi.$ Hence,
		\begin{equation}\label{inequalityu}
			u(\overline{x})\leq\exp(a'\rho(\overline{x},y(\overline{x}))^{v_0})\cdot u(y(\overline{x}))\leq\exp(a'\lambda^{-v_0}\rho(x,y)^{v_0})\cdot u(y(\overline{x})).
		\end{equation}
		Since $\varphi\in C^{0,v_0}(X,\rho)$, the function $\overline{\varphi}\in C^{0,v_0}(X,\rho)$ satisfies $a_0\geq\abs{\varphi}_{v_0,\rho}$. Hence, we have that
		\begin{equation}\label{inequalityexp}
			\exp(\overline{\varphi}(\overline{x}))\leq\exp(a_0\rho(\overline{x},y(\overline{x}))^{v_0})\cdot\exp(\overline{\varphi}(y(\overline{x})))\leq\exp(a_0\lambda^{-v_0}\rho(x,y)^{v_0})\cdot\exp(\overline{\varphi}(y(\overline{x})))
		\end{equation}
		for each $\overline{x}\in T^{-1}(x)$. Moreover, by (\ref{e:holderofloguvarphi}), for each $\overline{x}\in T^{-1}(x)$,
		\begin{equation}\label{e:inequalityuvarphi}
			\frac{u_{\varphi}(\overline{x})}{u_{\varphi}(x)}\leq \exp\biggl(\frac{2a_0\rho(x,y)^{v_0}}{\lambda^{v_0}-1}\biggr)\cdot \frac{u_{\varphi}(y(\overline{x}))}{u_{\varphi}(y)}.
		\end{equation}
		
		Thus, by (\ref{e:defofL}), (\ref{inequalityu}), (\ref{inequalityexp}), and (\ref{e:inequalityuvarphi}), we conclude that
		\begin{align*}
			L(u)(x)&=\sum_{\overline{x}\in T^{-1}(x)}\frac{u(\overline{x})u_{\varphi}(\overline{x})}{u_{\varphi}(x)}\cdot\exp(\overline{\varphi}(\overline{x}))\\
			&\leq\sum_{\overline{x}\in T^{-1}(x)}\frac{u(y(\overline{x}))u_{\varphi}(y(\overline{x}))}{u_{\varphi}(y)}\cdot
			         \exp\biggl( \overline{\varphi}(y(\overline{x}))+\biggl( \frac{a'}{\lambda^{v_0}}+ \frac{a_0}{\lambda^{v_0}}+ \frac{2a_0}{\lambda^{v_0}-1} \biggr) \rho(x,y)^{v_0} \biggr)\\
			&= \exp(\lambda_1a'\rho(x,y)^{v_0})\cdot L(u)(y)
		\end{align*}
		for each pair of $x,\,y\in X$ with $\rho(x,y)<\xi$, where the last equality holds due to the definitions of $\lambda_1$ and $a'$. Hence, $L(u)$ satisfies property~(ii) in Definition~\ref{def:coneavb} for $\cC(\lambda_1a',b')$. By induction, it follows from $\lambda_1<1$ that for each $s\in\N$, $L^s(u)$ satisfies property~(ii) in Definition~\ref{def:coneavb} for $\cC(\lambda_1a',b')$.
		
		Finally, we establish that $L^m(u)$ satisfies property~(iii) in Definition~\ref{def:coneavb} for $\cC(\lambda_1a',b')$, namely,
		\begin{equation}\label{e:resultofquotient}
			L^m(u)(x)\leq b'L^m(u)(y)\quad\text{ for each pair of }x,\,y\in X.
		\end{equation}
		Note that $L^m\colon C(X)\rightarrow C(X)$ is a linear operator, we can additionally assume that $\inf_{x\in X}u(x)=1$ without loss of generality. Since $u\in\cC(a',b)$, there exists a pair of $x_{*},\,x^*\in X$ with $u(x_{*})=1$ and $u(x^*)=\sup_{x\in X}u(x)\leq b$. Depending on the value of $u(x^*)$, we have the following two cases.
		
		\smallskip
		\emph{Case 1:} $b^{\prime}\leq u(x^*)\leq b$.
		
		Consider a point $w\in X$, we shall show that $L^m(u)(w)\leq b^{\prime}$. Since $G'$ is an $(\epsilon/2)$-net of $X$, there exists an integer $i\in [1,n']$ with $x_{*}\in B(x'_i,\epsilon/2)$. Since $T^m(B(x'_i,\epsilon/2))=X$, there exists $\overline{w}_{*}\in T^{-m}(w)\cap B(x_i^{\prime},\epsilon/2)$. Hence,  $\rho(\overline{w}_{*},x_{*})\leq\rho(\overline{w}_{*},x'_i)+\rho(x'_i,x_{*})<\epsilon<\xi$. Moreover, since $u\in\cC(a',b)$, we have
		\begin{equation}\label{e:lowerboundoverlinev}
			u(\overline{w}_{*})\leq\exp(a'\rho(\overline{w}_{*},x_{*})^{v_0})  \cdot u(x_{*})\leq\exp(a'\epsilon^{v_0}).
		\end{equation}
		Hence, by the definitions of $\epsilon,\,Q,\,b^{\prime}$, $L(\mathbbm{1})=\mathbbm{1}$, (\ref{e:inequalityuphi}), (\ref{e:defofL}), and (\ref{e:lowerboundoverlinev}), we have
		\begin{align*}
			L^m(u)(w)&=\frac{u(\overline{w}_{*})u_{\varphi}(\overline{w}_{*})}{u_{\varphi}(w)}\cdot e^{S_m\overline{\varphi}(\overline{w}_{*})}
			+\sum_{\overline{w}\in T^{-m}(w)\smallsetminus\{\overline{w}_{*}\}}\frac{u(\overline{w})u_{\varphi}(\overline{w})}{u_{\varphi}(w)}\cdot e^{S_m\overline{\varphi}(\overline{w})}\\
			&\leq\frac{\exp(a'\epsilon^{v_0})\cdot u_{\varphi}(\overline{w}_{*})}{u_{\varphi}(w)}\cdot e^{S_m\overline{\varphi}(\overline{w}_{*})}
			+u(x^*) \sum_{\overline{w}\in T^{-m}(w)\smallsetminus\{\overline{w}_{*}\}}\frac{u_{\varphi}(\overline{w})}{u_{\varphi}(w)}\cdot e^{S_m\overline{\varphi}(\overline{w})} \\
			&\leq u(x^*)-(u(x^*)-\exp(a'\epsilon^{v_0}))\cdot C^{-2}\inf_{x\in X}\exp(S_m\overline{\varphi}(x))\\
			&\leq b-(b^{\prime}-b/2)C^{-2}Q^m\\
			&=b^{\prime},
		\end{align*}
		where the last equality comes from a direct calculation.
		
		\smallskip
		\emph{Case 2:} $u(x^*)<b^{\prime}$. 
		
		Since $L(\mathbbm{1})=\mathbbm{1}$, we obtain that $\sup_{x\in X}L^m(u)(x)\leq\sup_{x\in X}u(x)<b^{\prime}.$
		
		To summarize, we always have that $\sup_{x\in X}L^m(u)(x)\leq b^{\prime}$. Moreover, it follows from $L(\mathbbm{1})=\mathbbm{1}$ that $\inf_{x\in X}L^m(u)(x)\geq 1$. Therefore, we have estalished (\ref{e:resultofquotient}) for each pair of $x,\,y\in X$.
		
		Therefore, we obtain (\ref{e:invarianceofL}) and complete the proof of Proposition~\ref{prop:invariancewrtRuelleoperator}.
	\end{proof}
	
	Let $\cC_+$ be the set of positive functions on $X$. It is not hard to show that $\cC_+$ is a convex cone. Let $\alpha_+$ and $\beta_+$ be the corresponding functions associated to $\cC_+$ defined in (\ref{e:defofalphabeta}), and $\theta_+$ be the projective metric associated to $\cC_+$ defined in (\ref{e:defoftheta}). Now, we establish the following result. Recall $\hK$ from (\ref{e:hK}).
	
	\begin{prop}\label{prop:finitediameter2}
		Fix constants $a>0$, $b>1$, and $0<\lambda_1<1$. Then
		\begin{equation}\label{e:computationandinequalitytheta+}
			\theta_+(\varphi_1,\varphi_2)=\log(\sup\{\varphi_2(x)/\varphi_1(x)\scolon x\in X\})-\log(\inf\{\varphi_2(y)/\varphi_1(y)\scolon y\in X\})\leq\theta_{a,b}(\varphi_1,\varphi_2)
		\end{equation} 
		for each pair of $\varphi_1,\,\varphi_2\in\cC(a,b)$, and 
		\begin{equation}\label{e:theta-diameterisbounded}
			\diam_{\theta_{a,b}} (\cC(\lambda_1a,b'))
			\leq \hK(\lambda_1,b,b').
		\end{equation}
	\end{prop}
	
	\begin{proof}
		Consider a pair of $\varphi_1,\,\varphi_2\in\cC(a,b)$. By the definition of $\alpha_+$ and $\beta_+$, 
		we have that $\alpha_+(\varphi_1,\varphi_2)=\inf_{t\in X}(\varphi_2(t)/\varphi_1(t))$, and $\beta_+(\varphi_1,\varphi_2)=\sup_{s\in X}(\varphi_2(s)/\varphi_1(s))$. 
		Hence, 
        \begin{equation*}
            \theta_+(\varphi_1,\varphi_2)=\log(\sup\{\varphi_2(x)/\varphi_1(x)\scolon x\in X\})-\log(\inf\{\varphi_2(y)/\varphi_1(y)\scolon y\in X\}).
        \end{equation*}
        By (\ref{e:expressionofalpha}), we have $\alpha_+(\varphi_1,\varphi_2)\geq\alpha_{a,b}(\varphi_1,\varphi_2)$. Similarly, we have $\beta_+(\varphi_1,\varphi_2)\leq\beta_{a,b}(\varphi_1,\varphi_2)$. Hence, $\theta_+(\varphi_1,\varphi_2)\leq\theta_{a,b}(\varphi_1,\varphi_2)$.
		
		Consider a pair of $\phi_1,\,\phi_2\in\cC(\lambda_1a,b')\subseteq\cC_+$. By the definition of $\cC(\lambda_1a,b')$, we have that
		\begin{equation*}
			\phi_2(y)\leq\phi_2(x)\exp(a\lambda_1\rho(x,y)^{v_0})\quad\text{ and }\quad \phi_1(y)\geq\phi_1(x)\exp(-a\lambda_1\rho(x,y)^{v_0})
		\end{equation*}
		for each pair of $x,\,y\in X$ with $\rho(x,y)<\xi$. Hence, we obtain that
		\begin{align}
			\frac{\exp(a\rho(x,y)^{v_0})\phi_2(x)-\phi_2(y)}{\exp(a\rho(x,y)^{v_0})\phi_1(x)-\phi_1(y)}
			&\geq\frac{\phi_2(x)}{\phi_1(x)}\cdot\frac{\exp(a\rho(x,y)^{v_0})-\exp(a\lambda_1\rho(x,y)^{v_0})}{\exp(a\rho(x,y)^{v_0})-\exp(-a\lambda_1\rho(x,y)^{v_0})}\nonumber\\
			&\geq \inf\limits_{z>1}\left\{\frac{z-z^{\lambda_1}}{z-z^{-\lambda_1}}\right\}\cdot\frac{\phi_2(x)}{\phi_1(x)}\label{e:inequalitythetaav}
			\geq\frac{1-\lambda_1}{1+\lambda_1}\cdot\frac{\phi_2(x)}{\phi_1(x)}
		\end{align}
		for each pair of $x,\,y\in X$ with $\rho(x,y)<\xi$.
		By Definition~\ref{def:coneavb}, we have that
		\begin{equation*}
			\phi_2(y)\leq b'\phi_2(x)<b\phi_2(x)\quad\text{ and }\quad 0<\phi_1(y)<b\phi_1(x)
		\end{equation*}
		for each pair of $x,\,y\in X$. Hence,
		\begin{equation}\label{e:equation:inequalitythetaavb}
			\frac{b\phi_2(x)-\phi_2(y)}{b\phi_1(x)-\phi_1(y)}\geq\frac{b-b'}{b}\cdot\frac{\phi_2(x)}{\phi_1(x)}
		\end{equation}
		for each pair of $x,\,y\in X$.
		
		By (\ref{e:expressionofalpha}), (\ref{e:inequalitythetaav}), and (\ref{e:equation:inequalitythetaavb}), we obtain that $\alpha_{a,b}(\phi_1,\phi_2)\geq\frac{1-\lambda_1}{1+\lambda_1}\cdot\frac{b-b'}{b}\cdot\alpha_+(\phi_1,\phi_2)$. Similarly, we can obtain that $\beta_{a,b}(\phi_1,\phi_2)\leq\frac{1+\lambda_1}{1-\lambda_1}\cdot\frac{b}{b-b'}\cdot\beta_+(\phi_1,\phi_2)$. Moreover, since $\phi_1,\,\phi_2\in\cC(\lambda_1a,b^{\prime})$, we have $\frac{1}{b}\leq\frac{\phi_2(x)}{\phi_1(x)}\leq b$ for each $x\in X$. Hence, it follows from (\ref{e:hK}) that 
		\begin{equation*}
			\begin{aligned}
				\theta_{a,b}(\phi_1,\phi_2)
				&=\log(\beta_{a,b}(\phi_1,\phi_2))-\log(\alpha_{a,b}(\phi_1,\phi_2))\\
				&\leq\log(\beta_{+}(\phi_1,\phi_2))-\log(\alpha_{+}(\phi_1,\phi_2))+2\log\biggl(\frac{1+\lambda_1}{1-\lambda_1}\cdot\frac{b}{b-b'}\biggr)\leq\hK (\lambda_1,b,b')
			\end{aligned}
		\end{equation*}
		for each pair of $\phi_1,\,\phi_2\in\cC(\lambda_1 a,b')$. This completes the proof of (\ref{e:theta-diameterisbounded}).
	\end{proof}
	
	We are now ready to establish Theorem~\ref{t:eigenfunction}.
	
	\begin{proof}[Proof of Theorem~\ref{t:eigenfunction}]
		Recall that  $\lambda_1\=\frac{1+\lambda^{-v_0}}{2},\,a'\=\frac{6a_0\lambda^{v_0}-2a_0}{(\lambda^{v_0}-1)^2}$, $C'>\max\{C^2,\, 2\exp(a')\}$, where $C$ is from Lemma~\ref{lemma:boundforquotient}. Write $D \=\frac{(C^{-2}Q^m+2)C'}{2C^{-2}Q^m+2}$.  
		Then by Proposition~\ref{prop:invariancewrtRuelleoperator}, we have
		$L^m(\cC(a',C'))\subseteq\cC(\lambda_1a', D )$,
		where $L$ is defined in (\ref{e:defofL}). 
		Thus, it follows from Proposition~\ref{prop:finitediameter2} that
		\begin{equation}\label{e:diameterconcrete}
				\diam_{\theta_{a',C'}}(L^m(\cC(a',C')))
				\leq\diam_{\theta_{a',C'}} (\cC (\lambda_1a',D ) )
				\leq\hK(\lambda_1,C',D).
		\end{equation}
		Obviously, we have $\mathbbm{1}\in\cC(a',C')$. By (\ref{e:holderofloguvarphi}) and (\ref{e:inequalityuphi}), it follows from $a'>\frac{a_0}{\lambda^{v_0}-1}$ and $C'>C^2$ that $1/u_{\varphi}\in\cC(a',C')$. 
	 	Then by Proposition~\ref{prop:Liscontraction} for $L^m\:\cC(a',C')\rightarrow\cC(a',C')$ and (\ref{e:diameterconcrete}),
		we have
		\begin{equation}\label{e:iterationofestimation}
				\theta_{a',C'}\bigl(L^{mk}(1/u_{\varphi}),\mathbbm{1}\bigr)
				\leq (1-\exp(-Z))^{k-1}\cdot\theta_{a',C'} (L^m(1/u_{\varphi}),\mathbbm{1} ) 
				\leq (1-\exp(-Z))^{k-1}Z
		\end{equation}
		for each $k\in\N$, where $Z\=\hK(\lambda_1,C',D)$. By Proposition~\ref{prop:finitediameter2}, we have \begin{equation}\label{e:inequatlityoftheta+thetaa}
			\theta_+\bigl(L^{mk}(1/u_{\varphi}),\mathbbm{1}\bigr)\leq\theta_{a',C'}\bigl(L^{mk}(1/u_{\varphi}),\mathbbm{1}\bigr).
		\end{equation}
		Then by (\ref{e:defofL}), (\ref{e:iterationofestimation}), and (\ref{e:inequatlityoftheta+thetaa}), one can see that
		\begin{equation}\label{e:equationtheta+esti}
			\theta_+\bigl(\mathcal{L}^{mk}_{\overline{\varphi}}(\mathbbm{1}),u_{\varphi}\bigr)
            =\theta_+\bigl(\mathcal{L}^{mk}_{\overline{\varphi}}(\mathbbm{1})/u_{\varphi},\mathbbm{1}\bigr)
            =\theta_+\bigl(L^{mk}(1/u_{\varphi}),\mathbbm{1}\bigr)
            \leq Z(1-\exp(-Z))^{k-1}.
		\end{equation}
		
		Since $L(\mathbbm{1})=\mathbbm{1}$, one can see that 
		\begin{equation*}
			\sup_{x\in X}\frac{\mathcal{L}^{mk}_{\overline{\varphi}}(\mathbbm{1})(x)}{u_{\varphi}(x)}
            =\sup_{x\in X}  \bigl\{L^{mk}(1/{u_{\varphi}(x)}) \bigr\}
            \geq 1 \geq\inf_{x\in X} \bigl\{ L^{mk}(1/{u_{\varphi}(x)}) \bigr\} 
            = \inf_{x\in X}\frac{\mathcal{L}^{mk}_{\overline{\varphi}}(\mathbbm{1})(x)}{u_{\varphi}(x)}.
		\end{equation*} 
		Hence, by (\ref{e:inequalityuphi}),  (\ref{e:computationandinequalitytheta+}), and (\ref{e:equationtheta+esti}), we have
		\begin{equation*}
		\begin{aligned}
			\Normbig{\mathcal{L}^{mk}_{\overline{\varphi}}(\mathbbm{1})(x)-u_{\varphi}(x)}_{\infty}&\leq\Normbig{u_{\varphi}(x)}_{\infty}\cdot \Normbigg{\frac{\mathcal{L}^{mk}_{\overline{\varphi}}(\mathbbm{1})(x)}{u_{\varphi}(x)}-1}_{\infty}\\
			&\leq C\bigl(\exp\bigl(\theta_+\bigl(\cL^{mk}_{\overline{\varphi}}(\mathbbm{1}),u_{\varphi}\bigr)\bigr)-1\bigr)\\
			&\leq   C\bigl(\exp\bigl(Z(1-\exp(-Z))^{k-1}\bigr)-1\bigr).
		\end{aligned}
		\end{equation*}
		Since for each $x\in [0,1]$, $e^x\leq ex+1,$ we have
		\begin{equation*}
			\Normbig{\mathcal{L}^{mk}_{\overline{\varphi}}(\mathbbm{1})(x)-u_{\varphi}(x)}_{\infty}\leq CeZ(1-\exp(-Z))^{k-1}
		\end{equation*}
		for each $k\in\N$ with $0\leq Z(1-\exp(-Z))^{k-1}\leq 1$. Therefore, we complete the proof.
	\end{proof}
	
	\subsection{Proof of the computability of the Jacobian}
	\label{subsec:computability of Jacobian}
	
	In this subsection, we establish the uniform computability of Jacobians, as stated in the following theorem.
	
	\begin{theorem}
		\label{Computability of Jacobian}
		There exists an algorithm with the following property:
		\smallskip
		
		For each $x_0\in X$, each $n\in\N$, and each quintet $(X,\,\rho,\,\mathcal{S},\,\varphi,\,T)$ satisfying the Assumption~A in Section~\ref{sec:intro}, the algorithm outputs a rational $2^{-n}$-approximation for the value of $J_{\varphi}(x_0)$, where the function $J_{\varphi}\colon X\rightarrow\R$ is defined by
		\begin{equation}\label{expressionofJ}
			J_{\varphi}(x) \= \frac{u_{\varphi}(T(x))}{u_{\varphi}(x)}\exp(P(T,\varphi)-\varphi(x)) \quad\text{ for each }x \in X,
		\end{equation}
		on the following input: 
		\begin{enumerate}
			\item[(i)] two algorithms computing $\varphi$ and $T$, respectively,
			\smallskip
			\item[(ii)] an algorithm outputting a net of $X$ with any given precision (in the sense of  Definition~\ref{definition of recursively precompactness}),
			\smallskip
			\item[(iii)] two rational constants $a_0>0$ and $v_0\in(0,1]$ with $\varphi\in C^{0,v_0}(X)$ and $a_0\geq\abs{\varphi}_{v_0,\rho}$,
			\smallskip
			\item[(iv)] three rational constants $\eta>0$, $\lambda>1$, and $\xi>0$ satisfying that $T$ is distance-expanding with these constants (in the sense of  Definition~\ref{defndistanceexpanding}),
			\smallskip
			\item[(v)] an oracle of the point $x_0\in X$,
			\smallskip
			\item[(vi)] the constant $n\in\N$.
		\end{enumerate}
	\end{theorem}
	
	The proof will be given at the end of this subsection. We begin with designing an algorithm that computes the function $\mathcal{L}_{\varphi}(\mathbbm{1})$.
	
	\begin{prop}\label{computabilityofinverse}
		Let $(X,\,\rho,\,\mathcal{S})$ be a computable metric space, $X$ a recursively compact set, and $T\:X\rightarrow X$ a computable distance-expanding map with respect to the metric $\rho$ with constants $\eta$, $\lambda$, and $\xi$. Then the inverse $T^{-1}$ of $T$ is computable (in the sense of  Definition~\ref{defn:computableinverse}).
	\end{prop}
	
	\begin{proof}
		For each $p\in X$, denote by $\oB(p,r)\=\{q\in X\scolon \rho(p,q)\leq r\}$ the closed ball of radius $r$. Consider a point $x\in X$. By \cite[Proposition~4]{GHR11}, we obtain that $X$ is recursively precompact, and consequently, there is an algorithm outputting an $\eta$-net of $X$, say $\{x_i\scolon i\in[1,n]\cap\N\}$. Because $T$ is distance-expanding with constants $\eta$, $\lambda$, and $\xi$, by (\ref{defnOfDistance-expansion}), there is at most one preimage point of $x$ in each closed ball with radius $\eta$. By Definition~\ref{definition of recursively compact}, $\overline{B}(x_i,\eta)$ is recursively compact for each integer $1\leq i\leq n$. Hence, by Proposition~\ref{prop:thecomplementofrecursivelycompactset}~(iii), the set $B_i=\overline{B}(x_i,\eta)\smallsetminus\bigcup_{j=1}^{i-1}B(x_j,\eta)$ is recursively compact for each integer $1\leq i\leq n$. Thus by \cite[Theorem~2.12]{BRY12}, the inverse $\bigl(T|_{B_i}\bigr)^{-1}\:T(B_i)\rightarrow B_i$ of the restriction of $T$ is computable for each $i\in\N$. Therefore, the inverse $T^{-1}$ of $T$ is computable (in the sense of  Definition~\ref{defn:computableinverse}).
	\end{proof}
	
	As an immediate consequence of Proposition~\ref{computabilityofinverse} and the computability of the exponential function, one gets the computability of the Ruelle--Perron--Frobenius operator in the following sense:
	
\begin{cor}\label{Computability of Ruelle operators}
	There exists an algorithm with the following property:
	
	\smallskip
	
	For each $x_0\in X$, each pair of $n,\,m\in\N$, and each quintet $(X,\,\rho,\,\mathcal{S},\,\varphi,\,T)$ satisfying the Assumption~A in Section~\ref{sec:intro}, this algorithm outputs a rational $2^{-n}$-approximation for the value of $\mathcal{L}^m_{\varphi}(\mathbbm{1})(x_0)$, after inputting an oracle of the point $x_0$, the constants $n,\,m\in\N$, and data~(i)~through~(iv) from Theorem~\ref{Computability of Jacobian} in this algorithm.
\end{cor}
	
	Now we apply Lemma~\ref{lemma:boundforquotient} to establish the computability of the value of the pressure $P(T,\varphi)$.
	
\begin{lemma}\label{Computability of Topology Pressure}
	There exists an algorithm with the following property:
	
	\smallskip
	
	For each $n\in\N$ and each quintet $(X,\,\rho,\,\mathcal{S},\,\varphi,\,T)$ satisfying the Assumption~A in Section~\ref{sec:intro}, this algorithm outputs a rational $2^{-n}$-approximation for the topological pressure $P(T,\varphi)$, after inputting the constant $n\in\N$ and data~(i)~through~(iv) from Theorem~\ref{Computability of Jacobian} in this algorithm.
\end{lemma}
	
	\begin{proof}
		We can design the algorithm following the steps below:
		\begin{enumerate}
			\smallskip
			\item[(1)] Compute a $\xi'$-net $G$ for the space $X$, where $\xi'\=\min\{\eta,\,\xi\}$. Then compute $N\in\N$ with $\bigcap_{x\in G}\bigcup_{k=0}^NT^k(x,\xi)=X$.
			\smallskip
			\item[(2)] Use \cite[Proposition~2.13]{BRY12} to compute $\Norm{\varphi}_{\infty}$.
			\smallskip
			\item[(3)] Compute $N_1\in\N$ with $2^{n+1}N\log(\card G)\cdot\bigl(\frac{4a_0\xi^{v_0}}{\lambda^{v_0}-1}+2N\norm{\varphi}_{\infty}\bigr)\leq N_1$.
			\smallskip
			\item[(4)] By Corollary~\ref{Computability of Ruelle operators}, we compute and output the value of
			\begin{equation*}
				v_n\approx w_n\=N_1^{-1}\log\bigl(\mathcal{L}_{\varphi}^{N_1}(\mathbbm{1})(y_0)\bigr)
			\end{equation*}
			with precision $2^{-n-1}$, where $y_0\in\mathcal{S}$ is an ideal point.
		\end{enumerate}
		Let us verify that $v_n$ satisfies $\abs{ v_n-P(T,\varphi) } <2^{-n}$ for each $n\in\N$. To see this, it suffices to check that $\abs{w_n-P(T,\varphi)}<2^{-n-1}$ for each $n\in\N$. Since $T\:X\rightarrow X$ is distance-expanding with constants $\eta$, $\lambda,$ and $\xi$, then $T$ is injective on each balls of radius $\xi'$. Hence, we have $D\=\max_{x\in X}\card\bigl(T^{-1}(x)\bigr)\leq\card G$. Moreover, by Lemma~\ref{lemma:boundforquotient}, and steps~(4) and~(5) above, we can conclude that
		\begin{equation*}
			\abs{ w_n-P(T,\varphi) }=\Absbig{ N_1^{-1}\log\bigl(e^{-{N_1}P(T,\varphi)}\mathcal{L}_{\varphi}^{N_1}(\mathbbm{1})(y_0)\bigr) } < (\log C)/N_1<2^{-n-1}.  \qedhere
		\end{equation*}
	\end{proof}
	
	Applying Theorem~\ref{t:eigenfunction}, we can show that the function $u_{\varphi}\colon X\rightarrow\R$ is computable.
	
\begin{lemma}\label{Computability of u}
	There exists an algorithm with the following property:
	
	\smallskip
	
	For each $x_0\in X$, each $n\in\N$, and each quintet $(X,\,\rho,\,\mathcal{S},\,\varphi,\,T)$ satisfying the Assumption~A in Section~\ref{sec:intro}, this algorithm outputs a rational $2^{-n}$-approximation for the value of $u_{\varphi}(x_0)$, after inputting an oracle of the point $x_0$, the constant $n\in\N$, and data~(i)~through~(iv) from Theorem~\ref{Computability of Jacobian} in this algorithm.
\end{lemma}
	
	\begin{proof}
		We can design the algorithm following the steps below:
		\begin{enumerate}
			\smallskip
			\item[(1)] Compute a $\xi'$-net $G$ for the space $X$, where $\xi'\=\min\{\eta,\,\xi\}$. Then compute a constant $N\in\N$ satisfying $\bigcap_{x\in G}\bigcup_{k=0}^NT^k(x,\xi)=X.$
			\smallskip
			\item[(2)] Use Lemma~\ref{Computability of Topology Pressure} to compute the topological pressure $P(T,\varphi)$.
			\smallskip
			\item[(3)] Use \cite[Proposition~2.13]{BRY12} to compute $\Norm{\varphi}_{\infty}$ and $Q\=\inf_{x\in X}\{\varphi(x)\}-P(T,\varphi)$.
			\smallskip
			\item[(4)] Use Theorem~\ref{t:eigenfunction} to compute a constant $C' \in \R$ satisfying $C'>\max\bigl\{\overline{C}^2,\,2\exp(a')\bigr\}$, where 
			\begin{equation}	\label{eq:defofcbar}
				\overline{C}\=(\card G)^N\exp\Bigl(\frac{4a_0\xi^{v_0}}{\lambda^{v_0}-1}+2N\norm{\varphi}_{\infty}\Bigr).
			\end{equation}
			\item[(5)] Use Theorem~\ref{t:eigenfunction} to compute $\epsilon\in(0,\xi)$ with $2\exp(a'\epsilon^{v_0})<C'.$ Then compute an $(\epsilon/2)$-net $G'$ for the space $X$, and a constant $m\in\N$ satisfying $\bigcap_{t\in G'}T^m(t,\epsilon/2)=X$.
			\smallskip
			\item[(6)] Use (\ref{e:hK}) to compute $\overline{Z}\=\hK\Bigl(\frac{\lambda^{-v_0}+1}{2},C',\frac{(\overline{C}^{-2}Q^m+2)C'}{2\overline{C}^{-2}Q^m+2}\Bigr)$.
			\smallskip
			\item[(7)] Find $k\in\N$ satisfying			\begin{equation}\label{e:4.10.3}
				0< \overline{Z}(1-\exp(-\overline{Z}))^{k-1}\leq 1\quad\text{ and }\quad \overline{C}e\overline{Z}(1-\exp(-\overline{Z}))^{k-1}<2^{-n-1}.
			\end{equation}
			\item[(8)] By Proposition~\ref{computabilityofinverse} and Lemma~\ref{Computability of Topology Pressure}, we compute and output the value of
			\begin{equation*}
				s_n\approx t_n\=e^{-mkP(T,\varphi)}\mathcal{L}_{\varphi}^{mk}(\mathbbm{1})(x_0)
			\end{equation*}
			with precision $2^{-n-1}$.
		\end{enumerate}
		Fix an integer $n$. Now we verify that $s_n$ satisfies $\abs{s_n-u_{\varphi}(x_0)}<2^{-n}$. To see this, it suffices to check that $\abs{t_n-u_{\varphi}(x_0)}< 2^{-n-1}$. Since $T\:X\rightarrow X$ is distance-expanding with constants $\eta$, $\lambda,$ and $\xi$, $T$ is injective on each ball of radius $\xi'$. Hence, we have that $D\=\max_{x\in X}\card\bigl(T^{-1}(x)\bigr)\leq\card G$. 
		Recall that in Lemma~\ref{lemma:boundforquotient} we define 
        \begin{equation*}
            C\=(\card G)^N\exp\biggl(\frac{4a_0\xi^{v_0}}{\lambda^{v_0}-1}+2N\norm{\varphi}_{\infty}\biggr).
        \end{equation*}
        So by (\ref{eq:defofcbar}), $\overline{C}\geq C$. Thus $C'>\max\bigl\{\overline{C}^2,\,2\exp(a')\bigr\}\geq\max\bigl\{C^2,\,2\exp(a')\bigr\}$. Moreover, by (\ref{e:hK}) and Step~(6) above, we obtain that 
        \begin{equation*}
        \overline{Z}\geq Z
        \=\hK\biggl(\frac{\lambda^{-v_0}+1}{2},C',\frac{(C^{-2}Q^m+2)C'}{2C^{-2}Q^m+2}\biggr).
        \end{equation*}
        By Theorem~\ref{t:eigenfunction}, it follows from $0<Z(1-\exp(-Z))^{k-1}\leq\overline{Z}(1-\exp(-\overline{Z}))^{k-1}\leq 1$ that
		\begin{equation*}
				      \abs{t_n-u_{\varphi}(x_0)}
						\leq\Normbig{\cL_{\overline{\varphi}}^{mk}(\mathbbm{1})(x_0)-u_{\varphi}(x_0)}_{\infty}
						\leq CeZ \bigl( 1-e^{-Z} \bigr)^{k-1}
						\leq\overline{C}e\overline{Z} \bigl( 1-e^{-\overline{Z}} \bigr)^{k-1}
						<2^{-n-1},
		\end{equation*}
		establishing the lemma.
	\end{proof}
	
	We are now ready to establish Theorem~\ref{Computability of Jacobian}.
	
	\begin{proof}[Proof of Theorem~\ref{Computability of Jacobian}]
		By Corollary~\ref{Computability of Ruelle operators}, Lemmas~\ref{Computability of Topology Pressure}, and~\ref{Computability of u}, it follows from (\ref{expressionofJ}) that the function $J_{\varphi}$ is computable.
	\end{proof}
	
	\subsection{Two subsets of $\PPP(X)$}\label{subsec:lemmalower-computable}
	
	In this subsection, we consider two subsets of $\PPP(X)$ which are closely related to the equilibrium state $\mu_{\varphi}$ and prove that they are both lower-computable open sets.
	
	\begin{lemma}\label{recursivelycompactofinvariantmeasure}
		Let $(X,\,\rho,\,\mathcal{S})$ be a computable metric space. Assume that $X$ is recursively compact, and $T\:X\rightarrow X$ is computable. Then $U\coloneqq\cP(X)-\cM(X,T)$ is a lower-computable open set.
	\end{lemma}
	
	\begin{proof}
		Define the function $L_T\:\cP(X)\rightarrow\cP(X)$ by duality in the following way: for each $\mu\in\cP(X),\,L_T(\mu)$ satisfies that $\int\!f\,\mathrm{d}L_T(\mu)=\int\!(f\circ T)\,\mathrm{d}\mu$ for each $f\in C(X)$. Since the map $T\:X\rightarrow X$ is computable, by \cite[Theorem~3.1]{GHR11}, the map $L_T\:\cP(X)\rightarrow\cP(X)$ is also computable. Define $\cI\:\cP(X)\rightarrow\R$ by $\cI(\mu)=W_{\rho}(\mu,\,L_T(\mu))$ for each $\mu\in\cP(X)$. Then $\cI$ is a computable function. By Definition~\ref{uniformlylowercomputableopen}, $\R\smallsetminus\{0\}$ is a lower-computable open set. Hence, by Proposition~\ref{definition of computable function}, $\cI^{-1}(\R\smallsetminus\{0\})$ is a lower-computable open set. By the definition of $U$, $U=\cI^{-1}(\R\smallsetminus\{0\})$. Therefore, $U$ is a lower-computable open set.
	\end{proof}

	\begin{lemma}\label{lemma:lower-computable}
		Let the quintet $(X,\,\rho,\,\mathcal{S},\,\varphi,\,T)$ satisfy the Assumption~A in Section~\ref{sec:intro}. Define the set $W$ so that a Borel probability measure $\mu$ is in $W$ if and only if $J_{\varphi}$ is not the Jacobian of $T\: X\rightarrow X$ with respect to $\mu$, where
		\begin{equation*}
			J_{\varphi}(x) \= \frac{u_{\varphi}(T(x))}{u_{\varphi}(x)}\exp(P(T,\varphi)-\varphi(x)).
		\end{equation*}
		Then $W$ is a lower-computable open set.
	\end{lemma}
	
	\begin{proof}
		Let $\{\varphi_j\}_{j\in\N}$ be a computable enumeration of $\mathfrak{E}(\mathcal{S})$ (see Definition~\ref{def:testfunction}) in the computable metric space $(X,\,\rho,\,\mathcal{S})$, and $\{x_k\scolon k\in[1,m]\cap\N\}$ be an $\eta$-net for $X$. Set $g_{k,n}(x)\=g_{x_k,\eta-\frac{1}{n},\frac{1}{n}}(x)$ (see Definition~\ref{def:hatfunction}) for all $k,\,n\in\N$ with $k\leq m$.
		
		\smallskip
		
		\textit{Claim.} $W=\bigcup_{k=1}^m\bigcup_{j,\,n\in\N}\Phi_j^{k,n}$, where
		\begin{equation}\label{def:phijk}
			\Phi_j^{k,n} \= \left\{\mu\in \PPP(X)\scolon\int\!(\varphi_j\cdot g_{k,n})\circ T^{-1}\,\mathrm{d}\mu-\int\!\varphi_j\cdot g_{k,n}\cdot J_{\varphi}\,\mathrm{d}\mu\neq 0\right\}
		\end{equation}
		for all $j,\,k,\,n\in\N$ with $k\leq m$.
		
		\smallskip
		
		By the definitions of $W$ and the function $J_{\varphi}$, we have $\Phi_{j}^{k,n}\subseteq W$ for all $j,\,k,\,n\in\N$ with $k\leq m$. Hence, to prove the claim, it suffices to show that  $W\subseteq\bigcup_{k=1}^m\bigcup_{j,\,n\in\N}\Phi_j^{k,n}$. Fix an integer $1\leq k\leq m$. Since $T$ is distance-expanding with constants $\eta$, $\lambda,$ and $\xi$, it is injective on $B(x_k,\eta)$. Thus, the function $(\varphi_j\cdot g_{k,n})\circ T^{-1}$ is well-defined in $X$ for all $j,\,n\in\N$.
		
		Consider
		\begin{equation*}
			\mu\in\bigcap\limits_{j,\,n\in\N}\bigl(\Phi^{k,n}_j\bigr)^c.
		\end{equation*}
		Then for each $j\in\N$ and each $n\in\N$, by (\ref{def:phijk}), we have 
		\begin{equation*}
			\int\!(\varphi_j\cdot g_{k,n})\circ T^{-1}\,\mathrm{d}\mu=\int\!\varphi_j\cdot g_{k,n}\cdot J_{\varphi}\,\mathrm{d}\mu.
		\end{equation*}
		Fix $j\in\N$ and let $n\to+\infty$. It follows from the Dominated Convergence Theorem that
		\begin{equation*}
			\int\! \bigl(\varphi_j\cdot\mathbbm{1}_{B(x_k,\eta)} \bigr)\circ T^{-1}\,\mathrm{d}\mu
			=\int\!\varphi_j\cdot\mathbbm{1}_{B(x_k,\eta)}\cdot J_{\varphi}\,\mathrm{d}\mu
		\end{equation*}
		for each $j\in\N$. Then by Proposition~\ref{test function}, we have $\mu_1=\mu_2$, where $\mu_1$ and $\mu_2$ are given by
		\begin{equation*}
			\mu_1(A) \= \frac{\mu(T(A\cap B(x_k,\eta)))}{\mu(T( B(x_k,\eta)))}\quad\text{ and }\quad\mu_2(A) \= \frac{\int_{A\cap B(x_k,\eta)}\!J_{\varphi}\,\mathrm{d}\mu}{\int_{ B(x_k,\eta)}\!J_{\varphi}\,\mathrm{d}\mu}
		\end{equation*}
		for each Borel set $A\subseteq X$. Thus,
          $\mu(T(A\cap  B(x_k,\eta)))=\int_{A\cap  B(x_k,\eta)}\!J_{\varphi}\,\mathrm{d}\mu$
		for each Borel set $A\subseteq X$.
		
		Hence, for each $\mu\in\bigcap\limits_{1\leq k\leq m}\bigcap\limits_{j,\,n\in\N}\bigl(\Phi^{k,n}_j\bigr)^c$ and each Borel set $A\subseteq X$, we can define $A_1\coloneqq A\cap B(x_1,\eta)$, and $A_k\coloneqq \bigl(A-\bigcup_{i=1}^{k-1}A_i\bigr)\cap B(x_k,\eta)$ for each integer $k\in[2,m]$. Since $X=\bigcup_{k=1}^mB(x_k,\eta)$, then $A=\bigsqcup_{k=1}^mA_k$. So
		$$\mu(T(A))=\sum\limits_{k=1}^m\mu(T(A_k))=\sum\limits_{k=1}^m\int_{A_k}\!J_{\varphi}\,\mathrm{d}\mu=\int_A \!J_{\varphi}\,\mathrm{d}\mu,$$
		namely, $\mu\in W^c.$ Hence, the claim holds.
		
		\smallskip
		
		By Definition~\ref{def:hatfunction}, Theorem~\ref{Computability of Jacobian}, and Proposition~\ref{computabilityofinverse}, 
        \begin{equation*}
            \bigl\{(\varphi_j\cdot g_{k,n})\circ T^{-1}-\varphi_j\cdot g_{k,n}\cdot J_{\varphi}: j,\,k,\,n\in\N\text{ with }k\leq m\bigr\}
        \end{equation*}
        is a uniformly computable sequence of functions. Then by Corollary~\ref{computabilityofintegraloperator}, the family $\mathcal{I}_j^{k,n}\: \PPP(X)\rightarrow\R,\,j,\,k,\,n\in\N\text{ with }k\leq m,$ defined by
		\begin{equation*}
			\mathcal{I}_j^{k,n}(\mu) \= \int\!\bigl( (\varphi_j\cdot g_{k,n})\circ T^{-1}-\varphi_j\cdot g_{k,n}\cdot J_{\varphi}\bigr)\,\mathrm{d}\mu\quad\text{ for all }\mu\in\cP(X)\text{ and }j,\,k,\,n\in\N\text{ with }k\leq m,
		\end{equation*}
		is a uniformly computable sequence of functions.  By Definition~\ref{uppercomputableclosed}, $\R\smallsetminus\{0\}$ is a lower-computable open set. Note that $\Phi_j^{k,n}=\bigl(\cI_j^{k,n}\bigr)^{-1}(\R\smallsetminus\{0\})$. Then by Corollary~\ref{cor:uniformlycomputablefunction}, $\Phi_j^{k,n},\,j,\,k,\,n\in\N\text{ with }k\leq m$ form a uniformly computable sequence of lower-computable open sets. Therefore, by Proposition~\ref{lower-computable open} and the claim, $W$ is a lower-computable open set.
	\end{proof}
	
	\subsection{Proofs of Theorems~\ref{maintheorem} and~\ref{maintheorem'}}
	\label{subsec:computability of equilibriumstates}

	Now we prove Theorem~\ref{maintheorem}.
	
	\begin{proof}[Proof of Theorem~\ref{maintheorem}]
		
		Since $X$ is recursively compact, by Proposition~\ref{recursively compactness of measure space}, $\PPP(X)$ is also recursively compact. By Corollary~\ref{cor:existenceofuvarphi}, there exists a unique $T$-invariant Gibbs state $\mu_{\varphi}$ for $T$ and $\varphi$. By Proposition~\ref{Jacobian}, $\mu_{\varphi}$ is also a unique equilibrium state $\mu_{\varphi}$ for $T$ and $\varphi$. Together with the definitions of $U$ and $W$, we have $\{\mu_{\varphi}\}=\PPP(X)\smallsetminus (U\cup W)$. By Lemmas~\ref{recursivelycompactofinvariantmeasure} and~\ref{lemma:lower-computable}, $U$ and $W$ are both lower-computable open sets. It follows from Proposition~\ref{prop:thecomplementofrecursivelycompactset}~(iii) that $\{\mu_{\varphi}\}$ is recursively compact. Therefore, by Proposition~\ref{prop:thecomplementofrecursivelycompactset}~(i), the  measure $\mu_{\varphi}$ is computable.
	\end{proof}
	
	Next, we apply Theorem~\ref{maintheorem} to establish Theorem~\ref{maintheorem'} as follows.
	
	\begin{proof}[Proof of Theorem~\ref{maintheorem'}]
		
		In the computable metric space $(D,\,\rho^{\prime},\,\mathcal{S}')$, by Theorem~\ref{maintheorem}, there exists an algorithm which, on input $n\in\N$, outputs finite sequences $\{k_i\}_{i=1}^m$ in $\Q^+$ and $\{x_i\}_{i=1}^m$ in $\mathcal{S}'$ satisfying $\sum_{i=1}^mk_i=1$, and
		$W_{\rho^{\prime}}(\mu_{\varphi},\mu_n)\leq 2^{-n},$ where $\mu_n=\sum_{i=1}^mk_i\delta_{x_i}\in\mathcal{R}_{\mathcal{S}'}$. Since $\rho^{\prime}$ is the restriction of the metric $\rho$, we have $W_{\rho}(\mu_{\varphi},\mu_n)\leq 2^{-n}$. Note that $\mathcal{S}'$ is a uniformly computable sequence of points in the computable metric space $(X,\,\rho,\,\mathcal{S})$. Then for each integer $i\in[1,m]$, we can find a point $y_i\in\mathcal{S}$ such that $\rho\bigl(x_i,y_i\bigr)<2^{-n}$.   Define $\hmu_n  \= \sum_{i=1}^mk_i\delta_{y_i}\in\mathcal{R}_{\mathcal{S}}.$ Thus we have $W_{\rho}(\mu_n,\hmu_n)<2^{-n}$. Then
		\begin{equation*}
			W_{\rho}(\mu_{\varphi},\hmu_n)\leq W_{\rho}(\mu_{\varphi},\mu_n)+W_{\rho}(\mu_n,\hmu_n)<2^{1-n}.
		\end{equation*}
		Hence, $\mu_{\varphi}$ is a computable point in the computable metric space $(\PPP(X),\,W_{\rho},\,\mathcal{R}_{\mathcal{S}})$.
	\end{proof}

	\section{Hyperbolic rational maps}\label{sec:exampleofhyperbolicrationalmap}
	In this section, we review some notations and results in complex dynamics and apply Theorem~\ref{maintheorem'} to prove Theorem~\ref{t:hyperbolic_rational}.
	
	Recall that the \defn{spherical metric} $d_{\widehat{\C}}$ on the Riemann sphere $\widehat{\C}=\C\cup\{\infty\}$, given by the length element
	$\mathrm{d}s(z)\=2 \abs{ \mathrm{d}z } / (1+\abs{z}^2)$ for each $z\in\widehat{\C}$, is a conformal metric. Moreover, $\bigl(\widehat{\C},\,d_{\widehat{\C}},\,\Q^2\bigr)$ is a computable metric space and $\widehat{\C}$ is recursively compact, where $\Q^2\coloneqq\{a+bi\scolon a,\,b\in\Q\}$. The \defn{spherical derivative} $f^{\#}$ of a holomorphic function $f\:\widehat{\C}\rightarrow\widehat{\C}$ is given by
	\begin{equation}\label{def:sphericalf}
		f^{\#}(z) \= \lim\limits_{w\to z} d_{\widehat{\C}}(f(w),f(z)) \big/ d_{\widehat{\C}}(w,z)
		=\abs{f'(z)} \cdot  \bigl( 1+\abs{z}^2 \bigr) \big/ \bigl( 1+\abs{f(z)}^2\bigr)
	\end{equation}
	for each $z\in\widehat{\C}$. If $z$ or $f(z)$ is equal to $\infty$, then the last expression of (\ref{def:sphericalf}) has to be understood as a suitable limit.
	
	We identify $\widehat{\C}$ with the unit sphere in $\R^3$ via stereographic projection. The \defn{chordal metric} $\sigma$ on $\widehat{\C}$ is the metric that corresponds to the Euclidean metric in $\R^3$ under this identification. More precisely,
	\begin{equation}\label{defofsigmametric}
		\sigma(z,w)= 2\abs{z-w} \big/\sqrt{(1+\abs{z}^2)(1+\abs{w}^2)}\quad\text{ for each pair of }z,\,w\in\C,
	\end{equation}
	and $\sigma(\infty,z)=\sigma(z,\infty)= 2 \big/ \sqrt{1+\abs{z}^2}$ for each $z\in\C$.
	
	Any holomorphic map $f\colon\widehat{\C}\rightarrow\widehat{\C}$ on the Riemann sphere can be expressed as a rational function, that is, as the quotient $f(z)=p(z)/{q(z)}$ of two polynomials $p$ and $q$. Here we may assume that $p$ and $q$ have no common roots. The \defn{degree} $d$ of the rational map $f$ is defined as the maximum of the degrees of $p$ and $q$. The \defn{Fatou set} $\mathcal{F}_f$ of $f$ is defined to be the set of points $z\in\widehat{\C}$ satisfying that there exists an open neighborhood $U(z)$ of $z$ on which the family $\{f^n|_{U(z)}\}_{n\in\N}$ is equicontinuous with respect to the spherical metric $d_{\widehat{\C}}$. The \defn{Julia set} $\cJ_f$ of $f$ is the complement of the Fatou set $\mathcal{F}_f$.

	\begin{prop}[{\cite[Corollary~4.13]{Mil06}}]\label{prop:juliaset}
		If $f\:\widehat{\C}\rightarrow\widehat{\C}$ is a rational map of degree $d\geq 2$, then for each $z_0\in\cJ_f$, the set $
		\bigcup_{n\in\N}f^{-n}(z_0)
		$ of all iterated preimages of $z_0$ is dense in $\cJ_f$.
	\end{prop}
	
	The maps we consider in this section are hyperbolic rational maps, which are defined as follows (see e.g.~\cite[Theorem~19.1]{Mil06}).
	
	\begin{definition}\label{hyperbolicity}
		A rational function $f$ of degree $d\geq 2$ is \defn{hyperbolic} if the closure of the set of postcritical points of $f$ is disjoint from $\cJ_f$.
	\end{definition}
	
	Recall that a point $x \in \widehat{\C}$ is a \emph{postcritical point} of $f$ if $x= f^n(c)$ for some critical point $c \in \widehat{\C}$ of $f$ and some $n\in\N$.
	For hyperbolic rational maps, we have the following two results.
	
	\begin{lemma}[{\cite[Lemma~V.2.1]{CG93}}]\label{hyperbolicityimplyexpansion}
		Assume that $f\colon\widehat{\C}\rightarrow\widehat{\C}$ is a hyperbolic rational map of degree $d\geq 2$ with $\infty\notin\cJ_f$. Then there exists $a>0$ and $A>1$ such that $\abs{(f^n)'(z)}\geq aA^n$ for all $n\in\N$ and $z\in\cJ_f$.
	\end{lemma}
	
	\begin{lemma}[{\cite[Theorem~V.2.3]{CG93}}]\label{zeroarea}
		If $f$ is a hyperbolic rational map of degree $d\geq 2$, then $\cJ_f$ has zero area.
	\end{lemma}
	
	The following lemma is due independently to Braverman \cite{Br04} (in the case of the hyperbolic polynomials) and Rettinger \cite{Re05} (in the case of the hyperbolic rational maps).
	
	\begin{lemma}\label{Hyperbolicjuliaset}
		If $f$ is a computable hyperbolic rational map of degree $d\geq 2$, then the distance function $x\mapsto d_{\widehat{\C}}(\cJ_f,x)$ is computable.
	\end{lemma}
	
	\begin{prop}\label{maintheoremofapp}
		Let $f\:\widehat{\C}\rightarrow\widehat{\C}$ be a computable hyperbolic rational map of degree $d\geq 2$. Then for each repelling periodic point $w$ of $f$, the following statements are true:
		\begin{enumerate}
			\smallskip
			\item[(i)] 
			There is an enumeration $\{w_i\}_{i\in\N}$ of the set $\bigcup_{n\in\N}f^{-n}(w)$ such that $\{w_i\}_{i\in\N}$ is a uniformly computable sequence of points in the computable metric space $\bigl(\widehat{\C},\,d_{\widehat{\C}},\,\Q^2\bigr)$.
			\smallskip
			\item[(ii)] $\bigl(\cJ_f,\,d_{\widehat{\C}},\,\bigcup_{n\in\N}f^{-n}(w)\bigr)$ is a computable metric space, and the Julia set $\cJ_f$ is recursively compact in the computable metric space $\bigl(\cJ_f,\,d_{\widehat{\C}},\,\bigcup_{n\in\N}f^{-n}(w)\bigr)$.
		\end{enumerate}
	\end{prop}
	
	\begin{proof}
		(i) First, we demonstrate that $w$ is computable in the computable metric space $\bigl(\widehat{\C},\,d_{\widehat{\C}},\,\Q^2\bigr)$. Since $w$ is a periodic point of $f$, there exists $m\in\N$ with $f^m(w)=w$. Then we apply a standard root-finding algorithm (see e.g.~\cite[Appendix~A]{BBY12}) for the function $f^m(x)-x$ to compute $w$. Thus $w$ is computable and lies in $\cJ_f$ by \cite[Lemma~4.6]{Mil06}.
		Hence, by the same root-finding algorithm, we can compute all the roots of the function $g_n(x)=f^n(x)-w$ for each $n\in\N$. Hence, there is an enumeration $\{w_i\}_{i\in\N}$ of the set $\bigcup_{n\in\N}f^{-n}(w)$ such that $\{w_i\}_{i\in\N}$ is a uniformly computable sequence of points in the computable metric space $\bigl(\widehat{\C},\,d_{\widehat{\C}},\,\Q^2\bigr)$.
		
		\smallskip
		(ii) By Proposition~\ref{prop:juliaset}, $\bigcup_{n\in\N}f^{-n}(w)$ is dense in $\mathcal{J}_f$. Together with statement~(i), it is not hard to see that $\bigl(\cJ_f,\,d_{\widehat{\C}},\,\bigcup_{n\in\N}f^{-n}(w)\bigr)$ is a computable metric space. Hence, to demonstrate statement~(ii), by Proposition~4 in \cite{GHR11}, it suffices to prove that $\bigl(\cJ_f,\,d_{\widehat{\C}},\,\bigcup_{n\in\N}f^{-n}(w)\bigr)$ is recursively precompact.
		
		Now we fix an integer $s$ and construct a $2^{-s}$-net of the Julia set $\cJ_f$ as follows. Since $\bigl(\widehat{\C},\,d_{\widehat{\C}},\,\Q^2\bigr)$ is recursively precompact, we can compute a $2^{-s-1}$-net $N''_s$ of $\widehat{\C}$. By Lemma~\ref{Hyperbolicjuliaset}, we can compute a $2^{-s-1}$-net $N'_s\subseteq N''_s$ of $\cJ_f$ satisfying $\cJ_f\cap B \bigl(y,2^{-s-1} \bigr) \neq \emptyset$ for each $y\in N'_s$. By Proposition~\ref{prop:juliaset} and statement~(i), we can compute a point $t(y)\in B \bigl( y,2^{-s-1} \bigr) \cap \bigl(\bigcup_{n\in\N}f^{-n}(w)\bigr)$ for each $y\in N'_s$. Hence, $\{t(y):y\in N'_s\}\subseteq \cJ_f$ is a $2^{-s}$-net of the Julia set $\cJ_f$. This completes the proof.
	\end{proof}
	
	\begin{prop}\label{maintheoremofapp3}
		Let $f\:\widehat{\C}\rightarrow\widehat{\C}$ be a hyperbolic rational map of degree $d\geq 2$ with $\infty\notin\mathcal{J}_f$. Let $C_f\=f(\{z\in\widehat{\C}:f'(z)=0\})$ be the set of critical values of $f$. Assume that there exist positive constants $V_1$, $V_2$, $V_3$, and $V_4$ satisfying the following properties for each pair of $x,\,y\in\cJ_f$:
		\begin{enumerate}
			\smallskip
			\item[(1)] $B(x,V_1)\cap C_f=\emptyset$;
			\smallskip
			\item[(2)] $8/V_4<V_2\leq \abs{f'(x)}\leq V_3$;
			\smallskip
			\item[(3)] $d_{\widehat{\C}}(x,y)\geq V_4\abs{x-y}$.
		\end{enumerate}
		Put $l\=V_1/(4V_3)$. Then the following statements are true:
		\begin{enumerate}
			\smallskip
			\item[(i)] The map $f$ is injective on $B(x,l)$ for each $x\in\cJ_f$.
			\smallskip
			\item[(ii)] The restriction $f|_{\cJ_f}\colon\cJ_f\rightarrow\cJ_f$ is a distance-expanding map with respect to $d_{\widehat{\C}}$ with constants $\eta\=V_4l/2$, $\lambda\=V_2V_4/8$, and $\xi\=V_2l/16$.
		\end{enumerate}
	\end{prop}
	
	\begin{proof}
		(i) Consider an arbitrary point $x\in\cJ_f$. Then $f(x)\in\cJ_f$. By property~(1), we obtain that $B(f(x),V_1)\cap C_f=\emptyset$. Hence, we can define a conformal inverse $f^{-1}_x\: B(f(x),V_1)\rightarrow\widehat{\C}$ such that $f^{-1}_x(f(x))=x$ and $f \bigl( f^{-1}_x(y) \bigr)=y$ for each $y\in B(f(x),V_1)$. Then by Koebe's one-quarter theorem (see e.g.~\cite[Theorem~I.1.3]{CG93}), we have $B(x,\abs{(f_x^{-1})'(f(x))}V_1/4)\subseteq f^{-1}_x(B(f(x),V_1))$. By the chain rule and property~(2), $\abs{(f_x^{-1})'(f(x))}=1/\abs{f'(x)}\geq V_3^{-1}  $. Thus, one can see that $l = V_1 / (4V_3) \leq \abs{(f_x^{-1})'(f(x))}V_1 / 4$. Therefore, $f$ is injective on $B(x,l)$.
		
		\smallskip
		
		(ii) Fix an arbitrary pair of $x,\,y\in\cJ_f$ with $d_{\widehat{\C}}(x,y)\leq 2\eta$. First, we show that $d_{\widehat{\C}}(f(x),f(y))\geq \lambda d_{\widehat{\C}}(x,y)$. By property~(3), we obtain that $\abs{x-y}\leq 2\eta/V_4=l$. Hence, by statement~(i), $f$ is injective on $B(x,\abs{x-y})$. By Koebe's one-quarter theorem, we have
		\begin{equation*}
			B (f(x), \abs{x-y}\cdot\abs{f'(x)} / 4 )\subseteq f(B(x,\abs{x-y})).
		\end{equation*}
		By $f(y)\in \partial f(B(x,\abs{x-y}))$, we obtain that $\abs{f(x)-f(y)}\geq \abs{x-y}\cdot\abs{f'(x)} /4 \geq  V_2 \abs{x-y} /4$. Note that by \cite[Theorem~III.1.3]{CG93}, $f(\cJ_f)=\cJ_f$. Then $f(x)$ and $f(y)$ both still belong to $\cJ_f$. By property~(3), we have
		\begin{equation*}
			d_{\widehat{\C}}(f(x),f(y))
			\geq V_4\abs{f(x)-f(y)}
			\geq V_2V_4 \abs{x-y} / 4
			\geq V_2V_4 d_{\widehat{\C}}(x,y)  / 8
			=\lambda d_{\widehat{\C}}(x,y).	
		\end{equation*}
				
		Next, we verify that $B_{d_{\widehat{\C}}}(f(x),\xi)\cap\cJ_f\subseteq f\bigl(B_{d_{\widehat{\C}}}(x,\eta)\cap\cJ_f\bigr)$ for each $x\in\cJ_f$. Indeed, by $d_{\widehat{\C}}(x,y)\leq 2\abs{x-y}$ for each pair of $x,\,y\in\widehat{\C}$, we have \begin{equation}\label{eq:containing1}
			B_{d_{\widehat{\C}}}(f(x),\xi)\subseteq B(f(x),2\xi)\quad\text{ for each }x\in\widehat{\C}.
		\end{equation} 
		By the definitions of $l,\,\xi$, property~(2), statement~(i), and Koebe's one-quarter theorem, we have
		\begin{equation}\label{eq:containing2}
			B(f(x),2\xi)
			=B(f(x),V_2l/8)\subseteq B(f(x),\abs{f'(x)}l/8)
			\subseteq f(B(x,l/2))\quad\text{ for each }x\in\cJ_f.
		\end{equation}
		By property~(3), we have $B(x,l/2)\cap\cJ_f\subseteq B_{d_{\widehat{\C}}}(x,\eta)\cap\cJ_f$. Note that by \cite[Theorem~III.1.3]{CG93}, $f(\cJ_f)=\cJ_f$. Then $f(A\cap\cJ_f)=f(A)\cap\cJ_f$ for each subset $A\subseteq\widehat{\C}$. Thus, we obtain that 
		\begin{equation}\label{eq:containing3}
			f(B(x,l/2))\cap\cJ_f=	f(B(x,l/2)\cap\cJ_f)\subseteq f \bigl( B_{d_{\widehat{\C}}}(x,\eta)\cap\cJ_f \bigr) \quad\text{ for each }x\in\cJ_f.
		\end{equation}
		Therefore, by (\ref{eq:containing1}),~(\ref{eq:containing2}),~and~(\ref{eq:containing3}), we have $B_{d_{\widehat{\C}}}(f(x),\xi)\cap\cJ_f\subseteq f \bigl( B_{d_{\widehat{\C}}}(x,\eta)\cap\cJ_f \bigr)$ for each $x\in\cJ_f$.	 
		This completes the proof of statement~(ii).
	\end{proof}
	
	\begin{prop}\label{maintheoremofapp2}
		Let $f\:\widehat{\C}\rightarrow\widehat{\C}$ be a hyperbolic rational map of degree $d\geq 2$ with $\infty\notin\cJ_f$.
		Let $x_k$, $1\leq k\leq m$, be an enumeration of $f^{-1}(\infty)\cup\{\infty\}$, and $r > 0$ be a constant satisfying $\cJ_f\subseteq K \= \bigcap_{k=1}^mB_{d_{\widehat{\C}}}(x_k,r)^c$ and $B_{d_{\widehat{\C}}}(x_i,r)\cap B_{d_{\widehat{\C}}}(x_j,r)=\emptyset$ for all $1\leq i<j\leq m$. Assume that there exist positive constants $D_1,\,D_2,\,C_1,\,C_2$, and $C$ satisfying the following properties:
		\begin{enumerate}
			\smallskip
			\item[(1)] $\abs{z} \leq D_1$ and $\abs{ f(z) } \leq D_2$ for each $z\in K$;
			\smallskip
			\item[(2)] $\abs{ f'(z)}\leq C_1$ and $\abs{f''(z)}\leq C_2$ for each $z\in K$;
			\smallskip
			\item[(3)] $\Abs{f^{\#}(z)}\geq C$ for each $z\in\cJ_f$.
		\end{enumerate}
		Then the following statements are true:
		\begin{enumerate}
			\item[(i)] $d_K(x,y)\leq\pi d_{\mathbb{\widehat{C}}}(x,y)$ for each pair of $x,\,y\in K$, where $d_K$ is the metric on $K$ given by the spherical length element
	$\mathrm{d}s(z)=2 \abs{ \mathrm{d}z } \big/ \bigl( 1+\abs{z}^2 \bigr)$ for each $z\in K$, i.e.,
			\begin{equation*}
				d_K(x,y) \= \inf\biggl\{\int_{\gamma} \mathrm{d}s\scolon \gamma(0)=x,\,\gamma(1)=y, \text{ and } \gamma\: [0,1]\rightarrow K \text{ is continuous}\biggr\}.
			\end{equation*}
			\item[(ii)] If $g\:\widehat{\C}\rightarrow\widehat{\C}$ is a rational map, then the restriction $\abs{g}|_{K}\colon K\rightarrow\R$ of the norm of $g$ is $\bigl(A\bigl(1+D_1^2\bigr)\pi/2\bigr)$-Lipschitz continuous with respect to $d_{\widehat{\C}}$, where $A\=\sup_{z\in K}\abs{g'(z)}$.
			\smallskip
			\item[(iii)] The restriction $\varphi_f|_{\cJ_f}\colon\cJ_f\rightarrow\R$ is $\bigl( BC^{-1} \bigr)$-Lipschitz continuous with respect to $d_{\widehat{\C}}$, where $\varphi_f(z) \= \log\bigl(f^{\#}(z)\bigr)$ for each $z\in\cJ_f$ and $B\=\bigl(1+D_1^2\bigr)\bigl(C_2\bigl(1+D_1^2\bigr)+D_2C_1^2\bigl(1+D_1^2\bigr)+C_1D_1\bigr)\pi$.
		\end{enumerate}
	\end{prop}
	
	\begin{proof}
		(i) The path-connectedness of $K$ is immediate from its definition. Hence, $d_K$ is well-defined. Consider an arbitrary pair of $x,\,y\in K$. By subdividing the spherical geodesic from $x$ to $y$, we only need to consider the case  where $x,\,y\in\partial B_{d_{\widehat{\C}}}(x_i,r)$ for some integer $1\leq i\leq m$. Then we have $d_K(x,y)\leq\pi\sigma(x,y)\leq\pi d_{\widehat{\C}}(x,y).$
		
		\smallskip
		
		(ii) Fix an arbitrary pair of $x,\,y\in K$. By the definition of $K$, $K$ is closed. Then there exists a continuous function $\gamma\:[0,1]\rightarrow K$ with $\gamma(0)=x,\,\gamma(1)=y,$ and $d_K(x,y)=\int_{\gamma} \mathrm{d}s$. Hence, we have
		\begin{equation*}
			\abs{\abs{g(x)}-\abs{g(y)}}\leq\Absbigg{\int_{\gamma}\!g'(z)\,\mathrm{d}z }
			\leq\int_{\gamma}\!\frac{A(1+\abs{z}^2)}{2}\,\mathrm{d}s(z)
			\leq\frac{A\bigl(1+D_1^2\bigr)d_K(x,y)}{2}
			\leq \frac{A\bigl(1+D_1^2\bigr)\pi d_{\widehat{\C}}(x,y)}{2}.
		\end{equation*}
		This completes the proof of statement~(ii).
		
		\smallskip
		
		(iii) Define $u(x)\=1+\abs{ x }^2$, $v(x)\=\frac{1}{1+\abs{f(x)}^2}$, and $w(x)\=\abs{f'(x)}$ for each $x\in K$. 
		
		By statement~(ii), conditions~(1) and~(2), we have
		\begin{equation}\label{e:lipu}
			\abs{u(x)-u(y)}=(\abs{x}+\abs{y})\Abs{\abs{x}-\abs{y}}\leq 2D_1\Abs{\abs{x}-\abs{y}}\leq D_1(1+D_1^2)\pi d_{\widehat{\C}}(x,y),
		\end{equation}
		\begin{equation}\label{e:lipv}
			\abs{v(x)-v(y)}=\frac{\abs{\abs{f(x)}^2-\abs{f(y)}^2}}{(1+\abs{f(x)}^2)(1+\abs{f(y)}^2)}\leq 2D_2\abs{\abs{f(x)}-\abs{f(y)}}\leq D_2C_1\bigl(1+D_1^2\bigr)\pi d_{\widehat{\C}}(x,y),
		\end{equation}
		and
		\begin{equation}\label{e:lipw}
			\abs{w(x)-w(y)}\leq C_2\bigl(1+D_1^2\bigr)\pi d_{\widehat{\C}}(x,y)
		\end{equation}
		for each pair of $x,\,y\in K$.
		By the definition of $K$, (\ref{def:sphericalf}) holds for each $z\in K$.
		Moreover, by (\ref{e:lipu}), (\ref{e:lipv}), and (\ref{e:lipw}), we have
		\begin{equation}\label{eq:lipschitzcontinuityoffderivative}
			\begin{aligned}
				 \Absbig{f^{\#}(x)-f^{\#}(y)} 
				&\leq \abs{u(x)v(x)} \abs{w(x)-w(y)}+ \abs{ u(x)w(y) } \abs{ v(x)-v(y) } + \abs{ v(y)w(y) } \abs{ u(x)-u(y) }\\
				&\leq \bigl(C_2\bigl(1+D_1^2\bigr)^2+D_2C_1^2\bigl(1+D_1^2\bigr)^2+C_1D_1\bigl(1+D_1^2\bigr)\bigr)\pi  d_{\widehat{\C}}(x,y)
				=Bd_{\widehat{\C}}(x,y)
			\end{aligned}
		\end{equation}
		for each pair of $x,\,y\in K$.
		Note that $\log x \leq x-1$ for each $x\geq1$. Then it follows from condition~(3) and (\ref{eq:lipschitzcontinuityoffderivative}) that
		\begin{equation*}
			\begin{aligned}
				\abs{ \varphi_f(x)-\varphi_f(y) }=\Absbigg{\log\biggl(\frac{\abs{f^{\#}(x)}}{\abs{f^{\#}(y)}}\biggr) }
				\leq  \frac{\abs{ f^{\#}(x)-f^{\#}(y) }}{\min\{\abs{ f^{\#}(x) },\,\abs{ f^{\#}(y) }\}}\leq BC^{-1}d_{\widehat{\C}}(x,y)
			\end{aligned}
		\end{equation*}
		for each pair of $x,\,y\in\cJ_f$.
	\end{proof}
	
	Now, we are ready to establish Theorem~\ref{t:hyperbolic_rational}.
	\begin{proof}[Proof of Theorem~\ref{t:hyperbolic_rational}]
		 Since $f$ is a hyperbolic rational map, it follows from Lemma~\ref{zeroarea} and the compactness of $\cJ_f$ that there exists a computable point $u$ in $\cJ_f^c\cap\Q^2$. By Lemma~\ref{Hyperbolicjuliaset}, there exists an algorithm which on input $u\in\Q^2$, halts if and only if $d_{\widehat{\C}}(u,\cJ_f)>0$. With this algorithm, we can compute a point $u\in\cJ_f^c$. Define, for each $z\in\widehat{\C}$,
		\begin{equation*}
			U(z)\=\frac{\ou z+1}{u-z}\quad\text{ and }\quad g(z)\=\bigl(U\circ f\circ U^{-1}\bigr)(z).
		\end{equation*}
		Then $U$ is a M\"obius transformation. Since the function $f$ and the point $u$ are both computable, by the definition of $g$, $g$ is also a computable hyperbolic rational map with $\cJ_g=U(\cJ_f)$. Since $u\in\cJ_f^c$, we have $\infty\in\cJ_g^c$.
		
		By \cite[Theorem~III.1.2]{CG93}, $\cJ_g\neq\emptyset$. Then by \cite[Theorem~III.3.1]{CG93}, there exists a repelling periodic point of $g$, say $w$. By Proposition~\ref{maintheoremofapp}, $\bigl(\mathcal{J}_g,\,d_{\widehat{\mathbb{C}}},\,\bigcup_{n\in\N}g^{-n}(w)\bigr)$ is a computable metric space and $\mathcal{J}_g$ is recursively compact in the computable metric space $\bigl(\mathcal{J}_g,\,d_{\widehat{\mathbb{C}}},\,\bigcup_{n\in\N}g^{-n}(w)\bigr)$. 
		
		Now we apply Proposition~\ref{maintheoremofapp3} to compute constants $n\in\N,\,\eta>0,\,\lambda>1$, and $\xi>0$ satisfying that the restriction $g^n|_{\cJ_g}$ is distance-expanding with respect to the spherical metric with constants $\eta,\,\lambda$, and $\xi$ above. 
		
		Put $V_4\=\frac{2}{1+r_0^2}$. By \cite[Proposition~2.13]{BRY12}, we can compute a constant $r_0>0$ with $\abs{z}\leq r_0$ for each $z\in\cJ_g$. Then by (\ref{defofsigmametric}), one can conclude that
		\begin{equation*}
			d_{\widehat{\C}}(x,y)\geq\sigma(x,y)
			\geq 2\abs{x-y} \big/ \bigl( 1+r_0^2 \bigr)=V_4\abs{x-y}\quad\text{ for each pair of }x,\,y\in\cJ_g.
		\end{equation*}
		
		By Lemma~\ref{hyperbolicityimplyexpansion}, there exists sufficiently large $n\in\N$ with $\abs{(g^n)'(z)}>4(1+r_0)^2$ for each $z\in\cJ_g$. Note that by \cite[Proposition~2.13]{BRY12}, there exists an algorithm which on input $n\in\N$, computes $\inf_{z\in\cJ_g}\abs{(g^n)'(z)}$. Hence, we can apply this algorithm to compute three constants $n\in\N,\,V_2>0$, and $V_3>0$ satisfying 
        \begin{equation*}
			V_3\geq\abs{(g^n)'(z)}\geq V_2>4(1+r_0)^2\quad\text{ for each }z\in\cJ_g.
		\end{equation*}
		
		Since $g$ is computable, we can apply a standard root-finding algorithm (see e.g.~\cite[Appendix~A]{BBY12}) to compute all the points in $\bigl\{z\in\widehat{\C}\scolon (g^n)^{\prime}(z)=0 \bigr\}$. Hence, we can compute all the points in the set $C_{g^n}$ of the critical values of $g^n$. By Lemma~\ref{Hyperbolicjuliaset}, we can compute a constant $V_1>0$ satisfying that $d(z,\cJ_g)\geq V_1$ for each $z\in C_{g^n}$. Note that, by \cite[Theorem~III.1.4]{CG93}, $\cJ_{g^n}=\cJ_g$. Then it is not hard to see that the above constants $V_1,\, V_2,\, V_3,\, V_4>0$ satisfy Properties~(1),~(2),~and~(3) in Proposition~\ref{maintheoremofapp3} for the hyperbolic rational map $g^n$. Hence, by Proposition~\ref{maintheoremofapp3}, we can compute the constants $\eta$, $\lambda,$ and $\xi$ satisfying that the restriction $g^n|_{\cJ_g}$ is distance-expanding with respect to the spherical metric with constants $\eta,\,\lambda$, and $\xi$.
		
		Next, we establish that $\varphi_{g^n}$ is $L$-Lipschitz continuous for some $L>0$ and compute one of such constant $L$. Since $g^n$ is a computable rational map, we can compute $(g^n)^{-1}(\infty)\cup\{\infty\}$, and we enumerate the points as $x_k,\,1\leq k\leq m$. Then by Lemma~\ref{Hyperbolicjuliaset}, we can compute $r>0$ satisfying $\cJ_g\subseteq K \= \bigcap_{k=1}^mB_{d_{\widehat{\C}}}(x_k,r)^c$ and $B_{d_{\widehat{\C}}}(x_i,r)\cap B_{d_{\widehat{\C}}}(x_j,r)=\emptyset$ for all $1\leq i<j\leq m$. By Proposition~\ref{prop:thecomplementofrecursivelycompactset}~(iii), $K$ is recursively compact. Hence, by Proposition~\ref{prop:thecomplementofrecursivelycompactset}~(iv), we can compute $D_1$, $D_2$, $C_1$, $C_2,\,C>0$ satisfying conditions~(1),~(2),~and~(3) of Proposition~\ref{maintheoremofapp2} for the map $g^n$. Hence, by Proposition~\ref{maintheoremofapp2}, we can compute $L>0$ satisfying that $\varphi_{g^n}$ is $L$-Lipschitz continuous with respect to the spherical metric.
		
		We are now ready to return to the proof of the main statement. By \cite[Theorem~III.3.2]{CG93}, $g^n|_{\mathcal{J}_g}$ is topologically exact. Moreover, by Propositions~\ref{maintheoremofapp},~\ref{maintheoremofapp3}, and~\ref{maintheoremofapp2}, for each repelling periodic point $w$ of $g$, the septet $\bigl( \widehat{\C},\,\cJ_g,\,d_{\widehat{\C}},\,\Q^2,\,\bigcup_{n\in\N}g^{-n}(w),\,t\varphi_{g^n},\,g^n \bigr)$ satisfies the Assumption~B in Section~\ref{sec:intro}. Moreover, by the above analysis, we can compute the five constants in Theorem~\ref{maintheorem'}~(iii) and~(iv) for the map $g^n$ and the function $t\varphi_{g^n}$. Proposition~\ref{maintheoremofapp}~(ii) gives the algorithm in Theorem~\ref{maintheorem'}~(i). Hence, by Theorem~\ref{maintheorem'}, the unique equilibrium state for $g^n$ and $t\varphi_{g^n}$ is computable.
		
		By \cite[Theorem~2.4.6~(a) and Lemma~3.2.7]{PU10}, the unique equilibrium state for $g$ and $t\varphi_g$ is exactly the unique equilibrium state for $g^n$ and $tS_n\varphi_g$. Moreover, by (\ref{def:sphericalf}), $S_n\varphi_g=\varphi_{g^n}$, because $\infty\notin\cJ_g$. Hence, the unique equilibrium state for $g$ and $t\varphi_{g}$ is computable. By simple computation, for each $z\in\widehat{\C}$,
		\begin{equation*}
			\mathrm{d}s(U(z))=\frac{2\abs{\mathrm{d}(U(z))}}{1+\abs{U(z)}^2}=\frac{2\abs{\mathrm{d}z}\abs{U^{\prime}(z)}}{1+\abs{U(z)}^2}=\frac{2\abs{\mathrm{d}z}(1+\abs{u}^2)}{\abs{u-z}^2+\abs{\ou z+1}^2}=\frac{2\abs{\mathrm{d}z}}{1+\abs{z}^2}=\mathrm{d}s(z).
		\end{equation*} 
		Hence, the map $U$ preserves the spherical metric. Therefore, the unique equilibrium state for $f$ and $t\varphi_f$ is also computable.
	\end{proof}

	\section{Non-uniqueness of equilibrium states}   \label{s:nonunique}
	
	In this section, we establish Theorem~\ref{t:nonunique} and then use it to prove Theorem~\ref{t:nonunique_spec}. 
	
	\begin{lemma}[{\cite[Proposition~7]{GHR11}}]\label{supportofmeasure}
		Let $(X,\,\rho,\,\mathcal{S})$ be a computable metric space. Assume that $X$ is recursively compact, and $\mu\in \PPP(X)$ is computable. Then there exists at least one computable point in the support of $\mu$.
	\end{lemma}
	
	\begin{lemma}\label{lemma:nonwanderingset}
		Let $(X,\rho)$ be a compact metric space, and $T\:X\rightarrow X$ be a continuous map. Then for each $\mu\in \MMM(X,T)$,  $\supp(\mu)$ is contained in the non-wandering set $\Omega(T)$ of $T$.
	\end{lemma}
	
	\begin{proof}
		Consider $\mu\in \MMM(X,T)$. By \cite[Theorem~6.15~(i)]{Wa82}, we have $\mu(\Omega(T))=1$. Note that by \cite[Theorem~5.6~(i)]{Wa82}, $\Omega(T)$ is closed. Then we obtain that $\supp(\mu)\subseteq\Omega(T)$.
	\end{proof}
	
	\begin{lemma}\label{lemma:computabilityofpressureandequilibriumstate}
			Let $(X,\,\rho,\,\mathcal{S})$ be a computable metric space and $X$ be a recursively compact set. Assume that $T\:X\rightarrow X$ is a computable map satisfying that the measure-theoretic entropy function $\mathcal{H}\:\PPP(X)\rightarrow\R$ given by
			\begin{equation*}
				\mathcal{H}(\mu)\=
				\begin{cases}
					h_{\mu}(T) & \text{if }\mu\in\MMM(X,T); \\
					0 & \text{if }\mu\notin\MMM(X,T)
				\end{cases}
			\end{equation*}
			is upper-computable on $\MMM(X,T)$, and that $\varphi\:X\rightarrow\R$ is a computable function satisfying that the topological pressure $P(T,\varphi)$ is lower-computable. Then the set $\cE(T,\varphi)$ of equilibrium states for $T$ and $\varphi$ is recursively compact.
	\end{lemma}
	
	\begin{proof}
		By Corollary~\ref{computabilityofintegraloperator}, the integral function $\Phi\:\PPP(X)\rightarrow\R$ given by $\Phi(\mu)\=\int\!\varphi\,\mathrm{d}\mu$, $\mu \in \PPP(X)$, is computable. Since $\mathcal{H}(\mu)$ is upper-computable on $\MMM(X,T)$, the measure-theoretic pressure function $\cP(\mu)\coloneqq\mathcal{H}(\mu)+\Phi(\mu)$ is upper-computable on $\MMM(X,T)$.
		
		Since $P(T,\varphi)$ is lower-computable, there exists a uniformly computable sequence of rational points $\{r_n\}_{n\in\N}$ which is increasing and converges to $P(T,\varphi)$. By Definition~\ref{defn:domainofcomputability}, there exists a uniformly computable sequence $\{U_n\}_{n\in\N}$ of lower-computable open sets with $\cP^{-1}(-\infty,r_n)\cap\MMM(X,T)=U_n\cap\MMM(X,T)$ for each $n\in\N$. Hence, by Definition~\ref{def:equilibriumstate}, we have
		\begin{equation}\label{e:setofequilibrium}
			\begin{aligned}
				\cE(T,\varphi)
                &=\MMM(X,T)\smallsetminus \cP^{-1}(-\infty,P(T,\varphi)) \\
				&=\MMM(X,T) \smallsetminus \Bigl( \bigcup_{n\in\N} \bigl( \cP^{-1}(-\infty,r_n)\cap\MMM(X,T) \bigr) \Bigr)
				=\MMM(X,T)\cap\Bigl(\bigcup_{n\in\N}U_n\Bigr)^c.
			\end{aligned}
		\end{equation}
		Moreover, by Propositions~\ref{lower-computable open} and~\ref{prop:thecomplementofrecursivelycompactset}~(iii), it follows from (\ref{e:setofequilibrium}) that the set $\cE(T,\varphi)$ of all equilibrium states for $T$ and $\varphi$ is recursively compact.
	\end{proof}
	
	Now, we are ready to prove Theorem~\ref{t:nonunique}.
	
	\begin{proof}[Proof of Theorem~\ref{t:nonunique}]
		First, we show that there exists at least one equilibrium state for $T$ and $\varphi$. Indeed, by Definition~\ref{defn:domainofcomputability}, every upper-computable function is upper semi-continuous. Since $\cH\colon\cP(X)\rightarrow\R$ is upper-computable on $\cM(X,T)$, $\cH\colon\cP(X)\rightarrow\R$ is upper semi-continuous on $\mu\in\MMM(X,T)$. Hence, by \cite[Proposition~3.5.3]{PU10}, there exists at least one equilibrium state for $T$ and $\varphi$.
		
		Now we prove Theorem~\ref{t:nonunique} by contradiction and suppose that $\mu$ is the unique equilibrium state for $T$ and $\varphi$. By Lemma~\ref{lemma:computabilityofpressureandequilibriumstate}, the singleton $\{\mu\}$ is a recursively compact set. Then by  Proposition~\ref{prop:thecomplementofrecursivelycompactset}~(i), $\mu$ is computable, which would contradict the existence of a non-computable equilibrium state. Therefore, there are at least two equilibrium states for $T$ and $\varphi$. This completes the proof of Theorem~\ref{t:nonunique}.
	\end{proof}
	
	As an immediate consequence, we establish Theorem~\ref{t:nonunique_spec}.
	
	\begin{proof}[Proof of Theorem~\ref{t:nonunique_spec}]
		Assume that $T\:X\rightarrow X$ is a computable map with zero topological entropy satisfying that there exists no computable point in $\Omega(T)$. Then by \cite[Proposition~3.5.3]{PU10}, there exists at least one equilibrium state for $T$ and $\varphi$, say $\mu$. Then by Lemma~\ref{lemma:nonwanderingset}, there exists no computable point in $\supp(\mu)$. It follows from Lemma~\ref{supportofmeasure} that $\mu$ is non-computable. Therefore, by Theorem~\ref{t:nonunique}, there exists at least two equilibrium states for $T$ and $\varphi$.
	\end{proof}

	\section{Counterexamples}\label{sec:counterExample}
	
	\subsection{Proof of Theorem~\ref{t:maincounterexample}}\label{subsec:constructionofT}
	
	In this subsection, we demonstrate that the computable homeomorphism $T\:\mathbb{S}^1\rightarrow\mathbb{S}^1$ constructed in \cite[Subsection~4.1]{GHR11} satisfies the statements~(i)~and~(ii) in Theorem~\ref{t:maincounterexample} for each function $\varphi\in C \bigl( \mathbb{S}^1 \bigr)$.
	
	\begin{proof}[Proof of Theorem~\ref{t:maincounterexample}]
		Recall that the homeomorphism $T$, which is constructed in \cite[Subsection~4.1]{GHR11}, is a computable homeomorphism of the unit circle $\mathbb{S}^1$ whose non-wandering set $\Omega(T)$ contains no computable points.
		
		Since $T$ is a homeomorphism of the circle, by \cite[Theorem~7.14]{Wa82}, we have $h_{\top}(T)=0$. Then, by \cite[Proposition~3.5.3]{PU10}, there exists at least one equilibrium state for $T$ and $\varphi$. Hence, $T$ satisfies statement~(i) in Theorem~\ref{t:maincounterexample}. 
		
		Suppose that $\mu$ is an equilibrium state for $T$ and $\varphi$. Then by Lemma~\ref{lemma:nonwanderingset}, there exists no computable point in $\supp(\mu)$. It follows from Lemma~\ref{supportofmeasure} that $\mu$ is non-computable. Therefore, $T$ satisfies statement~(ii) in Theorem~\ref{t:maincounterexample}.
	\end{proof}
	
	\subsection{Asymptotic \textit{h}-expansive}\label{asymptotichexpansiveness}
	
	In this subsection, we introduce some additional relevant notions and results from ergodic theory. We use them to construct computable systems with arbitrarily high topological entropy whose equilibrium states are all non-computable. We begin by recalling the following definition from \cite[Remark~6.1.7]{Dow11}.
	\begin{definition}\label{defRefSeqOpenCover}
		A sequence of open covers $\{\xi_i\}_{i\in\N_0}$ of a compact metric space $X$ is a \defn{refining sequence of open covers} of $X$ if the following conditions are satisfied:
		\begin{enumerate}
			\smallskip
			\item[(i)] $\xi_{i+1}$ is a refinement of $\xi_i$ for each $i\in\N_0$.
			
			\smallskip
			\item[(ii)] For each open cover $\eta$ of $X$, there exists $j\in\N$ such that $\xi_i$ is a refinement of $\eta$ for each integer $i\geq j$.
		\end{enumerate}
	\end{definition}
	
	By Lebesgue's covering lemma, it is clear that for each compact metric space, refining sequences of open covers always exist.
	
	The topological tail entropy was first introduced by Misiurewicz under the name ``topological conditional entropy'' \cite{Mis73, Mis76}. In our paper, we use the terminology in \cite{Dow11} (see \cite[Remark~6.3.18]{Dow11} for reference).
	
	\begin{definition}\label{defTopTailEntropy}
		Let $(X,d)$ be a compact metric space, and $g\:X\rightarrow X$ be a continuous map. The \defn{topological conditional entropy} $h(g|\lambda)$ of $g$ for some open cover $\lambda$, is defined by
		\begin{equation}\label{eqDefTopCondEntropy}
			h(g | \lambda) \=  \lim\limits_{l\to+\infty} \lim\limits_{n\to+\infty} \frac{1}{n} H\bigl((\xi_l)_g^n\big|\lambda_g^n\bigr),
		\end{equation}
		where $\{\xi_l\}_{l\in\N_0}$ is an arbitrary refining sequence of open covers. For each pair of open covers $\xi$ and $\eta$,
		\begin{equation} \label{eqH(xi|eta)}
			H(\xi | \eta)  \=  \log \Bigl(\max_{A\in\eta} \bigl\{\min\bigl\{\card \xi_A\scolon\xi_A  \subseteq \xi,\, A \subseteq \bigcup \xi_A \bigr\}\bigr\} \Bigr)
		\end{equation}
		is the logarithm of the minimal number of sets from $\xi$ sufficient to cover any set in $\eta$.
		
		The \defn{topological tail entropy} $h^*(g)$ of $g$ is defined by
		\begin{equation*}
			h^*(g) \= \lim\limits_{m\to+\infty}  \lim\limits_{l\to+\infty} \lim\limits_{n\to+\infty}  \frac{1}{n}  H\bigl((\xi_l)_g^n\big|(\eta_m)_g^n\bigr) ,
		\end{equation*}
		where $\{\xi_l\}_{l\in\N_0}$ and $\{\eta_m\}_{m\in\N_0}$ are two arbitrary refining sequences of open covers, and $H$ is defined in (\ref{eqH(xi|eta)}).
	\end{definition}
	
	The concept of $h$-expansiveness was introduced by Bowen in \cite{Bo72}. We use the formulation in \cite{Mis76} (see also \cite{Dow11}).
	
	\begin{definition}[$h$-expansive]\label{defHExp}
		A continuous map $g\: X\rightarrow X$ on a compact metric space $(X,d)$ is $h$-expansive if there exists a finite open cover $\lambda$ of $X$ such that $h(g|\lambda)=0$.
	\end{definition}
	
	\begin{rem}\label{rmk:h-expansive}
		Bowen gives an equivalent definition of $h$-expansiveness in \cite{Bo72} as follows: the map $g\:X\rightarrow X$ is called \defn{$h$-expansive} if there exists $\epsilon>0$ with
		$h_{\top}\bigl(g|_{{\Phi}_{\epsilon,g}(x)}\bigr)=0$ for each $x\in X$, where 
		\begin{equation}  \label{e:Phi_eps_g}
			\Phi_{\epsilon,g}(x)\=\{y\in X\scolon d(g^n(x),g^n(y))\leq\epsilon\text{ for each } n\in\N_0\}.
		\end{equation}
	\end{rem}
	
	A weaker property was introduced by Misiurewicz in \cite{Mis73} (see also \cite{Mis76, Dow11}).
	
	\begin{definition}[Asymptotic $h$-expansive] \label{defAsympHExp}
		We say that a continuous map $g\: X\rightarrow X$ on a compact metric space $X$ is \defn{asymptotically $h$-expansive} if $h^*(g)=0$.
	\end{definition}
	
	\begin{rem}\label{remark:entropyiszerothenexpanding}
		The topological entropy of $g$ is $h_{\top}(g)=h(g|\{X\})$, where $\{X\}$ is the open cover of $X$ consisting of only one open set $X$.
		It also follows from Definition~\ref{eqDefTopCondEntropy} that, for each pair of open covers $\iota$ and $\eta$ of $X$, we have that $h(g|\iota)\leq h(g|\eta)$ if $\iota$ is a refinement of $\eta$. Hence, we have $h^*(g)\leq h_{\top}(g)$.
	\end{rem}
	
	We shall say that a pair $(X,f)$ is a \defn{cascade} if $f\:X\rightarrow X$ is a continuous map on the non-empty compact Hausdorff space $X$. Under the above notations, we have the following two results.
	
	\begin{prop}[{\cite[Theorem~3.2]{Mis76}}]\label{prop:additiveoftailentropy}
		Let $(X_1,f_1)$ and $(X_2,f_2)$ be two cascades. Then
		\begin{equation*}
			h^*(f_1\times f_2)=h^*(f_1)+h^*(f_2),
		\end{equation*}
		where $f_1\times f_2 \: X_1 \times X_2 \rightarrow X_1 \times X_2$ is given by
		\begin{equation}  \label{e:product_map}
			(f_1\times f_2)(x_1,x_2) \= (f_1(x_2), f_2(x_2))  \qquad \text{ for all }x_1 \in X_1 \text{ and }x_2 \in X_2.
		\end{equation}
	\end{prop}
	
	\begin{prop}[{\cite[Corollary 8.1]{Mis76}}]\label{prop:asymimplyexist}
		Let $(X,f)$ be a cascade. If $f$ is asymptotically $h$-expansive, then for each $\varphi\in C(X)$, there exists at least one equilibrium state for $f$ and $\varphi$.
	\end{prop}
	
	Now we are ready to prove Proposition~\ref{prop:mainresult2}, a precise version of Theorem~\ref{t:maincounterexample2}. Here we identify $\mathbb{S}$ with $\R/ [0,1]$.
	
	\begin{prop}\label{prop:mainresult2}
		Assume that $d$ is an integer, and $T_d \: \mathbb{S}^1 \rightarrow \mathbb{S}^1$ is given by $T_d(x) \=  dx\pmod{1}$ for $x\in \mathbb{S}^1$. Define $\hT\=T\times T_d \: \mathbb{S}^1\times\mathbb{S}^1 \rightarrow \mathbb{S}^1\times\mathbb{S}^1$. Then statements~(i) and~(ii) in Theorem~\ref{t:maincounterexample2} hold for the map $\hT$ and each $\hvarphi\in C \bigl( \mathbb{S}^1\times\mathbb{S}^1 \bigr)$, and $h_{\top}\bigl(\hT\bigr)=\log d$.
	\end{prop}
	
	\begin{proof}
		First, we demonstrate that $\hT$ satisfies statement~(i) in Theorem~\ref{t:maincounterexample2}. It is not hard to see that $T_d$ is a distance-expanding map. Then for sufficiently small $\epsilon>0$, the set $\Phi_{\epsilon,T_d}(x)=\{x\}$ for each $x\in \mathbb{S}^1$ (see (\ref{e:Phi_eps_g})). Then by Remark~\ref{rmk:h-expansive}, $T_d$ is $h$-expansive, and thus $h^*(T_d)=0$. Noting that $h_{\top}(T)=0$, it follows from Remark~\ref{remark:entropyiszerothenexpanding} that $h^*(T)\leq h_{\top}(T)=0$, hence, is equal to $0$. Hence by Proposition~\ref{prop:additiveoftailentropy}, $h^*\bigl(\hT\bigr)=h^*(T)+h^*(T_d)=0,$ which implies that $\hT$ is asymptotically $h$-expansive. By Proposition~\ref{prop:asymimplyexist}, for each continuous potential $\hvarphi\:\mathbb{S}^1\times\mathbb{S}^1\rightarrow\R$, there exists at least one equilibrium state for $\hT$ and $\hvarphi$.
		
		Finally, we argue that there exists no computable $\hT$-invariant Borel probability measure by contradiction and assume that $\widehat{\mu}\in \MMM \bigl( \mathbb{S}^1\times\mathbb{S}^1, \hT \bigr)$ is computable. Then $\mu$ given by $\mu(A) \= \widehat{\mu} \bigl( A\times\mathbb{S}^1 \bigr)$ for each Borel set $A\subseteq\mathbb{S}^1$ is a computable $T$-invariant probability measure. This is impossible by Theorem~\ref{t:maincounterexample}~(ii). Thus $\hT$ satisfies statement~(ii) in Theorem~\ref{t:maincounterexample2}.
	\end{proof}

\end{document}